\documentclass[11pt, a4paper]{amsbook}

\usepackage{amsmath}
\usepackage{amssymb}
\usepackage{amsthm}
\usepackage[mathscr]{euscript}
\usepackage{tikz-cd}
\usepackage{hyperref}

\providecommand{\A}{\mathbf A}

\providecommand{\Z}{\mathbf Z}

\renewcommand{\P}{\mathbf P}
\renewcommand{\O}{\mathscr O}
\providecommand{\Q}{\mathbf Q}
\providecommand{\PP}{\mathscr P} 

\providecommand{\displayroot}[2]{\sqrt[\leftroot{-1}\uproot{2}{#1}]{#2}}
\providecommand{\pthroot}[1]{\displayroot p {#1}}

\renewcommand{\phi}{\varphi}
\renewcommand{\epsilon}{\varepsilon}
\renewcommand{\emptyset}{\varnothing}

\providecommand{\sheafHom}{\mathscr H\! \mathit{om}}

\providecommand{\Frac}{\operatorname{Frac}}
\providecommand{\Sym}{\mathrm{S}}
\providecommand{\T}{\mathrm{T}}
\providecommand{\N}{\mathrm{N}}

\providecommand{\jac}{\operatorname{jac}}

\providecommand{\coker}{\operatorname{coker}}
\providecommand{\im}{\operatorname{im}}
\providecommand{\Spec}{\operatorname{Spec}}

\providecommand{\Hom}{\operatorname{Hom}}

\newtheorem{theorem}{Theorem}
\numberwithin{theorem}{chapter} 
\newtheorem{proposition}[theorem]{Proposition}
\newtheorem{lemma}[theorem]{Lemma}
\newtheorem{corollary}[theorem]{Corollary}

\theoremstyle{definition}
\newtheorem{definition}[theorem]{Definition}
\newtheorem{example}[theorem]{Example}
\newtheorem{remark}[theorem]{Remark}

\newtheorem{construction}[theorem]{Construction}

\numberwithin{equation}{chapter}

\title{Irrational Complete Intersections}
\author{Lucas Braune}
\date{\today}

\begin{document}

\maketitle

\begin{abstract}
We prove that a complete intersection of $c$ very general hypersurfaces of degrees $d_1,\dots,d_c\ge 2$ in $N$-dimensional complex projective space is not ruled (and therefore not rational) provided that
$\sum_{i=1}^c d_i \ge \tfrac 2 3 N + c + 1$.
To this end we consider a degeneration to positive characteristic, following Koll\'ar.
Our argument does not require a resolution of the singularities of the special fiber of the degeneration.
It relies on a generalization of Koll\'ar's ``algebraic Morse lemma'' that controls the dimensions of the second-order Thom-Boardman singularities of general sections of Frobenius pullbacks of vector bundles.
\end{abstract}

\setcounter{tocdepth}{1}
\tableofcontents

\chapter*{Introduction}

Let $k$ be an algebraically closed field.
An algebraic variety over $k$ is said to be \emph{rational} if it contains a (Zariski-)dense open subset that is isomorphic to a dense open subset of $n$-dimensional projective space $\P^n_k$ for some $n$.
Varieties that are not rational are referred to as \emph{irrational}.
A projective variety of codimension $c$ in $\P^N_k$ is called a \emph{complete intersection} if it is the scheme-theoretic intersection of $c$ hypersurfaces (equivalently, if its ideal is generated by $c$ homogeneous polynomials).

In this thesis, we consider the following question: Over the field $\mathbf C$ of complex numbers, which smooth complete intersections are rational?

Positive results asserting the rationality of smooth complex complete intersections are few and scattered in the literature.
For example, quadric hypersurfaces are rational, as are even-dimensional cubic hypersurfaces in $\P^{2m+1}_\mathbf C$ containing a pair of skew $m$-planes (a class that includes all cubic surfaces in $\P^3_\mathbf C$) and complete intersections of two quadrics containing a line.
On the other hand, it seems that there are no known  examples of smooth rational hypersurfaces of odd dimension and degree 3, or of any dimension and degree at least 4.

Turning to negative results, let $X\subseteq \P^N_\mathbf C$ be a smooth complete intersection of $c$ hypersurfaces of degrees $d_1,\dotsc,d_c$.
By the adjunction formula, the canonical bundle of $X$ is $\omega_X = \O_{\P^N}(\sum  d_i - N-1)|_X$.
Thus, if $\sum d_i \ge N+1$, then $X$ has nonzero geometric genus and therefore is not rational. In fact, in this case $X$ is not even covered by rational curves, by a standard argument similar to the proof of Lemma \ref{not-uniruled} below.

Conversely, if $\sum d_i\le N$, then $X$ is \emph{Fano}, that is, has ample anti-canonical bundle.
This implies by a result of Campana and Koll\'ar-Miyaoki-Mori  \cite[Theorem V.2.13]{Kollar1996} that $X$ is rationally connected, hence covered by rational curves.
Thus Fano complete intersections are close to being rational.
A more classical result in this direction asserts that, if $\sum d_i \le N-1$, then there passes a line through every point of $X$ \cite[Exercise 4.10]{Kollar1996}.
Another states that, if the multi-degree $(d_1,\dotsc,d_c)$ is fixed, $N$ is sufficiently large, and $X$ is general, then $X$ is unirational \cite{paranjape1992}.

The first irrationality results for Fano complete intersections appeared in the early 1970s. Iskovskikh and Manin \cite{IM1971} showed that a quartic threefold $X_4\subset \P^4_\mathbf C$ has no birational automorphisms other than biregular ones, that is, $\mathrm{Bir}(X_4)=\mathrm{Aut}(X_4)$. This implies that quartic threefolds have finite birational automorphism groups and hence cannot be rational.

Iskovskikh and Manin's method has since been applied to prove the irrationality of Fano complete intersections of index 1 and 2, that is, complete intersections $X\subseteq \P^N_\mathbf C$ with $\sum d_i = N$ and with $\sum d_i = N-1$; see \cite{deFernex2013,deFernex2016,Suzuki2017} and references therein.

At around the same time as \cite{IM1971} appeared, Clemmens and Griffiths showed that all cubic threefolds $X_3\subset \P^4_\mathbf C$ are irrational \cite{CG1972}.
To this end, they introducted the intermediate Jacobian of a smooth complex projective threefold, which is a complex torus constructed via Hodge theory.
They showed that, for a rationally connected threefold, the intermediate jacobian is an Abelian variety with a natural principal polarization;
and that for a rational threefold, it is a product of Jacobians of curves (equipped with their natural polarizations).
By studying the singularites of the theta divisor giving the polarization of the intermediate Jacobian of the cubic threefold $X_3$, Clemmens and Griffiths were able to show that this intermediate Jacobian is not a product of Jacobians of curves, and therefore that $X_3$ is not rational.

Intermediate Jacobians have since been applied to prove the irrationality of other complete intersection threefolds; see \cite{Beauville2016} and references therein. 

Koll\'ar made a breakthrough with his 1995 paper \cite{kollar1995}, which establishes the irrationality of a large class of Fano hypersurfaces of dimension and index exceeding 3. Precisely, he proved that a very general hypersurface in $\P^N_\mathbf C$ of degree $d\ge 2\lceil (N+2)/3\rceil$ and dimension at least 3 is not ruled (and therefore not rational).

To exhibit a hypersurface that is not ruled, Koll\'ar considered a degeneration $f:Z\to S$ proposed by Mori \cite[Example 4.3]{Mori1975}. This is in particular a flat proper morphism of schemes with $S$ the spectrum of a discrete valuation ring. Its general fiber $Z_\eta$ is a smooth hypersurface of even degree in $\P^N_{\kappa(\eta)}$; and its special fiber $Z_s$ is a double cover of a smooth hypersurface $Y\subset \P^N_{\kappa(s)}$ of half that degree. Koll\'ar chose the discrete valuation ring to have fraction field of characteristic zero and residue field of characteristic 2; exploited the resulting inseparability of the cover $Z_s\to Y$ to construct a big invertible subsheaf $Q\subset \Omega_{B}^{\dim B-1}$, where $B\to Z_s$ is an explicit resolution of singularities; and noted that the existence of such a subsheaf implies that $B$ is not ruled.
From this it follows immediately that the special fiber $Z_s$ is not ruled, and applying a result of Matsusaka's one concludes that the geometric general fiber $Z_{\bar\eta}$ is not ruled.

Koll\'ar's work was built upon by Totaro \cite{Totaro2016}. Instead of ruledness, Totaro worked with the property of \emph{universal triviality} of the Chow group of zero-dimensional cycles modulo rational equivalence: a proper variety $V$ over a field $k$ is said to possess it if, for all extension fields $k\subset E$, the degree homomorphism $\deg: \mathrm{CH}_0(V_E)\to \Z$ is an isomorphism. Totaro's result is that a very general complex hypersurface $X\subset \P^N$ of degree $d\ge 2 \lceil (N+1)/3\rceil$ and dimension at least 3 does not have universally trivial $\mathrm{CH}_0$. Because the group $\mathrm{CH}_0$ is a birational invariant \cite[Chapter 16]{Fulton98}, this implies that the very general hypersurface $X$ is not stably rational, meaning that $X\times \P^m$ is irrational for every $m\ge 0$.

Totaro's proof uses Koll\'ar's degeneration to positive characteristic. He shows that the existence of a subsheaf $Q\subset \Omega_B^{\dim B-1}$ with nonzero global sections implies that $B$ does not have universally trivial $\mathrm{CH}_0$. Totaro then observes that this conclusion implies by a result of Colliot-Th\'el\`ene and Pirutka that the geometric generic fiber $Z_{\bar\eta}$ does not have universally trivial $\mathrm{CH}_0$.

Chatzistamatiou and Levine  \cite{CL17} apply Koll\'ar and Totaro's methods to complete intersections of codimension $c$ by degenerating a hypersurface in a fixed complete intersection of codimension $c-1$.
They show in particular that if $X\subseteq \P^N_\mathbf C$ is a very general complete intersection of multi-degree $(d_1,\dotsc,d_c)$ with $d_1$ divisible by a prime number $p\ge 3$, then $X$ is not stably rational provided that
$\frac {p+1} p d_1 + \sum_{i=2}^c d_i\ge N+1$
and not ruled provided that this inequality is strict.

In this thesis, we consider an analogue of Mori's degeneration in which all of the  equations defining a complete intersection are allowed to vary.
To state our main results, we introduce the following two finite sets of pairs of integers:
\begin{align*}
\mathscr E' := \{&
(1, 2),
(2, 3),
(2, 4),
(3, 4),
(3, 5),
(4, 5),
(4, 6),
(4, 7),
(5, 7),\\
&(5, 8),
(6, 9),
(7, 11)\}\\
\mathscr E'' :=\{&
(1, 3),
(2, 5),
(3, 6),
(4, 8),
(5, 9),
(6, 10),
(7, 12),
(8, 13)\}
\end{align*}
Let $p$ be a prime number.
If $p=2$, let $\mathscr E := \mathscr E'\sqcup \mathscr E''$; otherwise, let $\mathscr E := \mathscr E'$. 

\begin{theorem}[=Theorem \ref{MainResultIrrationalCIs}]
\label{IntroThm1}
Let $N,d_1,\dotsc,d_c$ be positive integers.
For $i=1,\dotsc,c$, let $r_i\in \{0,1,\dotsc,p-1\}$ be the remainder of the division of $d_i$ by $p$. Suppose that
\begin{enumerate}
\item $c\le \tfrac 1 2 N-1$, $c\le \frac 1 3 N +1$ and $(c,N-c)\not\in \mathscr E$; 
\item $d_1,\dotsc,d_c \ge p$; and
\item $\sum_{i=1}^c (d_i-r_i) > \frac{p}{p+1}(N+1)$.
\end{enumerate}
Then a complete intersection of $c$ very general hypersurfaces of degrees $d_1,\dotsc,d_c$ in $N$-dimensional complex projective space is not ruled (and therefore not rational).
\end{theorem}

\begin{corollary}
A complete intersection of $c$ very general hypersurfaces of degrees $d_1,\dots,d_c\ge 2$ in $N$-dimensional complex projective space is not ruled (and therefore not rational) provided that
$\sum_{i=1}^c d_i \ge \tfrac 2 3 N + c + 1$.
\end{corollary}

\begin{proof}
Let $X\subseteq \P^N_\mathbf C$ be the complete intersection of $c$ very general hypersurfaces of degrees $d_1,\dotsc,d_c$ such that $\sum_i d_i \ge \tfrac 2 3 N + c + 1.$
Then $X$ is smooth.
As was remarked above, if $\sum_i d_i \ge N+1$, then $X$ has nonzero geometric genus and is therefore not ruled (in fact, not even covered by rational curves). 
Hence we may assume that $\sum_i d_i \le N$.
It then follows that $c+1\le \tfrac 1 3 N$ and moreover that the integers $N,d_1,\dotsc,d_c$ satisfy the hypotheses of Theorem \ref{IntroThm1} with $p=2$.
Applying that theorem, the result follows.
\end{proof}

Setting $p=2$ and $c=1$ in Theorem \ref{IntroThm1}, we recover Koll\'ar's original result for hypersurfaces of dimension at least 4.
But Theorem \ref{IntroThm1} says nothing about the irrationality of the very general quartic threefold in $\P^4_\mathbf C$, which is established by Koll\'ar's result.

To prove Theorem \ref{IntroThm1}, we consider a degeneration whose generic fiber is a smooth complete intersection $Z_\eta\subseteq \P^N_K$ over a field $K$ of characteristic zero and whose special fiber is a finite insepable cover $Z_s$ of a smooth complete intersection $Y\subseteq \P^N_k$ of smaller multi-degree, but the same codimension $c$, over a field $k$ of characteristic $p$.
By Matsusaka's result (restated as Theorem \ref{Matsusaka} below), to prove that the complete intersection $Z_K$ is not ruled, it suffices to show that the special fiber $Z_s$ is not ruled.

The covering map $Z_s\to Y$ may be described in terms of Construction \ref{insep-cover-constr} below, which to a triple $(X,E,s)$, where
\begin{itemize}
\item $X$ is a scheme of characteristic $p$,
\item $E$ is a locally free sheaf of finite rank on $X$, and
\item $s\in \Gamma(X,F_X^* E)$ is a section of the pullback of $E$ along the absolute Frobenius morphism $F_X: X\to X$,
\end{itemize}
associates a finite morphism of schemes $X[\pthroot s]\to X$.
More precisely, there exist a locally free sheaf $E$ on $Y$ of rank $c$ (equal to the codimension of $Z_\eta$ in $\P^N_K$) and a section $s\in \Gamma(Y,F_Y^* E)$ such that $Z_s$ and $Y[\pthroot s]$ are isomorphic as schemes over $Y$.

We deduce the nonruledness of the special fiber $Z_s$ from the following result. More precisely, we deduce it from Theorem \ref{irrational-insep-covers} below, which is a mild generalization of Theorem \ref{IntroThm2} to nonclosed base fields.

\begin{theorem}
\label{IntroThm2}
Let $k$ be an algebraically closed field of characteristic $p$.
Let $X$ be a smooth, proper, $n$-dimensional, connected variety over $k$.
Let $E$ be an locally free sheaf of finite rank $e$ on $X$. Let $V\subset \Gamma(X,F_X^* E)$ be a $k$-linear subspace of finite dimension. 
Suppose that
\begin{enumerate}
\item $e\le n-1$, $e\le \tfrac 1 2 (n+3)$ and $(e,n)\not\in \mathscr E$;
\item the natural map $V\to (F_X^* E)/\frak m_x^3 (F_X^*E)$ is surjective for all closed points $x\in X$; and 
\item $\omega_X\otimes \det(E)^{\otimes p}$ is a big invertible sheaf on $X$.
\end{enumerate}
Let $s\in W$ be a general section.
Then the $k$-scheme $X[\pthroot s]$ is integral, normal and not separably uniruled (hence not ruled, nor rational).
\end{theorem}

If the locally free sheaf $E$ in Theorem \ref{IntroThm2} has rank $e=1$, then $F_X^* E = E^{\otimes p}$ and $X[\pthroot s]\to X$ reduces to a degree-$p$ cyclic cover of the type considered by Koll\'ar.
Comparing the rank-one case of Theorem \ref{IntroThm2} with \cite[Thm. V.5.11]{Kollar1996}, we find that the two results are identical, except that, when $p=2$, Koll\'ar's result applies as soon as dimension $n$ of $X$ is at least 3, whereas our result requires that $n\ge 4$.

To prove Theorem \ref{IntroThm2}, we exploit the inseparability of the structure map $X[\pthroot s]\to X$ to produce a big invertible subsheaf $Q\subseteq \Omega_B^{n-c}/T$, where $B\to X[\pthroot s]$ is a proper birational map and $T\subseteq \Omega_B^{n-c}$ is the subsheaf of torsion forms.
By Lemma \ref{not-uniruled} below, which has a simple proof and generalizes \cite[Lemma 7]{kollar1995} to varieties with singularities, the existence of the subsheaf $Q$ implies that $B$ and $X[\pthroot s]$ are not separably uniruled.

Thanks in part to Lemma \ref{not-uniruled}, the proof of Theorem \ref{IntroThm2} bypasses two closely related steps of Koll\'ar's argument in the rank-one case, namely (i) to explicitly classify the
singularities\footnote{In the context of Theorem \ref{IntroThm2}, it seems natural to define the \emph{singularity} of a section $s\in \Gamma(X,F_X^* E)$ at a point $x\in X$ to be the equivalence class of $e$-tuples of power series consisting of all images of $s$ under $k$-linear isomorphisms
\begin{equation*}
(F_X^* E)\otimes \widehat \O_{X,x} \xrightarrow\sim k[[x_1,\dotsc,x_n]]^{\oplus e}
\end{equation*}
induced by choices of formal coordinates around $x$ and of $\widehat\O_{X,x}$-bases for $E\otimes \widehat \O_{X,x}$.}
of the general section $s\in V$, and (ii) to explicity resolve the singularities of the inseparable cover $X[\pthroot s]$.

Step (i) fits into the framework of the theory of singularities of maps, where one of the guiding problems is the classification of  germs of smooth maps $f: (\mathbf R^n,0)\to (\mathbf R^e,0)$ up to diffeomorphisms of the source and the target. For two surveys of this subject, see \cite{Wall81, Arnold1985}.
It would be interesting to know whether, using ideas from singularity theory, one can carry out steps (i) and (ii) when the vector bundle $E$ has small rank, for example, 2 or 3.

In this thesis we do not obtain stable irrationality results for complete intersections, partly because we do not see how the counterpart in Totaro's argument to \cite[Lemma 7]{kollar1995}, namely \cite[Lemma 2.2]{Totaro2016}, could be generalized to singular varieties.
With a complete description of the singularities of the general section $s\in V$ and a resolution of the singularities of the inseparable cover $X[\pthroot s]$ at hand, one might hope to be able to apply \cite[Lemma 2.2]{Totaro2016} directly.

Our proof of Theorem \ref{IntroThm2} requires  preparation in the form of scheme-theoretic, positive-characteristic analogues of results about maps between smooth manifolds.
This preparation occupies Chapters 1 and 2 of this thesis.
We establish two main results: Proposition \ref{morse-corank-1} below, which concerns the local structure of a section of a vector bundle around a Thom-Boardman singularity of order 2 and corank 1, and Corollary \ref{generic-sigma-i-j} below, which concerns the dimensions of Thom-Boardman loci of order 2 of generic sections.
These two results are connected to Theorem \ref{IntroThm2} by the hypothesis in that theorem that the section $s\in V$ be generic.

Proposition \ref{morse-corank-1} follows from Proposition \ref{morse-parameters} below, which is a version of ``Morse's Lemma with parameters'' that holds in positive characteristics.
For another version, see \cite[Lemmas 3.9 and 3.12]{GN16}.
We derive Proposition \ref{morse-parameters} from general statements about finite determinacy and versal unfoldings of power series, namely Propositions \ref{DeterminacyBound} and \ref{versal-deform} below.
We believe that the latter two propositions are folklore, that is, that they have not yet appeared in the literature, but are well known to experts.
Analogues of these propositions in the context of maps between smooth manifolds are special cases of \cite[Theorems 1.2 and 3.4]{Wall81}.

Corollary \ref{generic-sigma-i-j} follows from the combination of Lemma \ref{serre-bertini} and Proposition \ref{univ-sigma-i-j-bis} below.
The first result is a lemma of Serre's similar in flavor to Bertini's theorem.
The second result is new. It extends to positive characteristics Levine's computation \cite[p. 55]{Levine71} of the codimensions of the universal Thom-Boardman loci of order two in jet bundles.
We note that our conventions for labelling Thom-Boardman loci are different from Levine's.
The codimensions we find agree with those found by Levine in all characteristics different from two, but disagree in characteristic two.

\chapter{Thom-Boardman singularities of order two}

\section{Vector bundles and connections}

Let $X$ be a scheme.

\begin{definition}
\label{vb-def}
Let $E$ be a locally free $\O_X$-module of finite rank. 
The \textbf{vector bundle} associated to $E$ is the $X$-scheme 
\begin{equation*}
\mathbf V(E) = \Spec_X \Sym(E^\vee).
\end{equation*}
\end{definition}

\begin{remark}
\label{vb-ump}
Let $T$ be a scheme over $X$ with structure morphism $t:T\to X$. The universal properties of the relative spectrum and of the symmetric algebra yield bijections as follows.
\begin{align*}
\Hom_X(T,\mathbf V(E)) 
&= \Hom_{\O_X\text{-alg.}}(\Sym (E^\vee),t_* \O_T)\\
&= \Hom_{\O_X}(E^\vee,t_* \O_T)\\
&= \Hom_{\O_T}(t^* E^\vee,\O_T)\\
&= \Gamma(T,E_T)
\end{align*}
\end{remark}

\begin{definition}
\label{tautological-section}
Let $E$ be a locally free $\O_X$-modules of finite rank. 
The \textbf{tautological section}
$h\in \Gamma(\mathbf V(E),E_{\mathbf V(E)})$
is the section corresponding to the identity morphism of $\mathbf V(E)$ as in Remark \ref{vb-ump}.
\end{definition}

\begin{remark}
Pullback of $h$ recovers the bijection
\begin{equation*}
\Hom_X(T,\mathbf V(E))\xrightarrow\sim \Gamma(T,E_T)
\end{equation*}
described in Remark \ref{vb-ump}.
\end{remark}

\begin{remark}
\label{vb-trivial}
Let $\A^e$ denote the affine space of dimension $e$ over $\Spec \Z$. Let $t_1,\dotsc,t_e$ denote the coordinates on $\A^e$. If $E=\O_X^{\oplus e}$, then $\mathbf V(E)=X\times \A^e$ as schemes over $X$ and $h=(t_1,\dotsc,t_e)$ as sections of $\O_{X\times \A^e}^{\oplus e}$.
\end{remark}

\begin{remark}
\label{vb-differentials}
Let $E$ be a locally free $\O_X$-module of finite rank.
Write $V:=\mathbf V(E)$.
Let $\pi: V\to X$ denote the projection.
Let $E^\vee \to \pi_* \Omega_{V/X}$ be the map that sends $\sigma \mapsto d(\sigma \cdot h)$. This map is linear over $\O_X$ hence, by adjunction, induces a map of $\O_V$-modules 
\begin{equation*}
\tau : E^\vee_V \to \Omega_{V/X}.
\end{equation*}
By the computation of the differentials on affine space, $\tau$ is an isomorphism.
\end{remark}

Let $S$ be a scheme.
Let $f : X\to S$ be a smooth morphism, which we use to regard $X$ as a scheme over $S$.
In sequel we will abuse our notation by denoting both the structure sheaf of $S$ and its inverse image along $f$ by $\O_S$.

\begin{definition}
\label{connection-def}
Let $E$ be a quasi-coherent $\O_X$-module.
A \textbf{connection} on $E$ is a $\O_S$-linear map
\begin{equation*}
\nabla:E\to \Omega_X\otimes E
\end{equation*}
satisfying the \textbf{Leibniz rule}
\begin{equation*}
\nabla(fs) = df\otimes s + f\nabla s
\end{equation*}
for all sections $f\in \O_X$ and $s\in E$ over a common open subset of $X$.
\end{definition}

\begin{example}
\label{trivial-connection}
The \textbf{trivial connection} on $\O_X^{\oplus e}$ is the sheaf morphism
\begin{equation*}
\nabla^\mathrm{triv}: \O_X^{\oplus e} \to \Omega_X \otimes \O_X^{\oplus e}
\end{equation*}
that sends $(f_1,\dotsc, f_e)\mapsto (df_1,\dotsc,df_e)$.
\end{example}

\begin{definition}
Let $E$ be a quasi-coherent $\O_X$-module.
Let $\nabla : E\to \Omega_X\otimes E$ be a connection on $E$.
Let $s$ be a section of $E$ defined over an open subset of $X$.
The \textbf{covariant derivative} of $s$ is the section $\nabla s\in \Omega_X\otimes E$. 
We say that $s$ is a \textbf{horizontal section} if its covariant derivative vanishes.
\end{definition}

\begin{remark}
\label{gend-horiz}
Let $E$ be a locally free $\O_X$-module of finite rank.
Let $\nabla : E\to \Omega_X\otimes E$ be a connection on $E$.
The subsheaf of horizontal sections $\ker(\nabla)\subseteq E$ is an $\O_S$-module.
Suppose that $E$ is generated as an $\O_X$-module by this subsheaf.
Then $X$ is covered by open subsets $U\subseteq X$ with the property that there exist sections $v_1,\dotsc,v_e\in \Gamma(U,E)$ which form a basis for $E$ over $U$ and satisfy $\nabla v_i=0$ for all $i=1,\dotsc, e$. If $s\in \Gamma(U,E)$ is another section, then $s = \sum_{i=1}^e f_i v_i$ for some $f_1,\dotsc,f_e\in \Gamma(U,\O_X)$ and $\nabla s = \sum_{i=1}^e df_i \otimes v_i$.
Thus, if a locally free sheaf is generated by the kernel of a connection on it, then that connection is locally modeled on the trivial one.
\end{remark}

In this work we will almost exclusively consider pairs $(E,\nabla)$ which satisfy the hypotheses of of Remark \ref{gend-horiz}.

\begin{construction}
\label{connection-subscheme}
Let $\Sigma$ be a subscheme of $X$ that is smooth over $S$.
Let $E$ be a quasi-coherent $\O_X$-module.
Let $\nabla : E\to \Omega_X\otimes E$ be a connection on $E$.
Then $\nabla$ induces a connection $E_\Sigma \to \Omega_\Sigma\otimes E_\Sigma$ on the restriction of $E$ to $\Sigma$, in the following way.

Let $I\subseteq \O_X$ be the ideal sheaf of $X$. 
If $f\in I$ and $s\in E$ are local sections, then $\nabla (fs)\in \Omega_X\otimes E$ maps to zero in $\Omega_\Sigma \otimes E_\Sigma$. Thus the composition of maps of abelian sheaves on $X$
\begin{equation*}
E\xrightarrow\nabla \Omega_X \otimes E \to \Omega_\Sigma \otimes E_\Sigma
\end{equation*}
vanishes on $I\cdot E$, hence factors through the quotient $E/I\cdot E = E_\Sigma$. The factoring map is the induced connection.
\end{construction}

\begin{definition}
\label{horiz-sub-def}
Let $E$ be a locally free $\O_X$-module of finite rank.
Let $V := \mathbf V(E)$ be the corresponding vector bundle over $X$.
A connection $\nabla$ on $E$ determines an $\O_V$-submodule $H\subseteq \T_V$ such that $H\oplus \T_{V/X} = \T_V$, see Construction \ref{horiz-sub-constr} below.
We call $H$ the \textbf{horizontal subbundle} of $\T_V$ corresponding to the connection $\nabla$. 
\end{definition}

The next result shows that the horizontal subbundle determines the connection.
A proof of it will be given following Construction \ref{horiz-sub-constr}.

\begin{proposition}[= Corollary \ref{horiz-sub-prop-2}]
\label{horiz-sub-prop}
Setup as in Definition \ref{horiz-sub-def}.
Let $s\in \Gamma(X,E)$ be a section.
Let $ds : \T_X\to s^* \T_V$ denote the differential of $s$ viewed as a morphism $X\to V$.
Then the following diagram of solid arrows commutes:
\begin{equation*}
\begin{tikzcd}
\T_X \ar[r,"ds"] \ar[d, "\nabla s"'] &
s^* \T_V \ar[r, two heads] &
s^*(\T_V/H) \\
E \ar[r,equals] &
s^* \T_{V/X} \ar[u,hook,dashed] \ar[ru, "\sim"']
\end{tikzcd}
\end{equation*}
\end{proposition}

\section{Degeneracy loci}

Let $X$ be a scheme. Let $\alpha:E\to F$ be a map of locally free $\O_X$-modules of finite rank.
Let $e$ and $f$ respectively denote the ranks of $E$ and $F$.
Let $m=\min(e,f)$. 

\begin{definition}
\label{degeneracy-locus-def}
Let $i$ be a nonnegative integer. The \textbf{$i$th degeneracy locus} of $\alpha$ is defined to be the subscheme $\Sigma^i(\alpha)\subseteq X$ defined by the equation
$\wedge^{m-i+1}\alpha=0$
if $i\le m+1$ and the empty scheme otherwise.
\end{definition}

Let $i$ be an integer such that $0\le i \le m$.

\begin{remark}
A point $x\in X$ lies in $\Sigma^i(\alpha)$ if, and only if, the $k(x)$-linear map $\alpha(x)$ has rank at most $m-i$.
\end{remark}

\begin{proposition}
There exist closed immersions
\begin{equation*}
\emptyset = \Sigma^{m+1}(\alpha)\subseteq \Sigma^m(\alpha) \subseteq \dotsb \subseteq 
\Sigma^0(\alpha) = X.
\end{equation*}
\end{proposition}

\begin{proof}
Follows from the Laplace expansion of the determinant.
\end{proof}

Our convention for indexing degeneracy loci has the following pleasant feature. 

\begin{proposition}
\label{deg-loci-dual}
Let $\alpha^\vee : F^\vee \to E^\vee$ be the dual of $\alpha$. Then $\Sigma^i(\alpha^\vee)=\Sigma^i(\alpha)$ as subschemes of $X$.
\end{proposition}

\begin{proof}
The minimum between the ranks of $E$ and $F$ is equal to the minimum between the ranks of $F^\vee$ and $E^\vee$.
The result therefore follows from the fact that, over any ring, the determinant of a matrix is equal to that of its transpose.
\end{proof}

\begin{proposition}
\label{deg-loci-functor}
Let $f: T\to X$ be a morphism of schemes. Then
\begin{equation*}
f^{-1}\Sigma^i(\alpha) = \Sigma^i(f^* \alpha)
\end{equation*}
as closed subschemes of $T$.
\end{proposition}

\begin{proof}
This holds because, for every $\O_X$-module $M$, the natural $\O_T$-linear map $f^*(\wedge^i M) \to \wedge^i (f^* M)$ is an isomorphism.
\end{proof}

\begin{lemma}[{\cite[Lemma 2.5]{Kleiman69}}]
\label{kleimans-lemma}
Let $R$ be a ring. Let $M$ be an $R$-module. Let $A, B\subseteq M$ be submodules. Suppose that $B$ is free, of finite rank $b$, and a direct summand of $M$. Let $q$ be an integer such that $q\ge b$. The following are equivalent:
\begin{enumerate}
\item The natural map $\wedge^q (A+B) \to \wedge^q M$ is zero.
\item The natural map $\wedge^{q-b} A\to \wedge^{q-b}(M/B)$ is zero. 
\end{enumerate}
\end{lemma}

Let $\Sigma$ denote the locally closed subscheme $\Sigma^i(\alpha)\setminus \Sigma^{i+1}(\alpha)\subseteq X$.

\begin{proposition}
\label{coker-deg-loci}
The $\O_\Sigma$-modules $\ker(\alpha|_\Sigma)$ and $\coker(\alpha|_\Sigma)$ are locally free of respective ranks $e-m+i$ and $f-m+i$.
If $t:T\to \Sigma$ is a map of schemes, then 
$
\ker (t^* (\alpha|_\Sigma)) = t^* \ker(\alpha|_\Sigma)$
and
$\coker (t^*(\alpha|_\Sigma)) = t^* \coker(\alpha|_\Sigma)$.
\end{proposition}

\begin{proof}
Let
\begin{equation}
\label{ses-locally-free}
\begin{tikzcd}
0 \ar{r} &
A \ar{r} &
B \ar{r} &
C \ar{r} &
0
\end{tikzcd}
\end{equation}
be a short exact sequence of $\O_\Sigma$-modules. Suppose that $B$ and $C$ are locally free of finite rank. Then $A$ is locally free of finite rank, and the sequence (\ref{ses-locally-free}) remains exact after pullback along any map $t:T\to \Sigma$.
Thus it suffices to show that $\coker(\alpha|_\Sigma)$ is locally free.

Let $x\in X\setminus \Sigma^{i+1}(\alpha)$. 
Then $\alpha(x): E(x)\to F(x)$ is a $k(x)$-linear map of rank greater than $m-i-1$.
Therefore there exist elements $\bar f_1,\dotsc,\bar f_{m-i}\in \im (\alpha (x))$ which may be completed to a basis of $F(x)$ as a vector space over $k(x)$. It follows by Nakayama's lemma that there exist an open subset $U\subseteq X\setminus \Sigma^{i+1}(\alpha)$ containing $x$ and sections $f_1,\dotsc,f_{m-i}\in \Gamma(U,\im(\alpha))$ which may be completed to a basis of $F_U$ as an $\O_U$-module. Replacing $\Sigma$ by $\Sigma\cap U$, we may assume that $U=X$. Then $\im(\alpha)$ contains an $\O_X$-submodule $F'$ that is locally free of rank $m-i$ and a direct summand of $F$. Let $\bar \alpha$ denote the composition $E\xrightarrow \alpha F\twoheadrightarrow F/F'$. 

The natural map $\wedge^{m-i+1}\im(\alpha|_\Sigma) \to \wedge^{m-i+1}(F|_\Sigma)$ is zero by definition of $\Sigma^i(\alpha)$. By Lemma \ref{kleimans-lemma}, this implies that the composition of natural maps
\begin{equation*}
\im(\alpha|_{\Sigma}) \twoheadrightarrow \im( \alpha|_{\Sigma}) \hookrightarrow (F/F')|_{\Sigma}
\end{equation*}
is zero. It follows that the composition of natural maps
\begin{equation*}
(F/F')|_{\Sigma} \twoheadrightarrow \coker( \alpha |_{\Sigma}) \twoheadrightarrow \coker(\bar \alpha|_{\Sigma})
\end{equation*}
is an isomorphism. Therefore $\coker( \alpha|_{\Sigma})$ is locally free.
\end{proof}

The following construction sometimes yields a resolution of the singularities of the degeneracy locus $\Sigma^i(\alpha)$.

\begin{construction}
\label{tjurina-transform}
Let $G:=G_{f-m+i}(F)$ denote the Grassmannian of rank $f-m+i$ quotients of $F$ over $X$. 
Let $q:F_G\to Q$ denote the universal quotient and let $A$ be its kernel. Thus $Q$ and $A$ are locally free $\O_G$-modules of ranks $f-m+i$ and $m-i$, respectively.
Let $Z\subseteq G$ be the closed subscheme defined by the equation $q\circ \alpha_G=0$.
Consider the following diagram,
where $\pi$ denotes the projection.
\begin{equation*}
\begin{tikzcd}
Z = \{q\circ  \alpha_G =0 \}  \ar[hook]{r} \ar[dashed, "\rho"]{d} & G \ar{d}\\
\Sigma^i( \alpha) \ar[hook]{r} & X
\end{tikzcd}
\end{equation*}
A dashed arrow $\rho$ making the diagram commute exists because $ \alpha_Z:E_Z\to F_Z$ factors though $A_Z$, which has rank $m-i$, so that $\wedge^{m-i+1}  \alpha_Z=0$.
\end{construction}

\begin{remark}
\label{tjurina-isomorphism}
The pullback of the universal quotient $q: F_G\twoheadrightarrow Q$ to $Z$ factors through a quotient of $\coker( \alpha)_Z$. This quotient induces an isomorphism of $Z$ with the Grassmannian $G_{f-m+i}(\coker( \alpha))$ of rank $f-m+i$ quotients of $\coker( \alpha)$ over $X$.

The map $\rho : Z\to \Sigma^i( \alpha)$ is an isomorphism over $\Sigma$. 
Indeed, $Z_\Sigma$ is isomorphic as a scheme over $\Sigma$ to the Grassmannian
$G_{f-m+i}(\coker( \alpha_\Sigma))$ of rank $f-m+i$ quotients of $\coker( \alpha_\Sigma)$, which is itself a locally free sheaf of rank $f-m+i$ on $\Sigma$, by Proposition \ref{coker-deg-loci}. Thus $Z_\Sigma \cong \Sigma$.
\end{remark}

Let $\pi : H\to X$ be the vector bundle associated to the locally free $\O_X$-module $\sheafHom_X(E,F)$. In symbols,
\begin{equation*}
H = \mathbf V(\sheafHom_X(E,F)).
\end{equation*}
Let $h : E_H\to F_H$ be the tautological map, see Definition \ref{tautological-section}.
Let $\Sigma'\subseteq H$ denote the degeneracy locus $\Sigma^i(h)\setminus \Sigma^{i+1}(h)$.

\begin{remark}
Viewing $\alpha \in \Gamma(X,\sheafHom_X(E,F))$ as a morphism $X\to H$, we have $\alpha^{-1}\Sigma = \Sigma'$.
\end{remark}

The following result is well known.

\begin{proposition}
\label{univ-deg-loci}
The universal degeneracy locus $\Sigma'\subseteq H$ is smooth over $X$, of relative codimension $i(|e-f|+i)$ in $H$ over $X$.
The canonical $\O_H$-linear isomorphism $\T_{H/X}\xrightarrow\sim \sheafHom_X(E,F)_H$ of Remark \ref{vb-differentials} induces an $\O_{\Sigma'}$-linear isomorphism
\begin{equation*}
\N_{\Sigma'/X} := (\T_{H/X})_{\Sigma'}/\T_{\Sigma'/X} \xrightarrow\sim
\sheafHom_{\Sigma'}(\ker(h_{\Sigma'}), \coker(h_{\Sigma'})).
\end{equation*}
\end{proposition}

\section{The intrinsic differential}

Let $S$ be a base scheme.

\begin{definition}
\label{nd}
Let $f:X\to Y$ be a morphism of smooth schemes over $S$.
Let $Z\subseteq Y$ be an $S$-smooth, locally closed subscheme of $Y$. 
\begin{enumerate}
\item The \textbf{normal differential} of $f$ with respect to $Z$, denoted $\mathbf d_Z f$, is the composition $\O_{f^{-1}Z}$-linear maps,
\begin{equation*}
\begin{tikzcd}
\T_X|_{f^{-1}Z} \ar["df"]{r} &
f^* (\T_Y)_Z \ar[two heads]{r} &
f^*(\T_Y)_Z/f^* \T_Z = f^* \N_Z.
\end{tikzcd}
\end{equation*}
\item The map $f$ is \textbf{transverse} to $Z$ if the normal differential $\mathbf d_Z f$ is surjective. We write $f\pitchfork Z$ to indicate that this condition holds.
\end{enumerate}
\end{definition}

\begin{remark}
\label{transversality}
By \cite[Proposition IV.17.13.2]{EGA}, the morphism $f:X\to Y$ is transverse to $Z$ if, and only if, the scheme-theoretic inverse image $f^{-1}Z$ is smooth over $S$ and, for all $x\in f^{-1}Z$, the relative codimension of $f^{-1}Z$ in $X$ at $x$ over $S$ is equal to the relative codimension of $Z$ in $Y$ at $f(x)$ over $S$.
\end{remark}

Let $X$ be a smooth scheme over $S$.
Let $E$ be a locally free sheaf of finite rank on $X$.
Let $V := \mathbf V(E)$ be the vector bundle associated to $E$. Let $\pi: V\to X$ be the projection.
Let $Z \subseteq V$ be an $S$-smooth, locally closed subscheme. Let $s\in \Gamma(X,E)$ be a section. Viewing $s$ as a morphism $X\to V$, we now consider its normal differential 
\begin{equation*}
\label{nd-section}
\mathbf d_Z s : (\T_X)_{s^{-1}Z} \to s^* \N_Z.
\end{equation*}

\begin{proposition}
\label{nd-horizontal}
Let $\nabla : E\to \Omega_X\otimes E$ be a connection.
Suppose that the tangent bundle $\T_Z$ contains the restriction to $Z$ of horizontal subbundle $H \subseteq \T_V$ corresponding to $\nabla$. 
Then the normal differential $\mathbf d_Z s$ is equal to the composition
\begin{equation*}
\begin{tikzcd}
(\T_X)_{s^{-1}Z} \ar[r,"\nabla s"] &
E_{s^{-1}Z} = (\T_{V/X})_{s^{-1}Z} \ar[r,hook]  &
(\T_V)_{s^{-1}Z} \ar[r,two heads] &
s^* \N_Z.
\end{tikzcd}
\end{equation*}
\end{proposition}

\begin{proof}
Follows immediately from Proposition \ref{horiz-sub-prop}.
\end{proof}

Let $A$ and $B$ be locally free sheaves of finite rank on $X$.
Let $\alpha :A \to B$ be an $\O_X$-linear map.
Let $i\ge 0$ be a nonnegative integer.
Let $\Sigma$ denote the degeneracy locus $\Sigma^i(\alpha)\setminus \Sigma^{i+1}(\alpha)\subseteq X$.

\begin{definition}
\label{id}
The \textbf{intrinsic differential} of $\alpha$ on $\Sigma$ is the $\O_{\Sigma}$-linear map 
\begin{equation*}
\mathbf d_{\Sigma} \alpha : (\T_X)_{\Sigma} \to
\sheafHom_{\Sigma}(\ker(\alpha_{\Sigma}), \coker( \alpha_{\Sigma} ))
\end{equation*}
defined as follows.

Let $H\to X$ be the vector bundle associated to the locally free sheaf $\sheafHom_X(A,B)$. Let $h: A_H\to B_H$ be the tautological map, see Definition \ref{tautological-section}. 
Let $\Sigma'\subseteq H$ denote the degeneracy locus $\Sigma^i(h)\setminus \Sigma^{i+1}(h)\subseteq H$.
Identifying $\alpha$ with a morphism $X\to H$, we have $\alpha^* h = \alpha$ and $\alpha^{-1}\Sigma' = \Sigma$ as subschemes of $X$.
By Proposition \ref{univ-deg-loci}, the scheme $\Sigma'$ is smooth over $X$ and there exists a canonical isomorphism
\begin{equation*}
(\T_{H/X})_{\Sigma'}/\T_{\Sigma'/X} \xrightarrow\sim
\sheafHom_{\Sigma'}(\ker(h_{\Sigma'}), \coker( h_{\Sigma'} )).
\end{equation*}
The normal differential of $\alpha : X\to H$ with respect to $\Sigma'$ is an $\O_\Sigma$-linear map $\mathbf d_{\Sigma'} \alpha : (\T_X)_\Sigma \to \alpha^* \N_{\Sigma'}$.
Because $\Sigma'$ is smooth over $X$, the natural $\O_{\Sigma'}$-linear map
\begin{equation*}
(\T_{H/X})_{\Sigma'}/\T_{\Sigma'/X} \to (\T_H)_{\Sigma'}/\T_{\Sigma'} = \N_{\Sigma'}
\end{equation*}
is an isomorphism.
Furthermore, by Proposition \ref{coker-deg-loci}, the natural $\O_\Sigma$-linear map $\alpha^* \ker(h_{\Sigma'})\to \ker(\alpha_\Sigma)$ is an  isomorphism.
Thus there exists a natural $\O_\Sigma$-linear isomorphism
\begin{equation*}
\delta : \alpha^* \N_{\Sigma'} \xrightarrow\sim
\sheafHom_{\Sigma}(\ker(\alpha_{\Sigma}), \coker( \alpha_{\Sigma} )).
\end{equation*}
Set $\mathbf d_{\Sigma} \alpha := \delta \circ  \mathbf d_{\Sigma'} \alpha$. 
\end{definition}

\begin{proposition}
\label{id-locally}
Suppose that the $\O_X$-modules $A$ and $B$ are free.
Choose bases on $A$ and $B$ and let 
\begin{equation*}
D : \sheafHom(A,B)\to \Omega_X\otimes \sheafHom(A,B)
\end{equation*}
be the trivial connection corresponding to theses bases, see Example \ref{trivial-connection}.
Then the intrinsic differential $\mathbf d_{\Sigma} \alpha$ is equal to the composition
\begin{equation*}
\begin{tikzcd}
(T_X)_{\Sigma} \ar[r,"D\alpha"] &
\sheafHom_X(A,B)_{\Sigma} \ar[r,two heads] &
\sheafHom_{\Sigma}(\ker(\alpha_{\Sigma}), \coker(\alpha_{\Sigma})).
\end{tikzcd}
\end{equation*}
\end{proposition}

\begin{proof}
Let $\Sigma'\subseteq H$ be as in Definition \ref{id}.
By Corollary \ref{nd-horizontal}, it suffices to show that $\T_{\Sigma'}$ contains the restriction to $\Sigma'$ of the horizontal subbundle $\tilde H\subseteq \T_V$ determined by $D$.
Let $a:=\operatorname{rank}(A)$ and $b:=\operatorname{rank}(B)$.  The chosen bases on $A$ and $B$ induce an isomorphism $\theta: V\xrightarrow\sim X\times \A^{ab}$ over $X$.
The differential
\begin{equation*}
d\theta : \T_V \to \theta^* \T_{X\times \A^{ab}} = \theta^* \mathrm{pr}_1^* \T_X \oplus \theta^* \mathrm{pr}_2^* \T_{\A^{ab}}
\end{equation*}
maps $\tilde H$ isomorphically onto the first summand.
The tautological map $h$ is given by a matrix whose entries are the coordinates on $\A^{ab}$.
Thus $\theta(\Sigma')=X\times \Sigma''$, where $\Sigma''\subseteq \A^{ab}$ is a locally closed subscheme that is smooth over $S$. Hence  $d\theta (\T_{\Sigma'}) \supseteq (\mathrm{pr}_1^* \T_X)_{\Sigma'}$, which implies the result.
\end{proof}

\section{Second-order singularities}

Let $S$ be a scheme.
Let $X$ be a smooth scheme over $S$.
Let $E$ be a locally free sheaf of finite rank on $X$.
Let $\nabla : E\to \Omega_X\otimes E$ be a connection on $E$. 
Let $s\in \Gamma(X,E)$ be a section.

\begin{definition}
Let $i$ be a nonnegative integer.
The \textbf{$i$th critical locus} of $s$ is the $i$th degeneracy locus of the $\O_X$-linear map $\nabla s : \T_X\to E$. It is a subscheme of $X$ that we denote by $\Sigma^i(s)$.
\end{definition}

Let $i$ be a nonnegative integer.
Let $\Sigma$ denote the critical locus $\Sigma^i(s)\setminus \Sigma^{i+1}(s)$.
Let $K := \ker((\nabla s)_\Sigma)$ and $C=\coker((\nabla s)_\Sigma)$.
Note that $K$ and $C$ are locally free $\O_\Sigma$-modules by Proposition \ref{coker-deg-loci}.

\begin{remark}
The intrinsic differential of $\nabla s : \T_X\to E$ on $\Sigma$ is an $\O_\Sigma$-linear map 
\begin{equation*}
\mathbf d_{\Sigma}(\nabla s) : (\T_X)_{\Sigma} \to \sheafHom_{\Sigma}(K,C),
\end{equation*}
see Definition \ref{id}.
Suppose that $i\le \min(n,e)$, so that $\Sigma$ has a chance of being nonempty.
Then $\Sigma$ is smooth of relative codimension $i(|n-e|+i)$ in $X$ over $S$ if, and only if, $\mathbf d_\Sigma(\nabla s)$ is surjective, see Remark \ref{transversality}.
\end{remark}

\begin{definition}
\label{second-id}
The \textbf{second intrinsic differential} of $s$ on $\Sigma = \Sigma^i(s)\setminus \Sigma^{i+1}(s)$ is the $\O_{\Sigma}$-linear map
\begin{equation*}
\mathbf d_\Sigma^2 s : K\to
\sheafHom_\Sigma(K,C)
\end{equation*}
obtained by restricting the intrinsic differential $\mathbf d_\Sigma(\nabla s)$ to $K$.
\end{definition}

\begin{definition}
\label{2nd-thom-boardman}
Let $j$ be a nonnegative integer.
The \textbf{(second-order) Thom-Boardman singularity} of \textbf{symbol} $(i,j)$ of the section $s$, denoted $\Sigma^{i,j}(s)$, is the $j$th degeneracy locus of the second intrinsic differential of $s$ on $\Sigma^i(s)\setminus \Sigma^{i+1}(s)$. In symbols,
\begin{equation*}
\Sigma^{i,j}(s) := \Sigma^j( \mathbf d_\Sigma^2 s) \subseteq X.
\end{equation*}
\end{definition}

\begin{remark}
Definition \ref{2nd-thom-boardman} was originally proposed by Porteous \cite{Porteous71}, although in a slightly different setting. Porteous considered maps between smooth manifolds (or nonsingular varieties), rather then sections of vector bundles.
His definition was motivated by an analogue of Proposition \ref{tb-vs-porteous} below. Corollary \ref{1st-sings-generic} below gives conditions under which the hypothesis of this proposition is satisfied.
\end{remark}

\begin{proposition}
\label{tb-vs-porteous}
Suppose that $\Sigma$ is smooth of relative codimension $i(|n-e|+i)$ in $X$ over $S$.
Endow the restriction of $E$ to $\Sigma$ with the connection induced by $\nabla$, see Definition \ref{connection-subscheme}.
Then it makes to talk about the critical loci of the restriction $s_\Sigma\in \Gamma(\Sigma, E_\Sigma)$, and we have $\Sigma^j(s_\Sigma) = \Sigma^{i,j}(s)$ for each nonnegative integer $j$.
\end{proposition}

\begin{proof}
As we won't need Proposition \ref{tb-vs-porteous} in the sequel, we just give the idea of the proof: one considers the commutative diagram of $\O_\Sigma$-modules
\begin{equation*}
\begin{tikzcd}[column sep=large]
&
&
K \ar[d,hook] \ar[rd,"\mathbf d_\Sigma^2 s"] \\
0 \ar[r] &
\T_{\Sigma} \ar[r] \ar[rd,"\nabla(s_{\Sigma})"'] &
(\T_X)_{\Sigma} \ar[d, "(\nabla s)_{\Sigma}" ] \ar[r,"\mathbf d_{\Sigma}(\nabla s)"] &
\sheafHom_{\Sigma}(K,C) \ar[r] &
0 \\
&
& E_{\Sigma}
\end{tikzcd}
\end{equation*}
and applies Lemma \ref{kleimans-lemma} to $K$ and $\T_\Sigma$ viewed as locally free, locally split $\O_\Sigma$-submodules of $(\T_X)_\Sigma$.
\end{proof}

\section{The second intrinsic differential, locally}

Let $S$ be a scheme.
Let $X$ be a smooth scheme over $S$.
Let $E$ be a locally free sheaf of finite rank on $X$.
Let $n$ denote the relative dimension of $X$ over $S$.
Let $e$ denote the rank of $E$. Let $m=\min(n,e)$.

Let $\nabla : E\to \Omega_X\otimes E$ be a connection on $E$. Suppose that the $\O_S$-submodule of horizontal sections $\ker(\nabla)\subseteq E$ generates $E$ as an $\O_X$-module.

Let $s\in \Gamma(X,E)$ be a section.
Let $x\in X$ be a point.
Let $i$ be the unique nonnegative integer such that $x$ is containted in the subscheme $\Sigma^i(s)\setminus \Sigma^{i+1}(s) \subseteq X$. Let $\Sigma$ denote this subscheme.

\begin{remark}
Note that $i\le m$, since $\Sigma^i(s)$ is nonempty.
The tensor product of the $\O_X$-linear map $\nabla s : \T_X\to E$ with residue field $k(x)$ has rank $m-i$.
\end{remark}

Because $E$ is generated by $\ker(\nabla)$, there exist an open subset $U\subseteq X$ containing $x$ and a basis $v_1,\dotsc,v_e\in \Gamma(U,E)$ of $E$ over $U$ such that $\nabla v_l=0$ for all $l=1,\dotsc,e$. 
Fix a choice of such an open subset $U$ and basis $v_1,\dotsc,v_n$. 
Let $f_1,\dotsc,f_e\in \Gamma(U,\O_X)$ be the components of $s$ in this basis, so that 
\begin{equation*}
s_U = \sum_{l=1}^e f_l v_l
\qquad\text{and}\qquad
(\nabla s)_U = \sum_{l=1}^e df_l\otimes v_l.
\end{equation*}

\begin{definition}
Let $W$ be a smooth scheme of relative dimension $n$ over $S$. Let $x_1,\dotsc,x_n\in \Gamma(W,\O_W)$ be sections.
Then the following conditions are equivalent:
\begin{enumerate}
\item The $S$-morphism $(x_1,\dotsc,x_n) : W \to \mathbf A_S^n$ is \'etale.
\item The differentials $dx_1,\dotsc, dx_n\in \Gamma(W,\Omega_W)$ form a basis for $\Omega_W$ as an $\O_W$-module.
\end{enumerate}
If these conditions hold, we say that $x_1,\dotsc,x_n$ are \textbf{\'etale coordinates} on $W$.
\end{definition}

\begin{proposition}
\label{local-coords-prop}
After possibly shrinking $U$ to a smaller Zariski open neighborhood of $x$ in $X$ and reordering the sections $v_1,\dotsc,v_l\in\Gamma(U,E)$, there exist \'etale coordinates $x_1,\dotsc,x_n$ on $U$ such that
\begin{equation*}
(f_1,\dotsc,f_e) = (x_1,\dotsc,x_{m-i}, f_{m-i+1},\dotsc,f_e).
\end{equation*}
\end{proposition}

\begin{proof}
The images of the sections $df_1,\dotsc,df_e$ span a $k(x)$-linear subspace of $\Omega_X(x)$ of dimension $m-i$. Reordering the basis $v_1,\dotsc,v_e$, we may assume that the images of $df_1,\dotsc,df_{m-i}$ form a basis for this subspace. The stalk $\Omega_{X,x}$ is generated as an $\O_{X,x}$-module by sections of the form $dg$ with $g\in \O_{X,x}$. After possibly shrinking $U$, we may find sections $x_{m-i+1},\dotsc,x_n$ whose differentials complete the images of $df_1,\dotsc,df_{m-i}$ to a basis of $\Omega_X(x)$. By Nakayama's lemma, after further shrinking $U$, we may assume that $df_1,\dotsc,df_{m-i},dx_{m-i+1},\dotsc, dx_n$ form a basis for $\Omega_X$ over $U$.
\end{proof}

\begin{corollary}
\label{local-coords-cor}
Suppose that $S$ is the spectrum of a field $k$ and that $x\in X$ is a $k$-rational point.
After possibly shrinking $U$ to a smaller Zariski open neighborhood of $x$ in $X$ and reordering the sections $v_1,\dotsc,v_l\in\Gamma(U,E)$, there exist \'etale coordinates $x_1,\dotsc,x_n$ on $U$ such that $x_a(x)=0$ for all $a=1,\dotsc, n$ and
\begin{equation*}
(f_1,\dotsc,f_e) = (f_1(x) + x_1,\dotsc,f_{m-i}(x) + x_{m-i}, f_{m-i+1},\dotsc,f_e).
\end{equation*}
\end{corollary}

Reordering the basis sections $v_1,\dotsc, v_l\in \Gamma(U,E)$ if necessary, fix \'etale coordinates $x_1,\dotsc,x_n$ as in Proposition \ref{local-coords-prop} (or as in Corollary \ref{local-coords-cor}, if $S$ is spectrum of a field $k$ and $x\in X$ is a rational point).

Let $\partial_1,\dotsc,\partial_n\in \Gamma(U,(\T_X)_U)$ be the basis of $(\T_X)_U=(\Omega_X)^\vee_U$ that is dual to $dx_1,\dotsc,dx_n$. We identify the vector field $\partial_a$ with the derivation $\partial_a\circ d : \O_U\to \O_U$, for all $a=1,\dotsc, n$.

\begin{remark}
If $g\in \O_U$ is a local section, then 
$dg = \partial_1 g\cdot dx_1 + \dotsb + \partial_n g\cdot dx_n$.
In particular, 
\begin{align*}
(\nabla s)_U &=
\sum_{l=1}^e \sum_{b=1}^n
(\partial_b f_l) dx_b\otimes v_l 
= \sum_{l\le m-i} dx_l \otimes v_l +
\sum_{l>m-i} \sum_{b=1}^n (\partial_b f_l) dx_b \otimes v_l.
\end{align*}
\end{remark}

Regarding the covariant derivative $\nabla s$ as an $\O_X$-linear map $\T_X\to E$, 
let $K:= \ker((\nabla s)_{\Sigma\cap U})$ and $C:= \coker((\nabla s)_{\Sigma\cap U})$.

\begin{proposition}
The images of $\partial_{m-i+1},\dotsc,\partial_n$ in $(\T_X)_{\Sigma\cap U}$ lie in $K$ and form a basis for $K$ as an $\O_{\Sigma\cap U}$-module. The images of $v_{m-i+1},\dotsc,v_e$ in $C$ form a basis for $C$ as an $\O_{\Sigma\cap U}$-module.
\end{proposition}

\begin{proof}
Recall that $K$ and $C$ are locally free $\O_{\Sigma}$-modules of respective ranks $n-m+i$ and $e-m+i$, see Proposition \ref{coker-deg-loci}.

In $C$, we have
\begin{equation}
\label{local-in-cokernel}
0 = \nabla s (\partial_b) =
\begin{cases}
v_b + \sum_{l>m-i} (\partial_b f_l) v_l & \text{if $b\le m-i$}\\
\sum_{l>m-i} (\partial_b f_l) v_l &\text{if $b>m-i$}.
\end{cases}
\end{equation}
The equalities (\ref{local-in-cokernel}) for $b\le m-i$ show that $v_{m-i+1},\dotsc,v_e$ generate, hence form a basis for $C$ as an $\O_\Sigma$-module. This fact combined with the equalities (\ref{local-in-cokernel}) for $b>m-i$ show that $\partial_b f_l=0$ as a section of $\O_\Sigma$ for all $l=m-i+1,\dotsc,e$ and $b=m-i+1,\dotsc,n$. In particular, the sections $\partial_b\in (\T_X)_\Sigma$ with $b>m-i$, lie in $K$ for all $b>m-i$, hence form a basis for $K$.
\end{proof}

The following proposition shows that the second intrinsic differential is represented by the bottom-right square submatrices of size $n-e+i$ of the Hessian matrices of $f_{m-i+1},\dotsc, f_e$.

\begin{proposition}
\label{2nd-id-locally}
The second intrinsic differential $\mathbf d_\Sigma^2 s$ is the unique $\O_\Sigma$-linear map $K\to \sheafHom_\Sigma(K,C)$ such that
\begin{equation*}
(\mathbf d_\Sigma^2 s)(\partial_a)(\partial_b)=
\sum_{l=m-i+1}^e (\partial_a\partial_b f_l) v_l
\end{equation*}
for all $a, b = m-i+1,\dotsc, n$.
\end{proposition}

\begin{proof}
Let $H(s) : \T_X \to \sheafHom_X(\T_X, E)$ be the unique $\O_X$-linear map such that
\begin{equation*}
H(s)(\partial_a) =
\sum_{l=1}^e \sum_{b=1}^n
(\partial_a\partial_b f_l) dx_b\otimes v_l = 
\sum_{l>m-i} \sum_{b=1}^n (\partial_a \partial_b f_l) dx_b \otimes v_l
\end{equation*}
for all $a=1,\dotsc, n$.
Let 
\begin{equation*}
\delta : \sheafHom_X(\T_X,E)_\Sigma \to
\sheafHom_\Sigma(K,C)
\end{equation*}
be the natural map induced by the inclusion $K\hookrightarrow (\T_X)_\Sigma$ and the projection $(\T_X)_\Sigma\twoheadrightarrow C$. Then
\begin{equation*}
\delta \circ H(s)_\Sigma = \mathbf d_{\Sigma^i}(\nabla s)
\end{equation*}
by Proposition \ref{id-locally}. The result follows.
\end{proof}

For the remainder of this section, we assume that $n\ge e$ and $i=1$.
Then the second intrinsic differential $\mathbf d_\Sigma^2 s : K\to \sheafHom_\Sigma(K,C)$ is a map of locally free sheaves of rank $n-e+1$ on $\Sigma := \Sigma^1(s)\setminus \Sigma^2(s)$.
Consider the natural inclusions of second-order Thom-Boardman loci
\begin{equation*}
\emptyset = \Sigma^{1,n-e+2}(s) \subseteq
\Sigma^{1,n-e+1}(s) \subseteq \dotsb \subseteq
\Sigma^{1,0}(s) = \Sigma^1(s)\setminus \Sigma^2(s). 
\end{equation*}
Let $j$ be the unique integer such that $x\in \Sigma^{1,j}(s)\setminus \Sigma^{1,j+1}(s)$.

\begin{corollary}
\label{2nd-intrinsic-corank-1}
Let $\overline H$ denote the lower-right minor of size $n-e+1$ in the Hessian matrix of $f_e\in \Gamma(X,\O_X)$, that is, let
\begin{equation*}
\overline H := (\partial_a \partial_b f_e)_{e\le a,b\le n}
\in \Gamma(X,\O_X)^{\oplus( n-e+1 )^2}.
\end{equation*}
Then $\overline H$ has rank $n-e+1-j$ at $x$.
\end{corollary}

\begin{proof}
Follows immediately from Proposition \ref{2nd-id-locally}.
\end{proof}

Suppose that $S$ is the spectrum of an algebraically closed field $k$.
Suppose that $x\in \Sigma^{1,j}(s)\setminus \Sigma^{1,j+1}(s)$ is a closed point.
Suppose that the \'etale coordinates $x_1,\dotsc,x_n\in \Gamma(U,\O_X)$ are as in Corollary \ref{local-coords-cor}.
Then  $x_a\in \frak m_x$ for all $a=1,\dotsc, n$, and
\begin{equation*}
\widehat\O_{X,x} = k[[x_1,\dotsc,x_n]].
\end{equation*}

\begin{proposition}
\label{morse-corank-1}
The nonnegative integer $n-e+1-j$ is even if $k$ has characteristic 2. 
Let
\begin{equation*}
q := \begin{cases}
x_e^2 + \dotsb + x_{n-j}^2 &
\text{if }\operatorname{char}(k)\ne 2, \\
x_e x_{e+1} + \dotsb + x_{n-j-1} x_{n-j} &
\text{if }\operatorname{char}(k)=2.
\end{cases}
\end{equation*}
Then there exist a power series $h\in k[[x_1,\dotsc,x_n]]$ contained in the ideal
\begin{equation*}
\langle x_1,\dotsc,x_{e-1} \rangle + \langle x_{n-j+1},\dotsc,x_n \rangle^2
\end{equation*}
and which doesn't depend on the variables $x_e,\dotsc,x_{n-j}$, and an automorphism $\phi$ of
\begin{equation*}
\widehat \O_{X,x} = k[[x_1,\dotsc,x_n]]
\end{equation*}
as a local $k[[x_1,\dotsc, x_{e-1}]]$-algebra such that $\phi(f_e) = f_e(x) + q + h$.
\end{proposition}

\begin{proof}
Write $f_e = g_0 + g_1 + g_2$, where $g_0\in k$ is a constant, $g_1$ is homogeneous polynomial of degree 1 in $x_1,\dotsc,x_n$, and $g_2\in \langle x_1,\dotsc, x_n\rangle^2$.
Then $g_1$ only involves the variables $x_1,\dotsc,x_{e-1}$, since by assumption $x\in \Sigma^1(s)$, which translates to
\begin{equation*}
0 = (df_1\wedge \dotsb \wedge df_e) (x) = 
(dx_1\wedge \dotsb \wedge dx_{e-1}\wedge dg_1)(x). 
\end{equation*}
Let $\bar g_2 := g_2(0,\dotsc,0,x_e,\dotsc,x_n)\in k[[x_e,\dotsc,x_n]]$.
By Proposition \ref{2nd-intrinsic-corank-1}, the Hessian matrix of $\bar g_2$ has rank $n-e+1-j$ at the origin.
Viewing $g_2$ as an unfolding of $\bar g_2$ over $R:= k[[x_1,\dotsc,x_{e-1}]]$ and applying Morse's Lemma with Parameters (Proposition \ref{morse-parameters} below), we may find an automorphism $\phi$ of $k[[x_1,\dotsc,x_n]]$ as a local $k[[x_1,\dotsc,x_{e-1}]]$-algebra that sends $g_2$ to $q + h'$ for some power series $h'\in k[[x_1,\dotsc,x_n]]$ that does not involve the variables $x_e,\dotsc,x_{n-j}$. Setting $h := g_1 + h'$, the result follows. 
\end{proof}

\section{Power series with finite Milnor number}

Let $k$ be a field. 
Let $x=(x_1,\dotsc,x_n)$ be a finite set of indeterminates.
Let $f\in k[[x]]$ be a power series.

\begin{definition}
The \textbf{Jacobian ideal} of $f$, denoted $\operatorname{jac}(f)$, is the ideal generated in the power series ring $k[[x]]$ by the partial derivatives $\partial f/\partial x_1,\dotsc, \partial f/\partial x_n$.
The quotient $k[[x]]/\operatorname{jac}(f)$ is called the \textbf{Milnor algebra} of $f$.
Its (possibly infinite) dimension as a vector space over $k$ is called the \textbf{Milnor number} of $f$ and denoted by $\mu(f)$.
\end{definition}

\begin{definition}
Let $r$ be a positive integer.
We say that $f\in k[[x]]$ is \textbf{$r$-determined} if for every power series $g\in k[[x]]$ such that $f-g\in \langle x\rangle^{r+1}$, there exists an automorphism of $k[[x]]$ as a local $k$-algebra that sends $g$ to $f$.
We say that $f$ is \textbf{finitely determined} if it is $r$-determined for some $r\ge 1$.
\end{definition}

\begin{proposition}
\label{DeterminacyBound}
If $f\in k[[x]]$ has finite Milnor number, then $f$ is finitely determined.
More precisely, let $r$ be a positive integer.
If $\langle x\rangle^r\subseteq \operatorname{jac}(f)$, then $f$ is $2r$-determined.
\end{proposition}

\begin{proof}
Follows from the proof of Lemma 10.8 in \cite{Milnor68}.
\end{proof}

\begin{remark}
In \cite{BGM12} it is shown that, in fact, $f$ has finite Milnor number if, and only if, it is finitely determined.
They also show that, if $f\in \langle x\rangle^2$ and $\langle x\rangle^{r+2}\subseteq \langle x\rangle^2\operatorname{jac}(f)$, then $f$ is $(2r-\operatorname{ord}(f)+2)$-determined, where $\operatorname{ord}(f)$ is the maximum integer $o$ such that $f\in \langle x\rangle^o$.

The analogues of these results for germs of real- or complex-valued, analytic or infinitely differentiable functions follow from Theorem 1.2 in \cite{Wall81}.
\end{remark}

\begin{example}
Suppose that $n$ is even if $k$ has characteristic 2, and let 
\begin{equation*}
q=\begin{cases}
x_1^2 + \dotsb + x_n^2 & \text{if $\operatorname{char}(k)\ne 2$}\\
x_1x_2 + \dotsb + x_{n-1}x_n & \text{if $\operatorname{char}(k)=2$}.
\end{cases}
\end{equation*}
Let $f\in k[[x]]$ be a power series.
If $f - q\in \langle x\rangle^3$, there exists an automorphism of $k[[x]]$ as a local $k$-algebra that sends $f$ to $q$. Indeed, $\jac(q) = \langle x\rangle$, so $q$ is 2-determined by Proposition \ref{DeterminacyBound}.
\end{example}

Let $\mathbf C$ be the category whose objects are complete, Noetherian, local $k$-algebras with residue field $k$, and whose morphisms are maps of local $k$-algebras.

\begin{definition}
Let $R$ be a complete local $k$-algebra in $\mathbf C$. 
\begin{enumerate}
\item An \textbf{unfolding} of $f$ over $R$ is a power series $F\in R[[x]]$ that maps to $f\in k[[x]]$ under the quotient map $R\to k$.
\item Let $F, F' \in R[[x]]$ be unfoldings of $f$ over $R$.
A \textbf{morphism} (or \textbf{right-equivalence}) $F\to F'$ is a local $R$-algebra map $\phi : R[[x]]\to R[[x]]$ that lifts the identity of $k[[x]]$ and sends $F$ to $F'$.
\end{enumerate}
\end{definition}

\begin{remark}
Unfoldings of $f$ over $R$ and morphisms between them form a category (in fact, a groupoid) that we denote by $\mathscr D(R)$.
A map $b: R\to R'$ of complete local $k$-algebras in $\mathbf C$ induces an obvious functor functor $b_*: \mathscr D(R)\to \mathscr D(R')$. 
\end{remark}

\begin{definition}
The \textbf{functor of unfoldings} of $f$ is the functor
\begin{equation*}
D: \mathbf C\to (\mathrm{Sets})
\end{equation*}
that sends a complete local $k$-algebra $R\in \mathbf C$ to the set $D(R)$ of isomorphism classes of unfoldings of $f$ over $R$, and acts on morphisms in the obvious way.
\end{definition}

\begin{definition}
Let $R$ be a complete local $k$-algebra in $\mathbf C$. Let $F\in R[[x]]$ be a unfolding of $f$ over $R$. We say that $F$ is \textbf{right-complete} if, for every complete local $k$-algebra $A$ in $\mathbf C$, the map 
\begin{equation*}
\Hom_{\mathbf C}(R,A)\to D(A)
\end{equation*}
that sends $b\mapsto b_*F$ is surjective.
\end{definition}

\begin{proposition}
\label{versal-deform}
Suppose that $f$ has finite Milnor number.
Let $g_1, \dotsc, g_\mu \in k[[x]]$ be power series whose images span the Milnor algebra $k[[x]]/\jac(f)$ as a vector space over $k$.
Let $s= (s_1,\dotsc, s_\mu)$ be a set of $\mu$ indeterminates.
Then
\begin{equation*}
F := f + s_1 g_1 + \dotsb + s_\mu g_\mu \in k[[s,x]]
\end{equation*}
is a right-complete unfolding of $f$ over $k[[s]]$.
\end{proposition}

\begin{proof}
Follows from the method of proof of Corollary 1.17 in \cite{GLS2007}.
\end{proof}

\begin{remark}
The analogues of Proposition \ref{versal-deform} for unfoldings of  germs of real- or complex-valued, analytic or infintely differentiable functions are special cases of Theorem 3.4 in \cite{Wall81}.
\end{remark}

\begin{example}
\label{morse-versal}
Suppose that $n$ is even if $k$ has characteristic 2, and let 
\begin{equation*}
q=\begin{cases}
x_1^2 + \dotsb + x_n^2 & \text{if $\operatorname{char}(k)\ne 2$}\\
x_1x_2 + \dotsb + x_{n-1}x_n & \text{if $\operatorname{char}(k)=2$}.
\end{cases}
\end{equation*}
Let $s$ be an indeterminate. Then $q+s\in k[[s,x]]$ is a right-complete unfolding of $q$. Indeed, $\jac(q) = \langle x\rangle$, so the Milnor algebra $k[[x]]/\jac(q)$ is a one-dimensional $k$-vector space generated by the image of $1\in k[[x]]$.
\end{example}

\section{Morse's Lemma with parameters}

Let $k$ be an algebraically closed field of characteristic $p\ge 0$.
Let $V$ be a finite dimensional vector space over $k$.
Let $q\in \Sym^2( V^\vee)$ be a quadratic form on $V$.
Let $h\in (V\otimes V)^\vee$ be the bilinear form associated to $q$, so that
\begin{equation*}
h(u \otimes v) = q(u+v) - q(u) - q(v)
\end{equation*}
for all $u,v\in V$.

\begin{lemma}
\label{classification-of-QFs}
Let $r$ denote the rank of $h$. Then $r$ is even if $p=2$.
Let $x_1,\dotsc,x_n$ be a $k$-linear basis of $V^\vee$. Let 
\begin{equation*}
q_0=\begin{cases}
x_1^2 + \dotsb + x_r^2 & \text{if $p\ne 2$}\\
x_1x_2 + \dotsb + x_{r-1}x_r & \text{if $p=2$}.
\end{cases}
\end{equation*}
If $p\ne 2$, then there exists a $k$-linear automorphism $\phi$ of $V^\vee$ such that $\Sym^2(\phi)$  sends $q_0$ to $q$.
If $p=2$, then there exists a $k$-linear automorphism $\phi$ of $V$ such that $\Sym^2(\phi)$ sends $q_0$ to either $q$ or $q+x_{r+1}^2$. 
\end{lemma}

\begin{proof}
We consider the case where $k$ has characteristic 2, which is less standard.
By \cite[Satz 2 on p. 150]{Arf1941}, without using the hypothesis that $k$ is algebraically closed, we may find a basis of $V^\vee$ consisting of vectors
\begin{equation*}
u_1,\dotsc, u_s, v_1,\dotsc, u_s, w_1,\dotsc, 
w_{n-2s}
\end{equation*}
such that 
\begin{equation*}
q = \sum_{i=1}^{s} (a_i u_i^2 + b_i u_i v_i + c_i v_i^2) + \sum_{j=1}^t d_j w_j^2.
\end{equation*}
Here $s$ and $t$ are nonnegative integers satisfying $2s+t\le n$; $a_i, b_i, c_i$ and $d_j$ are elements of $k$ for all $i=1,\dotsc, s$ and $j=1,\dotsc, t$; and $b_i\ne 0$ for all $i=1,\dotsc, s$.
Using the hypothesis on $k$, it is easy to find an automorphism of $V$ that preserves the subspaces
$\langle u_i, v_i \rangle$ and $\langle w_j\rangle$ of $V^\vee$ for all $i$ and $j$, and that sends 
\begin{equation*}
a_i u_i^2 + b_i u_i v_i + c_i v_i^2 \mapsto u_i v_i
\qquad\text{and}\qquad
d_j w_j^2 \mapsto w_j^2
\end{equation*}
for all $i$ and $j$.
This implies the result, since $\sum_{j=1}^t w_j^2 = (\sum_{j=1}^t w_j)^2$. 
\end{proof}

Let $x=(x_1,\dotsc,x_n)$ be a finite set of indeterminates.
Let $f\in k[[x]]$ be a power series. Suppose that $f\in \langle x\rangle^2$ and that the Hessian matrix of $f$ has rank $r$ at the origin. Then $r$ is even if $p=2$.
Let
\begin{equation*}
q=\begin{cases}
x_1^2 + \dotsb + x_r^2 & \text{if $p\ne 2$}\\
x_1x_2 + \dotsb + x_{r-1}x_r & \text{if $p=2$}.
\end{cases}
\end{equation*}

\begin{lemma}
\label{morse-approximation}
If $p\ne 2$, then there exists a local $k$-algebra automorphism $\phi:k[[x]]\to k[[x]]$ such that $\phi(f)\equiv q$ modulo $\langle x\rangle^3$.
If $p= 2$, then there exists a local $k$-algebra  automorphism $\phi:k[[x]]\to k[[x]]$ such that either $\phi(f)\equiv q$ or $\phi(f)\equiv q+x_{r+1}^2$ modulo $\langle x\rangle^3$.
\end{lemma}

\begin{proof}
Let $\frak n$ denote the maximal ideal $\langle x\rangle \subset k[[x]]$. Let $q(f)$ denote the image of $f\in \frak n^2$ inside $\frak n^2/\frak n^3 = \Sym^2 (\frak n/\frak n^2)$. Then $q(f)$ is a quadratic form whose associated bilinear form is represented by the Hessian matrix of $f$ at the origin. 
Therefore, by Lemma \ref{classification-of-QFs}, there exists a $k$-linear automorphism $\phi_1$ of $\frak n/\frak n^2$ such that $\Sym^2(\phi_1)$ sends $q(f)$ to either $q$ or $q+x_{r+1}^2$.
We may take $\phi$ to be the linear automorphism of $k[[x]]$ corresponding to $\phi_1$, which characterized by the following property: for all $i$, the self-map of $\frak n^i/\frak n^{i+1} = \Sym^i(\frak n/\frak n^2)$ induced by $\phi$ is equal to $\Sym^i(\phi_1)$.
\end{proof}

\begin{proposition}[Morse's Lemma]
\label{morse-lemma}
If $r=n$, then there exists an automorphism of $k[[x]]$ as a local $k$-algebra that maps $f$ to $q$.
\end{proposition}

\begin{proof}
Because $r=n$, we have $\langle x\rangle \subseteq \operatorname{jac}(q)$. It follows from Proposition \ref{DeterminacyBound} that $q$ is 2-determined. Hence it suffices to show that there exists an automorphism of $k[[x]]$ as local $k$-algebra that sends $f$ to $q$ modulo $\langle x\rangle^3$. This follows from Lemma \ref{morse-approximation} above.
\end{proof}

\begin{proposition}
\label{morse-parameters}
Let $R$ be a complete local $k$-algebra with residue field $k$.
Let $F\in R[[x]]$ be a power series with residue $f$ in $k[[x]]$.
Then there exist a power series $h\in R[[x_{r+1},\dotsc,x_n]]$ and an automorphism of $R[[x]]$ as a local $R$-algebra that sends $F$ to $q+h$.
\end{proposition}

Proposition \ref{morse-parameters} belongs to a class of results that are usually referred to as ``Morse's Lemma with Parameters'' or ``the Splitting Lemma'' in the literature.
Note that the power series $h$ does not involve the variables occurring in $q$.

\begin{proof}
By Lemma \ref{morse-approximation}, there exists a local $k$-algebra automorphism of $k[[x]]$ that maps $f$ to either $q$ or $q + x_{r+1}^2$ modulo $\langle x\rangle^2$. 
After lifting such an automorphism to a local $R$-algebra automorphism of $R[[x]]$, we may assume that $f$ is congruent to either $q$ or $q + x_{r+1}^2$ modulo $\langle x\rangle^2$.

Let $R'$ denote the complete local $k$-algebra $R[[x_{r+1},\dotsc,x_n]]$.
Let $\bar f$ denote the image of $f$ under the map $k[[x_1,\dotsc,x_n]]\to k[[x_1,\dotsc,x_r]]$ that sends $x_i\mapsto x_i$ for $i\le r$ and $x_i\mapsto 0$ for $i>r$.
After replacing $R$ by $R'$ and $f$ by $\bar f$, we may assume that $r=n$ and $f\equiv q$ modulo $\langle x\rangle^3$.

By Morse's Lemma, there exists a local $k$-algebra automorphism of $k[[x_1,\dotsc,x_r]]$ that sends $f$ to $q$. After lifting such an automorphism to a local $R$-algebra automorphism of $R[[x]]$, we may assume that $f=q$. Put differently, we may assume that $F$ is a unfolding of $q$ over $R$.

By Example \ref{morse-versal} and the assumption that $r=n$, the power series
\begin{equation*}
q + t\in k[[t,x]],
\end{equation*}
is a versal unfolding of $q$ over $k[[t]]$.
We may therefore find a map of local $k$-algebras $a : k[[t]]\to R$ and an isomorphism of unfoldings $\phi : q+ a(t) \to F$ of $q$ over $R$. The element $h := a(t)\in R$ and the automorphism of $R[[x]]$ underlying $\phi$ satisfy the conclusions of the proposition.
\end{proof}

\chapter{Singularities of generic sections}

\section{Principal bundles}

\begin{definition}
\label{principal-bundle}
Let $G$ be a group scheme over $X$. A \textbf{principal $G$-bundle} over $X$ is a scheme $P$ over $X$, equipped with a a $G$-action $\rho: G\times_X P \to P$, such that
\begin{enumerate}
\item the map 
\begin{equation*}
(\rho,\mathrm{pr}_2): G\times_X P\to P\times_X P
\end{equation*}
is an isomorphism; and
\item there exists a cover of $X$ by Zariski open subsets, $\{U_i\subseteq X\}_{i\in I}$ such that the second projection $P_{U_i}:= P\times_X U_i \to U_i$ admits a section for all $i$.
\end{enumerate}
A \textbf{trivialization} of a principal $G$-bundle $P$ over $X$ is a $G$-equivariant isomorphism of $X$-schemes $P\xrightarrow\sim G$. A principal $G$-bundle equipped with a trivialization is \textbf{trivial}. 
\end{definition}

\begin{remark}
If the structure morphism $P\to X$ has a section, then pullback of $(\rho,\mathrm{pr_2})$ along this section defines a trivialization of $P$. Conversely, a trivial principal $G$-bundle has a canonical section corresponding to the identity morphism $e:X\to G$. Thus, according to the Definition \ref{principal-bundle}, a principal $G$-bundle admits trivializations locally in the Zariski topology.
\end{remark}

\begin{remark}
While Definition \ref{principal-bundle} is suitable for the purposes of the present work, one often finds in the literature definitions that are less stringent in that they only require local triviality in the \'etale, fppf or fpqc topologies.
\end{remark}

\begin{example}
\label{ses-vb}
Let $X$ be a scheme. Let
\begin{equation*}
\begin{tikzcd}
0 \ar{r} & A \ar["\alpha"]{r}
& B \ar["\beta"]{r} & \ar{r} C \ar{r} & 0
\end{tikzcd}
\end{equation*}
be a short exact sequence of locally free sheaves of finite rank on $X$.
Let $T$ be as scheme over $X$. Let $T\to \mathbf V(C)$ be a morphism over $X$ corresponding to a section $s\in \Gamma(T,C_T)$. Let
\begin{equation*}
T':= T\times_{s,\mathbf V(C),\beta} \mathbf V(B).
\end{equation*}

Let $Z$ be a scheme over $T$. Note the natural bijection
\begin{equation*}
\Hom_T(Z,T') = \{s'\in \Gamma(Z,B_Z)~|~\beta(s')=s_Z\}
\end{equation*}
and define an action
\begin{equation*}
\rho_Z:\Gamma(Z,A_Z)\times \Hom_T(Z,T') \to \Hom_T(Z,T')
\end{equation*}
of the additive group $\Gamma(Z,A_Z)$ by $\rho_Z(t,s')=t+s'$.

The maps $\rho_Z$ correspond (by Yoneda's Lemma) to an action of the $T$-group scheme $\mathbf V(A)_T$ on $T'$ which gives $T'$ the structure of a principal $\mathbf V(A)_T$-bundle over $T$.
\end{example}

\begin{proposition}
\label{vb-surj}
Let $X$ be a scheme. Let $\beta : B\to C$ be a map of locally free sheaves of finite rank on $X$. 
Let $\mathbf V(\beta) : \mathbf V(B)\to \mathbf V(C)$ be the corresponding morphism of vector bundles over $X$.
Suppose that $\beta$ is surjective.
Let $b$ and $c$ denote the respective ranks of $B$ and $C$.
Then $\mathbf V(\beta)$ is surjective and smooth of relative dimension $b-c$.
\end{proposition}

\begin{proof}
Let $A$ denote the kernel of $\beta$. Then $A$ is a locally free $\O_X$-module of rank $b-c$. 
By Example \ref{ses-vb}, the $\mathbf V(C)$-scheme $\mathbf V(B)$ is naturally a principal $\mathbf V(A)_{\mathbf V(C)}$-bundle. The result follows as $\mathbf V(A)\to X$ is surjective and smooth of relative dimension $b-c$.
\end{proof}

\begin{remark}
If $k$ is an infinite field, any nonempty subset of a positive-dimensional affine space over $k$ contains (infinitely many) $k$-rational points. This observation renders the notion of a general point of a finite dimensional vector space over $k$ meaningful.
\end{remark}

\begin{proposition}[Serre]
Let $k$ be an infinite field.
Let $X$ be a scheme of finite type over $k$.
Let $E$ be a locally free sheaf of rank $e$ on $X$.
Let $\Sigma\subset \mathbf V(E)$ be a locally closed subscheme of pure dimension $d$.
Let $W\subset \Gamma(X,E)$ be a $k$-linear subspace of finite dimension.
Suppose that $W$ generates $E$ as an $\O_X$-module.
Let $s\in W$ be a general section. 
Then every irreducible component of the locally closed subscheme $s^{-1}\Sigma\subseteq X$ has dimension $d-e$.
Furthermore, if $\Sigma$ is smooth and $k$ has characteristic zero, then $s^{-1}\Sigma$ is smooth over $k$.
\label{serre-bertini}
\end{proposition}

\begin{proof}
By assumption, the natural map $\alpha:W\otimes_k \O_X\to E$ is surjective. Let $N$ denote the dimension of $W$ over $k$. Then the morphism $\alpha:W\times_k X\to \mathbf V(E)$ is surjective and smooth of relative dimension $N-e$, see Proposition \ref{vb-surj}. It follows that $\alpha^{-1}\Sigma$ has pure dimension $d+N-e$ and is smooth over $k$ if $\Sigma$ is. Let $s\in W$ be a $k$-rational point viewed as a map $\Spec(k)\to W$. Then $s^{-1}\Sigma$ is the pullback of $\alpha^{-1}\Sigma$ along $s$. The proposition is trivial if the first projection $\alpha^{-1}\Sigma\to W$ is not dominant, so we may assume that it is.
Now the result follows from well known theorems on generic flatness or, in characteristic zero, generic smoothness applied to the morphism $\alpha^{-1}\Sigma\to W$.
\end{proof}

\section{Sheaves of principal parts}

In this section we collect definitions and basic facts about sheaves of principal parts.
The reader is referred to \cite[Chapter IV]{EGA} for proofs.

Let $S$ be a scheme. Let $X$ be a scheme over $S$. In the sequel we will abuse our notation by denoting both the structure sheaf of $S$ and its inverse image under the natural map $X\to S$ by $\O_S$.

\begin{definition}
\label{def-DO}
Let $E$ and $F$ be $\O_X$-modules.
Given an open subset $U\subset X$ and elements $D\in \Hom_{\O_S}(E_U,F_U)$ and $f\in \Gamma(U,\O_X)$, let
$
[D,f]:E_U\to F_U
$
be the $\O_S$-linear map that sends $s\mapsto D(fs)-f D(s)$.

Let $i\ge 1$ be a positive integer.
An $\O_S$-linear \textbf{differential operator} of \textbf{order} $i$ on $X$ is an $\O_S$-linear map $D:E\to F$ with the property that, if $U\subseteq X$ is an open subset and $f\in \Gamma(U,\O_X)$, then $[D,f]:E_U\to F_U$ is an $\O_S$-linear differential operator of order $i-1$ on $U$; by definition, an $\O_S$-linear differential operator of order 0 is an $\O_X$-linear map $E\to F$.
\end{definition}

\begin{example}
\label{connection-DO}
Suppose that $X$ is smooth over $S$. Let $E$ be a quasi-coherent sheaf on $X$. Let $\nabla : E\to \Omega_X\otimes E$ be a connection on $E$. Then $\nabla$ is a differential operator of order 1. To see this, let $U\subseteq X$ be an open subset and $f\in \Gamma(U,\O_X)$ a section. The commutator $[\nabla, f]: E_U\to (\Omega_X\otimes E)_U$ is given by
\begin{equation*}
s\mapsto [\nabla ,f](s) = \nabla(fs) - f \nabla s = df\otimes s,
\end{equation*}
hence is $\O_U$-linear, that is, a differential operator of order 0 on $U$.
\end{example}

\begin{remark}
Let $E$ and $F$ be $\O_X$-modules.
For each nonnegative integer $i$, let $\operatorname{Diff}_{\O_S}^i(E,F)$ denote the set of $\O_S$-linear differential operators with source $E$, target $F$, and order $i$.
Then
\begin{equation*}
\Hom_X(E,F) = \operatorname{Diff}_{\O_S}^0(E,F)
\subseteq 
\dotsb \subseteq
\operatorname{Diff}_{\O_S}^{i}(E,F) \subseteq
\operatorname{Diff}_{\O_S}^{i+1}(E,F) \subseteq
\dotsb 
\end{equation*}
as subsets of $\Hom_{\O_S}(E,F)$.
For example, an $\O_X$-linear map $\alpha : E\to F$ is a first-order differential operator because $[\alpha,f]=0$ for all $f\in \O_X$.
\end{remark}

\begin{proposition}
Let $E$, $F$ and $G$ be $\O_X$-modules.
Let $D:E\to F$ and $D': F\to G$ be $\O_S$-linear differential operators of orders $i$ and $j$.
Then the composition $D'\circ D : E\to G$ is a differential operator of order $i+j$.
\end{proposition}

\begin{proposition}
\label{univ-DO}
Let $E$ be a sheaf of $\O_X$-modules.
There exist a sheaf of $\O_X$-modules $\PP_X^i E$ and an $\O_S$-linear differential operator of order $i$ 
\begin{equation*}
d_E^i : E\to \PP_X^i(E)
\end{equation*}
with the following \textbf{universal property}:
Let $F$ be a sheaf of $\O_X$-modules and $D: E\to F$ is an $\O_S$-linear differential operator of order $i$. Then there exists a unique $\O_X$-linear map $\overline D : \PP_X^i E\to F$ such that $\overline D \circ d_E^i  = D$.
\end{proposition}

\begin{definition}
For each nonnegative integer $i$ and each sheaf of $\O_X$-modules $E$, we fix a choice of universal $\O_S$-linear differential operator $d_E^i : E\to \PP_X^i(E)$ of order $i$, as in Proposition \ref{univ-DO},
and call the $\O_X$-module $\PP_X^i E$ the \textbf{sheaf of principal parts} of \textbf{order} $i$ of $E$. We often write $\PP_X^i$ instead of $\PP_X^i(\O_X)$ and $d_X^i$ instead of $d_{\O_X}^i$. 
\end{definition}

Let $E$ and $F$ be a sheaves of $\O_X$-modules.
Let $i$ be a nonnegative integer.

\begin{remark}
\label{surj-pp}
Let $j$ be a nonnegative integers such that $j\le i$.
Differential operators of order $j$ are also differential operators of order $i$, so there exists a unique $\O_X$-linear map 
\begin{equation*}
\epsilon_{i,j} : \PP_X^i E \to \PP_X^j E
\end{equation*}
such that $\epsilon_{i,j}\circ d_X^i  = d_X^j$, by Proposition \ref{univ-DO}.
Setting $j=0$, we obtain a surjection
\begin{equation*}
\epsilon_{i,0} : \PP_X^i E\twoheadrightarrow \PP_X^0 E = E
\end{equation*}
that sends $d_E^i s \mapsto s$ for all $s\in E$.
\end{remark}

\begin{remark}
\label{higher-univ-DO}
Let $D: E\to F$ be an $\O_S$-linear differential operator of order $i$.
Then $D$ induces an $\O_X$-linear map $\overline D : \PP_X^i E\to F$, by Proposition \ref{univ-DO}.
More generally, for each nonnegative integer $m$, there exists a unique $\O_X$-linear map $\overline D : \PP_X^{i+m} E\to \PP_X^m F$ such that the diagram
\begin{equation*}
\begin{tikzcd}
E \ar[r,"D"] \ar[d,"d_E^{i+m}"] &
F \ar[d,"d_F^m"] \\
\PP_X^{i+m} E \ar[r,"\overline D"] &
\PP_X^m F
\end{tikzcd}
\end{equation*}
commutes, by Proposition \ref{univ-DO} applied to $d_F^m\circ D$, a differential operator of order $i+m$.
\end{remark}

\begin{proposition}
Let $P \subseteq \PP_X^i E$ be the $\O_X$-submodule generated by the image of the universal differential operator $d_E^i : E\to \PP_X^i E$.
Then $P =\PP_X^i E$.
\end{proposition}

\begin{proposition}
\label{jet-map}
Suppose that $S$ is the spectrum of a field $k$.
Let $x\in X$ be a $k$-rational point.
Then the universal $k$-linear differential operator $d_E^i : E\to \PP_X^i E$ induces a $k$-linear isomorphism
\begin{equation*}
E/\frak m_x^{i+1} E\xrightarrow\sim \PP_X^i E \otimes_{\O_X} \kappa(x)
\end{equation*}
that sends $\bar s\mapsto d_E^i s \otimes 1$ for all local sections $s\in E$ defined around $x$.
\end{proposition}

Motivated by Proposition \ref{jet-map}, we make the following definition:

\begin{definition}
Let $x :T\to X$ be a morphism of schemes.
An \textbf{order}-$i$ \textbf{jet of a section} of $E$ at $x$ is an element of $\Gamma(T,x^* \PP_X^i E)$. 
\end{definition}

Propositions \ref{algebra-pp}-\ref{smooth-pp} below describe fundamental properties of the $\O_X$-modules $\PP_X^i E$.
When $S$ is the spectrum of a field $k$, each property reduces to a familiar fact about the $k$-vector spaces $E/\frak m_x^{i+1}E$ after tensor product with the residue field of a $k$-rational point $x\in X$. Proposition \ref{algebra-pp}, for example, reflects the fact that $\O_X/\frak m_x^{i+1}$ is not merely a $k$-vector space, but naturally a $k$-algebra.

\begin{proposition}
\label{algebra-pp}
There exists a unique $\O_X$-linear map
$
\PP_X^i \otimes_{\O_X} \PP_X^i \to \PP_X^i
$
giving $\PP_X^i$ the structure of an $\O_X$-algebra such that \begin{equation*}
d_X^i f \cdot d_X^i g = d_X^i(fg)
\end{equation*}
for all local sections $f,g\in \O_X$ defined over a common open subset of $X$.
\end{proposition}

\begin{remark}
The $\O_X$-algebra structure of Proposition \ref{algebra-pp} is commutative with unit $d_X^i(1)\in \Gamma(X,\PP_X^i)$.
By abuse of notation, we identify $f\in \O_X$ with the product $f\cdot d_X^i(1)\in \PP_X^i$.
\end{remark}

\begin{proposition}
The natural $\O_X$-linear map $\epsilon_{i,0} : \PP_X^i \to \PP_X^0 = \O_X$ is a map of $\O_X$-algebras.
Let $I$ denote its kernel.
Then $I$ is an ideal sheaf generated over $\PP_X^i$ by sections of the form $d_X^i f - f$ with $f\in \O_X$, and $I^{i+1}=0$.
\end{proposition}

\begin{proposition}
There exists a unique $\O_X$-linear map $\PP_X^i \otimes_{\O_X}\PP_X^i E\to \PP_X^i E$ giving $\PP_X^iE$ the structure of an $\PP_X^i$-module such that
\begin{equation*}
d_X^i f\cdot d_E^i s = d_E^i (fs)
\end{equation*}
for all local sections $f\in \O_X$ and $s\in E$ defined over a common open subset of $X$.
\end{proposition}

\begin{proposition}
\label{pp-module}
There exists a unique isomorphism of $\PP_X^i$-modules 
\begin{equation*}
\PP_X^i \otimes_{d_X^i, \O_X} E \xrightarrow\sim \PP_X^i E
\end{equation*}
that sends $\alpha\otimes s \mapsto \alpha \cdot d_E^i s$ for all local sections $\alpha \in \PP_X^i$ and $s\in E$ defined over a common open subset of $X$. 
\end{proposition}

\begin{proposition}
\label{pp-hom}
There exists an unique $\PP_X^i$-linear map
\begin{equation}
\PP_X^i \sheafHom_{\O_X}(E,F) \to \sheafHom_{\PP_X^i}(\PP_X^i E, \PP_X^i F)
\label{hom-pp}
\end{equation}
such that
\begin{equation*}
(d_{\sheafHom(E,F)}^i \alpha) ( d_E^i s )  = d_F^i (\alpha s)
\end{equation*}
for all local sections $\alpha \in \sheafHom(E,F)$ and $s\in E$ defined over a common open subset of $X$.
\end{proposition}

\begin{remark}
By Proposition \ref{pp-hom}, the construction $E\mapsto \PP_X^i E$ extends naturally to a functor from the category of $\O_X$-modules to that of $\PP_X^i$-modules.
\end{remark}

\begin{proposition}
Suppose that $i\ge 1$.
There exists a unique $\O_X$-linear map $\delta_i : \Sym^i(\Omega_X)\otimes E \to \PP_X^i E$ such that 
\begin{equation*}
\delta_i(df_1\dotsb df_i \otimes s) = 
(d_X^1 f_1-f_1)\dotsb
(d_X^1 f_i-f_i)\cdot d_E^i s
\end{equation*}
for all local sections $f_1,\dotsc,f_i\in \O_X$ and $s\in E$ defined over a common open subset of $X$.
The sequence of $\O_X$-modules
\begin{equation}
\begin{tikzcd}
0 \ar[r] & 
\Sym^i(\Omega_X)\otimes_{\O_X} E \ar[r,"\delta_i"] &
\PP_X^i E \ar[r,"\epsilon_{i,i-1}"] &
\PP_X^{i-1} E \ar[r] &
0
\end{tikzcd}
\label{pp-exact-seq}
\end{equation}
is exact if $i=1$ and right-exact in general.
\end{proposition}

\begin{proposition}
\label{smooth-pp}
Suppose that $X$ is smooth over $S$. Then
\begin{enumerate}
\item the sequence (\ref{pp-exact-seq}) is exact if $i \ge 1$;
\item the sheaf $\PP_X^i$ is a locally free $\O_X$-module of finite rank; and
\item the map (\ref{hom-pp}) is bijective for all $\O_X$-modules $F$.
\end{enumerate}
\end{proposition}

\begin{example}
Suppose that $X$ is \'etale over $\A^n_S$, that is, suppose that $X$ is smooth over $S$ and that there exist sections $x_1,\dotsc,x_n\in \Gamma(X,\O_X)$ whose differentials form a basis of $\Omega_X$ as an $\O_X$-module.
Let $\epsilon_1,\dotsc,\epsilon_n$ be indeterminates. Let $\mathscr A$ denote the sheaf of $\O_X$-algebras $\O_X[\epsilon_1,\dotsc,\epsilon_n]$. Let $\frak M\subset \mathscr A$ denote the sheaf of ideals generated by $\epsilon_1,\dotsc,\epsilon_n$. Then the map of $\O_X$-algebras
\begin{equation*}
\O_X[\epsilon_1,\dotsc,\epsilon_n]\to \PP_X^i
\end{equation*}
that sends $\epsilon_a \mapsto d_X^i x_a - x_a$ for all $a=1,\dotsc, n$ induces a map of short exact sequences of $\O_X$-modules:
\begin{equation*}
\begin{tikzcd}
0 \ar{r} 
& \frak M^i/\frak M^{i+1} \ar{r} \ar{d} 
& \mathscr A/\frak M^{i+1} \ar{r} \ar{d}
& \mathscr A/\frak M^{i} \ar{r} \ar{d}
& 0 \\
0 \ar{r}
& \Sym^i \Omega_X \ar{r}
& \PP_X^i \ar{r}
& \PP_X^{i-1} \ar{r}
& 0
\end{tikzcd}
\end{equation*}
Combining the five lemma and induction on $i$, we see that this map of short exact sequences is an isomorphism. In particular, $\PP_X^i$ is a free $\O_X$-module with basis given by the monomials of degree at most $i$ in $\epsilon_1,\dotsc,\epsilon_n$.
\end{example}

\begin{proposition}
\label{1st-order-pp-ses}
Suppose that $X$ is smooth over $S$.
Suppose that the $\O_X$-module $E$ is quasi-coherent.
There exists a natural bijection between the set of connections on $E$ and the set of splittings of the short exact sequence
\begin{equation}
\label{1st-pp-ses-display}
\begin{tikzcd}
0\ar[r] &
\Omega_X\otimes E \ar[r,"\delta_1"] &
\PP_X^1 E \ar[r,"\epsilon_{1,0}"] &
E \ar[r] &
0
\end{tikzcd}
\end{equation}
obtained by setting $i=1$ in (\ref{pp-exact-seq}). If $\nabla : E\to \Omega_X\otimes E$ is a connection on $E$, the corresponding splitting of (\ref{1st-pp-ses-display}) is given by the unique $\O_X$-linear map $\overline \nabla : \PP_X^1 E\to \Omega_X\otimes E$ such that $\overline\nabla \circ d_E^1 = \nabla$, see Proposition \ref{univ-DO}.
\end{proposition}

\begin{proof}
Let $\overline \nabla : \PP_X^1 E\to \Omega_X\otimes E$ be an $\O_X$-linear map. Let $\nabla : E\to \Omega_X\otimes E$ be the corresponding differential operator of order 1, so that $\nabla = \overline \nabla\circ d_E^1$.
Let $f\in \O_X$ and $s\in E$ be local sections defined over a common open subset of $X$.
We compute:
\begin{align*}
\overline\nabla \circ \delta_1 (df\otimes s)
&=\overline \nabla ((d_X^1 f-f)\cdot d_E^1 s)\\
&= \overline \nabla (d_X^1 f \cdot d_E^1 s) -f\overline \nabla d_E^1 s \\
&= \overline \nabla d_E^1(fs) -f\overline \nabla d_E^1 s \\
&= \nabla (fs) - f\nabla s
\end{align*}
Thus the $\O_X$-linear map $\overline \nabla$ satisfies $\overline \nabla \circ \delta_1 = \mathrm{id}$, that is, defines a splitting of (\ref{1st-order-pp-ses}), if, and only if, the differential operator $\nabla$ is a connection.
\end{proof}

We write $\PP_{X/S}^i E$ instead of $\PP_X^i E$ when we wish to emphasize the base scheme $S$.

\begin{remark}
\label{pp-base-change-map}
Let
\begin{equation*}
\begin{tikzcd}
X' \ar[r, "u"] \ar[d]
& X \ar[d]\\
S' \ar[r] & S
\end{tikzcd}
\end{equation*}
be a commutative diagram of schemes. Then there exists a canonical $\O_{X'}$-linear map
\begin{equation*}
u^* \PP_{X/S}^iE \to \PP_{X'/S'}^i (u^* E).
\end{equation*}
This map is the adjoint of the $\O_X$-linear map $\PP_{X/S}^i E\to u_* \PP_{X'/S'}$ induced the $\O_S$-linear differential operator between $\O_X$-modules
\begin{equation*}
\begin{tikzcd}[column sep = large]
E \ar[r] &
u_* u^* E \ar[r,"u_* (d_{u^* E}^i)"] &
u_* \PP_{X'/S'}^i (u^* E).
\end{tikzcd}
\end{equation*}
\end{remark}

\begin{proposition}
\label{pp-base-change}
Suppose the diagram in Remark \ref{pp-base-change-map} is Cartesian, so that $X' = X\times_S S'$. Then the canonical $\O_{X'}$-linear map 
\begin{equation*}
u^* \PP_{X/S}^iE \to \PP_{X'/S'}^i (u^* E).
\end{equation*}
is an isomorphism.
\end{proposition}

The next result will be applied in the sequel in combination with Proposition \ref{serre-bertini} above.

\begin{proposition}
\label{sufficient-jets}
Suppose that $S$ is the spectrum of a field $k$.
Let $W\subset \Gamma(X,E)$ be a $k$-linear subspace of finite dimension.
Let $\bar k$ be an algebraic closure of $k$.
Then $d_E^i(W)$ generates $\PP_X^i E$ as an $\O_X$-module if, and only if, the natural map 
\begin{equation}
\label{suff-jets-eq}
W\otimes_k \bar k \to E_{X_{\bar k}}/\frak m_x^{i+1} E_{X_{\bar k}}
\end{equation}
is surjective for every closed point $x\in X_{\bar k}$.
\end{proposition}

\begin{example}
Suppose that $E$ is an invertible sheaf. Then $d_X^1(W)$ generates $\PP_X^1 E$ if, and only if, the linear system $(E,W)$ separates tangent vectors on $X$.
\end{example}

\begin{proof}[Proof of Proposition \ref{sufficient-jets}]
Let $u : X_{\bar k}\to X$ be the projection.
The natural $\O_X$-linear maps
\begin{equation*}
\begin{tikzcd}
d_E^i(W) \otimes_k \O_X \ar[r,"\alpha"] &
\PP_{X/k}^i E \ar[r] &
u_* \PP_{X_{\bar k}/\bar k}^i (u^* E)
\end{tikzcd}
\end{equation*} 
correspond under the adjunction between $u^*$ and $u_*$ to $\O_{X_{\bar k}}$-linear maps
\begin{equation*}
\begin{tikzcd}
d_E^i(W) \otimes_k \O_{X_{\bar k}} \ar[r,"u^* \alpha"] &
u^* \PP_{X/k}^i E \ar[r,"\sim"] &
\PP_{X_{\bar k}/\bar k}^i (u^* E),
\end{tikzcd}
\end{equation*} 
the second of which is an isomorphism by Proposition \ref{pp-base-change}. 
The projection $u : X_{\bar k}\to X$ is faithfully flat, so the $\O_X$-linear map $\alpha$ is surjective if, and only if, its pullback $u^* \alpha$ is surjective.
By Proposition \ref{jet-map}, the latter holds if, and only if, the map (\ref{suff-jets-eq}) is surjective for all closed points $x\in X_{\bar k}$.
\end{proof}

\section{First-order singularities of generic sections}

Let $S$ be a scheme. Let $X$ be a smooth scheme over $S$.
Let $E$ be a locally free sheaf of finite rank on $X$.
Let $\nabla : E\to \Omega_X\otimes E$ be a connection on $E$.
Let $\overline \nabla : \PP_X^1 E\to \Omega_X\otimes E$ be the unique $\O_X$-linear map such that $\overline\nabla \circ d_E^1 = \nabla$, see Proposition \ref{univ-DO}.

\begin{definition}
Let $x: T\to X$ be a morphism of schemes.
Let $s\in \Gamma(T,x^* \PP_X^1 E)$ be a section.
Let $i\ge 0$ be a nonnegative integer.
The \textbf{$i$th critical locus} of $s$ is the $i$th degeneracy locus of the $\O_T$-linear map $\overline \nabla s : (\T_X)_T\to E_T$.
It is a subscheme of $T$ that we denote by $\Sigma^i(s)$.
\end{definition}

Let $J := \mathbf V(\PP_X^1 E)$ be the vector bundle corresponding the sheaf of principal parts $\PP_X^1 E$, which is locally free. 
Let $h\in \Gamma(J, (\PP_X^1 E)_J)$ be the tautological section.
Let $n$ denote the relative dimension of $X$ over $S$.
Let $e$ denote the rank of $E$.
Let $i$ be a nonnegative integer such that $i\le \min(n,e)$.

\begin{proposition}
\label{univ-1st-sings}
The critical locus $\Sigma^i(h)\setminus \Sigma^{i+1}(h)\subseteq J$ is smooth over $X$, of pure relative codimension $i(|n-e|+i)$ in $J$ over $X$.
\end{proposition}

\begin{proof}
Let $H := \mathbf V(\sheafHom_X(\T_X,E))$ be the vector bundle corresponding to the locally free $\O_X$-module $\sheafHom_X(\T_X,E)$.
The $\O_X$-linear map $\overline \nabla : \PP_X^1 E\to \Omega_X\otimes E$ induces a morphism of schemes $J\to H$ over $X$, which we also denote by $\overline\nabla$.
Let $\alpha : (\T_X)_H\to E_H$ be the tautological map over $H$.
Then, for each nonnegative integer $j$,
\begin{equation*}
\overline \nabla^{-1}(\Sigma^j(\alpha)) =
\Sigma^j( \overline \nabla^* \alpha ) =
\Sigma^j( \overline \nabla h) =
\Sigma^j( h ).
\end{equation*}
The morphism of schemes $\overline\nabla : J\to H$ is smooth and surjective because it is induced by a surjective map between $\O_X$-modules, see Proposition \ref{vb-surj} and Proposition \ref{1st-order-pp-ses}. The result therefore follows from Proposition \ref{univ-deg-loci}, according to which the degeneracy locus $\Sigma^i(\alpha)\setminus \Sigma^{i+1}(\alpha)$ is smooth over $X$, of relative codimension $i(|n-e|+i)$ in $H$.
\end{proof}

Suppose that $S$ is the spectrum of an infinite field $k$.

\begin{corollary}
\label{1st-sings-generic}
Let $W\subset \Gamma(X,E)$ be a $k$-linear subspace of finite dimension.
Suppose that $d_E^1(W)\subseteq \Gamma(X,\PP_X^1 E)$ generates $\PP_X^1 E$ as an $\O_X$-module.
Let $s\in W$ be a general section.
Then every irreducible component of the critical locus $\Sigma^i(s)\setminus \Sigma^{i+1}(s)$ has codimension $i(|n-e|+i)$ in $X$.
Furthermore, if $k$ has characteristic zero, then $\Sigma^i(s)\setminus \Sigma^{i+1}(s)$ is smooth.
\end{corollary}

\begin{proof}
The map of sheaves of $k$-vector spaces $d_E^1 : E\to \PP_X^1 E$ is a split injection, so $W\cong d_E^1(W)$ as vector spaces over $k$.
Let $d_E^1 s \in d_E^1(W)$ be a general section, which we view as a morphism $X\to \mathbf V(\PP_X^1 E) = J$. 
Let $h\in \Gamma(J, (\PP_X^1 E)_J)$ be the tautological section.
By Propositions \ref{serre-bertini} and \ref{univ-1st-sings}, every irreducible component of 
\begin{equation*}
(d_E^1 s)^{-1}( \Sigma^i(h)\setminus \Sigma^{i+1}(h) ) \subseteq X
\end{equation*} 
has codimension $n-i(|n-e|+i)$, and this subscheme of $X$ is smooth if $k$ has characteristic zero.
The result follows, since $(d_E^1 s)^{-1} ( \Sigma^j(h) ) = \Sigma^j(s)$ for every nonnegative integer $j$.
\end{proof}

\section{The differential of a jet}

Let $S$ be a scheme. Let $X$ be a scheme over $S$. Let $E$ be a locally free sheaf of finite rank on $X$. Let $V:= \mathbf V(E)$ be the corresponding vector bundle. Let $\pi: V\to X$ be the projection. 

\begin{remark}
The canonical isomorphism $\pi_* \O_V = \Sym (E^\vee)$ allows one to view sections of $\O_X$ and $E^\vee$ as sections of $\O_V$. The differentials of sections of $\O_V$ arrising in this way generate  $\Omega_V$ as an $\O_V$-module.

For example, if $E=\O_X$, then $V = X\times \A^1$. In this case, the sheaf of differentials $\Omega_V$ is generated as an $\O_X$-module by $\mathrm{pr}_1^*\Omega_X$ and the differential of the coordinate on $\A^1$.
\end{remark}

\begin{remark}
A section $s\in \Gamma(X,E)$ may be viewed as a map $X\to V$, which we also denote by $s$. It therefore has a co-differential
\begin{equation*}
ds^* : s^* \Omega_V\to \Omega_X,
\end{equation*}
which is the unique $\O_X$-linear map that sends $s^* df\mapsto df$ and $s^* dt\mapsto d(t\cdot s)$ for all local sections $f\in \O_X$ and $t\in E^\vee$.
\end{remark}

Let $x: T\to X$ be a morphism of schemes.
Let $s\in \Gamma(T,x^* \PP_X^1 E)$ be a section.
Let $s_0 \in \Gamma(T,x^* E)$ be the image of $s$ under the natural $\O_X$-linear map $\PP_X^1 E\to  \PP_X^0 E = E$, see Remark \ref{surj-pp}.
Recall that $s_0$ defines a morphism $T\to V$ over $X$, which we also denote by $s_0$.
We may think of $s$ as a first-order jet of a section of $E$ mapping the $T$-valued point $x\in X(T)$ to $s_0\in V(T)$, see Proposition \ref{jet-map}.

\begin{definition}
\label{diff-jet-def}
The \textbf{co-differential} of the first-oder jet $s\in \Gamma(T, x^*\PP_X^1 E)$ is the $\O_T$-linear map $ds^* : s_0^* \Omega_V  \to x^* \Omega_X$ of Proposition \ref{diff-jet-prop} below.
If $X$ is smooth over $S$, we call the dual $ds : x^* \T_X \to s_0^* \T_V$ of this $\O_T$-linear map the \textbf{differential} of $s$.
\end{definition}

\begin{remark}
Let $t\in \Gamma(U,E^\vee)$ be a section. 
Then
\begin{equation*}
d_{\sheafHom(E,\O)}^1 t \in 
\PP_X^1 \sheafHom_{\O_X}(E,\O_X) = \sheafHom_{\PP_X^1}(\PP_X^1 E, \PP_X^1),
\end{equation*}
so we have a natural product $(d_{\sheafHom(E,\O)}^1 t)\cdot s \in \Gamma(x^{-1} U,x^* \PP_X^1)$. To lighten the notation, we denote this product by $t\cdot s$. 
\end{remark}

Let $\overline d: \PP_X^1 \to \Omega_X$ be the $\O_X$-linear map induced by the universal derivation $d: \O_X \to \Omega_X$, see Proposition \ref{univ-DO}.

\begin{proposition}
\label{diff-jet-prop}
There exists a unique map of $\O_T$-modules
\begin{equation*}
d s^* : s_0^* \Omega_V \to x^* \Omega_X
\end{equation*}
that sends $s_0^* df \mapsto x^* df$ and $s_0^* dt\mapsto x^* (\bar d (t\cdot s))$ for all local sections $f\in \O_X$ and $t\in E^\vee$.
\end{proposition}

\begin{proof}
Uniqueness is clear, so it suffices to show existence Zariski-locally on $X$. Hence we may assume that $E$ is trivial and choose an $\O_X$-linear basis $v_1,\dotsc,v_e\in \Gamma(X,E)$. Then, first, $V=X\times \A^r$, where coordinates of the second factor are given by the dual basis,
\begin{equation*}
v_1^\vee,\dotsc,v_e^\vee \in E^\vee\subseteq \Sym (E^\vee) = \pi_* \O_V.
\end{equation*}
Second,
\begin{equation*}
\Omega_V = \pi^* \Omega_X \oplus \left(\oplus_{i=1}^e \O_V\cdot d(v_l^\vee)\right),
\end{equation*}
where $v_l^\vee$ is identified with a global section of $\O_V$. And third, the sheaf $\PP_X^1 E$ is a free $\PP_X^1$-module with basis $d_E^1 (v_1^\vee), \dotsc, d_E^1 (v_e^\vee)$. 

Write $s = \textstyle\sum_{l=1}^e s_l d^1_E v_l$ with $s_1,\dotsc,s_e\in \Gamma(T,x^* \PP_X^1)$, and define $ds^* : s_0^* \Omega_V \to x^* \Omega_X$ by requiring by that its restriction to $s_0^* \pi^* \Omega_X = x^* \Omega_X$ be the identity and that it send $s_0^* d(v_l^\vee)\mapsto \bar d s_l$ for $l=1,\dotsc, e$.

To see that this works, let $U\subseteq X$ be an open subset and $t\in \Gamma(U,E^\vee)$ be a section. Write $t=\sum_{l=1}^e t_l v_l^\vee$ with $t_l\in \Gamma(U,\O_X)$. Then
\begin{align*}
t\cdot s &= (d_{E^\vee}^1 t) \cdot \bigg( \sum_{l=1}^e s_l d_E^1 v_l \bigg) 
= \sum_{l=1}^e s_l (d_{E^\vee}^1 t \cdot d_E^1 v_l)\\ 
&= \sum_{l=1}^e s_l d_X^1 (t\cdot v_l) 
= \sum_{l=1}^e s_l d_X^1 t_l. 
\end{align*}
Let $(s_0)_1,\dotsc,(s_0)_e\in \Gamma(T,\O_T)$ be the images of $s_1,\dotsc, s_e \in \Gamma(T, x^* \PP_X^1)$ under the natural map $\PP_X^1 \to \O_X$, see Remark \ref{surj-pp}.
By the Leibniz rule, 
\begin{equation*}
\bar d(t\cdot s) = \sum_{l=1}^e t_l \bar d(s_l) + (s_0)_l dt_l 
\end{equation*}
in $\Gamma(x^{-1}U, x^* \Omega_X)$.
On the other hand, $s_0 = \sum_{l=1}^e (s_0)_l v_l$ in $\Gamma(T,x^* E)$, so the pullback of $dt\in \Gamma(\pi^{-1} U, \Omega_V)$ along $s_0 : T\to V$ satisfies
\begin{equation*}
s_0^* dt
= s_0 ^* \sum_{l=1}^e  t_l d (v_l^\vee) + v_l^\vee d t_l
= \sum_{l=1}^e t_l s_0^*d (v_l^\vee) + (s_0)_l d t_l
\end{equation*}
in $\Gamma(x^{-1}U,s_0^*\Omega_V)$. Thus $ds^* : s_0^* \Omega_V\to x^* \Omega_X$ sends $s_0^* dt\mapsto x^* (\bar d (t\cdot s))$, as required.
\end{proof}

Suppose that $X$ is smooth over $S$.
Let $\nabla : E\to \Omega_X\otimes E$ be a connection on $E$.
Let $\overline \nabla : \PP_X^1 E\to \Omega_X\otimes E$ be the $\O_X$-linear map induced by $\nabla$, see Proposition \ref{univ-DO}.

\begin{construction}
\label{horiz-sub-constr}
By Proposition \ref{1st-order-pp-ses}, the map $\overline \nabla$ defines a splitting of the canonical short exact sequence
\begin{equation*}
\begin{tikzcd}
0 \ar[r] &
\Omega_X\otimes E \ar[r,"\delta_1"] &
\PP_X^1 E \ar[r,"\epsilon_{1,0}"] &
E \ar[r] &
0.
\end{tikzcd}
\end{equation*}
In other words, the restriction of $\epsilon_{i,0}$ to $\ker(\overline \nabla)$ is an isomorphism of $\O_X$-modules $\epsilon :\ker(\overline\nabla) \xrightarrow\sim E$.
Let $h\in \Gamma(V,E_V)$ be the tautological section. Let $\tilde h$ denote the section $\epsilon^{-1} h
\in \Gamma(V, \ker(\overline\nabla)_V)$, which we regard as a first-order jet of a section of $E$ (at $\pi$).
Let
\begin{equation*}
d\tilde h : (\T_X)_V \to h^* \T_V = \T_V
\end{equation*}
be the differential of $\tilde h$.
Clearly, $d\tilde h$ is an injection split by the differential $d\pi : \T_V \to \pi^* \T_X$.
Let $H\subseteq \T_V$ denote the image of $d\tilde h$. We call $H$ the \textbf{horizontal subbundle} corresponding to the connection $\nabla$, see Definition \ref{horiz-sub-def}.
\end{construction}

The following proposition shows that horizontal subbundle determines the covariant derivative $\overline\nabla s$ of the first-order jet $s\in \Gamma(T, x^* \PP_X^1 E)$.

\begin{proposition}
\label{horiz-sub-prop-3}
Setup as in Construction \ref{horiz-sub-constr}.
The following diagram of solid arrows is commutative:
\begin{equation*}
\begin{tikzcd}
x^* \T_X \ar[r,"ds"] \ar[d, "\overline\nabla s"'] &
s_0^* \T_V \ar[r, two heads] &
s_0^*(\T_V/H) \\
x^* E \ar[r,equals] &
s_0^* \T_{V/X} \ar[u,hook,dashed] \ar[ru, "\sim"']
\end{tikzcd}
\end{equation*}
\end{proposition}

\begin{proof}
Dually, we wish to show that the following diagram of solid arrows is commutative:
\begin{equation*}
\begin{tikzcd}
s_0^*(\T_V/H)^\vee \ar[r,hook] \ar[rd,"\sim"'] & 
s_0^* \Omega_V \ar[r,"ds^\vee"] \ar[d, two heads, dashed] &
x^*\Omega_X\\
&
s_0^* \Omega_{V/X} &
s_0^* E^\vee \ar[l,equals] \ar[u,"\overline \nabla s"']
\end{tikzcd} 
\end{equation*}
Note that
\begin{equation*}
(\T_V/H)^\vee =
\coker( d\tilde h )^\vee =
\ker( (d \tilde h)^\vee )
\end{equation*}
as split $\O_V$-submodules of $\Omega_V$. Therefore, it suffices to show that the following diagram commutes:
\begin{equation*}
\begin{tikzcd}[column sep = huge]
s_0^* \Omega_V \ar[r,"ds^\vee - s_0^* (d\tilde h^\vee)"] \ar[d, two heads] &
x^* \Omega_X\\
s_0^* \Omega_{V/X} &
x^* E^\vee \ar[l, equals] \ar[u,"\overline \nabla s"']
\end{tikzcd} 
\end{equation*}

Consider the two maps
\begin{equation*}
ds^\vee, s_0^* (d\tilde h^*)  : s^* \Omega_V \to \Omega_X.
\end{equation*}
The sheaf $\Omega_V$ is generated by sections of the form $df$ with $f\in \O_X$ and $dt$ with $t\in E^\vee$. Both maps send $s_0^*df\mapsto x^* df$, whereas
\begin{align*}
(ds^\vee - s_0^* (d\tilde h^\vee)) (s_0^*dt) 
&= \bar d(t\cdot s) - \bar d(t\cdot s_0^* \tilde h) \\
&= \bar d(t\cdot (s - s_0^* \tilde h)).
\end{align*}
The difference $s - s_0^* \tilde h \in \Gamma(T,x^* \PP_X^1 E)$ maps to $ s_0-s_0^* h = 0 \in \Gamma(T,x^* E)$, so is contained in the image of the canonical inclusion $\delta_1 : \Omega_X\otimes E\hookrightarrow \PP_X^1 E$.
Therefore
\begin{equation*}
s - s_0^* \tilde h
= \delta_1 \overline \nabla ( s - s_0^* \tilde h )
=\delta_1 (\overline \nabla s - 0).
\end{equation*}
Using the fact that the $\O_X$-linear map $\overline d : \PP_X^1 \to \Omega_X$ splits the canonical inclusion $\Omega_X \hookrightarrow \PP_X^1$, we conclude that 
\begin{align*}
(ds^\vee - s_0^* (d\tilde h^\vee)) (s_0^*dt)
= \bar d(t\cdot \delta_1 \overline\nabla s)
= t\cdot \overline\nabla s,
\end{align*}
which completes the proof.
\end{proof}

\begin{corollary}[= Proposition \ref{horiz-sub-prop}]
\label{horiz-sub-prop-2}
Setup as in Definition \ref{horiz-sub-def}.
Let $s\in \Gamma(X,E)$ be a section.
Let $ds : \T_X\to s^* \T_V$ denote the differential of $s$ viewed as a morphism $X\to V$.
Then the following diagram of solid arrows is commutative:
\begin{equation*}
\begin{tikzcd}
\T_X \ar[r,"ds"] \ar[d, "\nabla s"'] &
s^* \T_V \ar[r, two heads] &
s^*(\T_V/H) \\
E \ar[r,equals] &
s^* \T_{V/X} \ar[u,hook,dashed] \ar[ru, "\sim"']
\end{tikzcd}
\end{equation*}
\end{corollary}

\begin{definition}
\label{nd-of-jet}
Let $\Sigma\subseteq V$ be a subscheme.
Suppose that $\Sigma$ is smooth over $S$.
The \textbf{normal differential} of $s\in \Gamma(T,x^* \PP_X^1 E)$ with respect to $\Sigma$, denoted $\mathbf d_\Sigma s$, is the composition of $\O_{s^{-1}_0\Sigma}$-linear maps
\begin{equation*}
\begin{tikzcd}
(\T_X)_{s_0^{-1}\Sigma} \ar[r, "ds"] & 
s_0^* (\T_V)_{\Sigma} \ar[r, two heads] &
s_0^* \N_\Sigma.
\end{tikzcd}
\end{equation*}
\end{definition}

\begin{proposition}
\label{nd-jet-prop}
Setup as in Definition \ref{nd-of-jet}.
Suppose that the tangent bundle $\T_\Sigma$ contains the restriction to $\Sigma$ of horizontal subbundle $H \subseteq \T_V$ corresponding to the connection $\nabla$.
Then the normal differential $\mathbf d_\Sigma s$ is equal to the composition
\begin{equation*}
\begin{tikzcd}
(\T_X)_{s_0^{-1} \Sigma} \ar[r,"\nabla s"] &
E_{s_0^{-1} \Sigma} = (\T_{V/X})_{s_0^{-1} \Sigma} \ar[r,hook]  &
(\T_V)_{s_0^{-1} \Sigma} \ar[r,two heads] &
s_0^* \N_\Sigma.
\end{tikzcd}
\end{equation*}
\end{proposition}

\begin{proof}
Follows immediately from Proposition \ref{horiz-sub-prop-3}.
\end{proof}

Let $A$ and $B$ be locally free sheaves of finite rank on $X$.
Let  $\alpha\in \Gamma(T,x^* \PP_X^1 \sheafHom_X(A,B))$ be a section. 
Let $\alpha_0 : A_T\to B_T$ be the image of $\alpha$ under the natural map
\begin{equation*}
\PP_X^1 \sheafHom_X(A,B)\to \sheafHom_X(A,B).
\end{equation*}
Let $i\ge 0$ be a nonnegative integer.
Let $\Sigma$ denote the degeneracy locus $\Sigma^i(\alpha_0)\setminus \Sigma^{i+1}(\alpha_0)\subseteq T$.

\begin{definition}
\label{id-of-jet}
The \textbf{intrinsic differential} of $\alpha$ on $\Sigma$ is the $\O_{\Sigma}$-linear map 
\begin{equation*}
\mathbf d_{\Sigma} \alpha : (\T_X)_{\Sigma} \to
\sheafHom_{\Sigma}(\ker((\alpha_0)_{\Sigma}), \coker( (\alpha_0)_{\Sigma} ))
\end{equation*}
defined as follows.

Let $\rho : H\to X$ denote the vector bundle associated to the locally free sheaf $\sheafHom_X(A,B)$. Let $h: A_H\to B_H$ be the tautological map, see Definition \ref{tautological-section}. 
Let $\Sigma'$ denote the universal degeneracy locus $\Sigma^i(h)\setminus \Sigma^{i+1}(h)\subseteq H$.
Identifying $\alpha_0$ with a morphism $T\to H$, we have $(\alpha_0)^* h = \alpha_0$ and $(\alpha_0)^{-1}\Sigma' = \Sigma$ as subschemes of $T$.
As in Definition \ref{id}, the scheme $\Sigma'$ is smooth over $X$ and there exists a 
natural $\O_{\Sigma}$-linear isomorphism
\begin{equation*}
\delta : (\alpha_0)^* \N_{\Sigma'} \xrightarrow\sim
\sheafHom_{\Sigma}(\ker((\alpha_0)_\Sigma, \coker((\alpha_0)_\Sigma)).
\end{equation*}
Let $\mathbf d_{\Sigma'} \alpha : (\T_X)_\Sigma \to (\alpha_0)^* \N_{\Sigma'}$ be the normal differential of $\alpha$ with respect to $\Sigma$ and set $\mathbf d_{\Sigma} \alpha := \delta \circ  \mathbf d_{\Sigma'} \alpha$.
\end{definition}

\begin{proposition}
\label{id-jet-prop}
Suppose that $A$ and $B$ are free $\O_X$-modules.
Choose bases for $A$ and $B$ and let 
\begin{equation*}
D : \sheafHom(A,B)\to \Omega_X\otimes \sheafHom(A,B)
\end{equation*}
be the trivial connection corresponding to theses bases.
Let
\begin{equation*}
\overline D : \PP_X^1\sheafHom(A,B)\to \Omega_X\otimes \sheafHom(A,B)
\end{equation*}
be the $\O_X$-linear map induced by $D$, see Proposition \ref{univ-DO}.
Then the intrinsic differential $\mathbf d_{\Sigma} \alpha$ is equal to the composition
\begin{equation*}
\begin{tikzcd}
(\T_X)_{\Sigma} \ar[r,"\overline D\alpha"] &
\sheafHom_X(A,B)_{\Sigma} \ar[r,two heads] &
\sheafHom_{\Sigma}(\ker((\alpha_0)_{\Sigma}), \coker((\alpha_0)_{\Sigma})).
\end{tikzcd}
\end{equation*}
\end{proposition}

\begin{proof}
Follows from Proposition \ref{nd-jet-prop}, see the proof of Proposition \ref{id-locally}.
\end{proof}

\section{The second intrinsic differential of a jet}

Let $S$ be a scheme.
Let $X$ be a smooth scheme over $S$.
Let $E$ be a locally free sheaf of finite rank on $X$.
Let $\nabla : E\to \Omega_X\otimes E$ be a connection. 

Let 
\begin{equation*}
d_{\Omega\otimes E}^1 : \Omega_X\otimes E\to
\PP_X^1 (\Omega_X\otimes E)
\end{equation*}
be the universal first-order differential operator.
The composition $d_{\Omega\otimes E}^1\circ \nabla$ is a second-order differential operator; let
\begin{equation*}
\overline\nabla : \PP_X^2 E\to \PP_X^1 (\Omega_X\otimes E)
\end{equation*}
be the the unique $\O_X$-linear map such that $\overline\nabla \circ d_E^2 = d_{\Omega\otimes E}^1\circ \nabla$, see Proposition \ref{univ-DO}.

Let $x: T\to X$ be a morphism of schemes.
Let
\begin{align*}
s&\in \Gamma(T,x^* \PP_X^2 E)\\
\overline\nabla s&\in \Gamma(T,x^* \PP_X^1 (\Omega_X\otimes E))\\
(\overline\nabla s)_0 &\in \Hom_T((\T_X)_T, E_T)
\end{align*}
be a section, its covariant derivative, and the image of this covariant derivative under the natural map $\PP_X^1 (\Omega_X\otimes E)\to \Omega_X\otimes E$.

Let $i$ be a nonnegative integer. 
Let $\Sigma$ denote the critical locus $\Sigma^i(s)\setminus \Sigma^{i+1}(s)\subseteq T$.
Let $K$ and $C$ respectively denote the kernel and cokernel of the map of $\O_\Sigma$-modules $((\overline\nabla s)_0)_{\Sigma} : (\T_X)_\Sigma \to E_\Sigma$ obtained by retricting the covariant derivative $(\overline\nabla s)_0$ to the subscheme $\Sigma\subseteq T$.

\begin{definition}
\label{order-2-sings-jets}
The \textbf{second intrinsic differential} of $s$ on $\Sigma$ is the map of $\O_\Sigma$-modules
\begin{equation*}
\mathbf d_{\Sigma}^2 s : K \to \sheafHom_{\Sigma}(K,C)
\end{equation*}
obtained by restricting to $K$ the intrinsic differential of $\overline\nabla s$ on $\Sigma$,
\begin{equation*}
\mathbf d_{\Sigma} (\overline \nabla s) : (\T_X)_{\Sigma} \to \sheafHom_{\Sigma}(K,C). 
\end{equation*}
\end{definition}

\begin{definition}
Let $j$ be a nonnegative integer. 
The \textbf{second-order Thom-Boardman singularity} $\Sigma^{i,j}(s)\subseteq T$ is the $j$th degeneracy locus of the second intrinsic differential $\mathbf d_\Sigma^2 s$. In symbols,
\begin{equation*}
\Sigma^{i,j}(s) := \Sigma^j(\mathbf d_\Sigma^2 s).
\end{equation*}
The pair of integers $(i,j)$ is called the \textbf{symbol} of  $\Sigma^{i,j}(s)$. 
\end{definition}

Let $\beta : \Sym^2 (K^\vee) \to (K\otimes K)^\vee$ be the $\O_\Sigma$-linear map that sends $uv \mapsto u\otimes v + v\otimes u$. In other words, $\beta$ sends a quadratic form on $K$ to its associated bilinear form. Let
\begin{equation*}
\theta : \sheafHom_\Sigma(K\otimes K, C) \xrightarrow\sim
\sheafHom_\Sigma(K,\sheafHom_\Sigma(K,C))
\end{equation*}
be the canonical $\O_\Sigma$-linear isomorphism that sends $b\mapsto (v\mapsto b(v\otimes -))$.
Let $B$ be the composition of natural $\O_\Sigma$-linear maps,
\begin{equation*}
\begin{tikzcd}[column sep = small]
(\Sym^2 \Omega_X \otimes E)_{\Sigma} \ar[r,two heads] &
\Sym^2 (K^\vee) \otimes C \ar[r, "\beta\otimes \mathrm{id}"] &
\sheafHom_\Sigma(K\otimes K, C) \ar[r,"\theta","\sim"'] &
\sheafHom_{\Sigma}( K, \sheafHom_\Sigma(K, C)).
\end{tikzcd}
\end{equation*}

\begin{remark}
\label{image-quad-to-bil}
The images of $\beta$ and $B$ depend on the characteristic of the base scheme $S$.
If 2 is invertible in $\O_S$, then $\operatorname{im}(\beta) = (\Sym^2 K)^\vee$ and 
\begin{equation*}
\operatorname{im}(B) = \theta( \sheafHom_\Sigma(\Sym^2 K,C) ) .
\end{equation*}
If $2=0$ in $\O_S$, then $\operatorname{im}(\beta) = (\wedge^2 K)^\vee$ and 
\begin{equation*}
\operatorname{im}(B) = \theta( \sheafHom_\Sigma(\wedge^2 K,C) ) .
\end{equation*}
In mixed characteristic, the two images need not be locally free.
\end{remark}

Recall the canonical short exact sequence
\begin{equation*}
\begin{tikzcd}
0 \ar[r] &
\Sym^2(\Omega_X)\otimes E \ar[r,"\delta"] &
\PP_X^2 E \ar[r, two heads] &
\PP_X^1 E \ar[r] &
0,
\end{tikzcd}
\end{equation*}
see Proposition \ref{smooth-pp} (3).

The following proposition reflects two basic properties, one being the symmetry, of the Hessian matrix of a function.

\begin{proposition}
\label{symmetry-2nd-id}
If $t \in \Gamma(T,x^* (\Sym^2 \Omega_X \otimes E))$ is a section, the second intrinsic differentials of $s$ and $s+\delta t$ are $\O_\Sigma$-linear maps $K\to\sheafHom_\Sigma(K,C)$ related by the equation
\begin{equation*}
\mathbf d_\Sigma^2 (s+\delta t)  = 
\mathbf d_\Sigma^2 s  + B(t).
\end{equation*}
If $E$ is generated as an $\O_X$-module by its subsheaf of horizontal sections $\ker(\nabla) \subseteq E$, then $\mathbf d_\Sigma^2 s$ is contained in the image of $B$.\end{proposition}

\begin{proof}
The questions being Zariski-local on $X$, we may assume 
that there exist \'etale coordinates defined on all of $X$, that is, sections $x_1,\dotsc,x_n\in \Gamma(X,\O_X)$ whose differentials form an $\O_X$-linear basis of $\Omega_X$;
that the $\O_X$-module $E$ is free with basis given by sections $v_1,\dotsc,v_e\in \Gamma(X,E)$; and
that $\nabla v_l=0$ for all $l=1,\dotsc,e$ if $E$ is generated by its subsheaf of horizontal sections.

Let
\begin{equation*}
D: \Omega_X\otimes E \to \Omega_X\otimes (\Omega_X\otimes E)
\end{equation*}
be the trivial connection corresponding to the basis of $\Omega_X\otimes E$ formed by the sections $dx_a\otimes v_l$, where $a=1,\dotsc, n$ and $l=1,\dotsc, e$.
Thus $D$ is the unique connection on $\Omega_X\otimes E$ that vanishes on the sections $dx_a\otimes v_l$. Let
\begin{equation*}
\overline D : \PP_X^1 (\Omega_X\otimes E ) \to \Omega_X\otimes (\Omega_X \otimes E)
\end{equation*}
be the corresponding $\O_X$-linear map.

By Proposition \ref{id-jet-prop}, the second intrinsic differentials of $s$ and $s+\delta t$ are the images of these sections under the composition of $\O_\Sigma$-linear maps
\begin{equation*}
\begin{tikzcd}[column sep = small]
(\PP_X^2 E)_\Sigma \ar[r,"\overline D\overline \nabla"] &
\sheafHom_X(\T_X, \sheafHom_X(\T_X, E))_\Sigma \ar[r, two heads, "\eta"] &
\sheafHom_\Sigma(K,\sheafHom_\Sigma(K,C)).
\end{tikzcd}
\end{equation*}

For $a=1,\dotsc, n$, let $\epsilon_a := d_X^2 x_a - x_a \in \Gamma(X,\PP_X^2)$. Then the sections $1, \epsilon_a, \epsilon_a\epsilon_b \in \Gamma(X,\PP_X^2)$, where $a,b = 1,\dotsc, n$, form a basis for $\PP_X^2$ as an $\O_X$-module; and the sections $d_E^2 v_1,\dotsc, d_E^2 v_e\in \Gamma(X,\PP_X^2 E)$ form a basis for $\PP_X^2 E$ as a $\PP_X^2$-module.

To prove both assertions, it suffices to show that 
\begin{align*}
\overline D\overline \nabla ( d_E^2 v_l ) &= D\nabla v_l \\
\overline D\overline \nabla ( \epsilon_a\cdot  d_E^2 v_l ) &= dx_a \otimes \nabla v_l \\
\overline D\overline \nabla ( \epsilon_a \epsilon_b\cdot d_E^2 v_l ) &= (dx_a\otimes dx_b + dx_b \otimes dx_a)\otimes v_l
\end{align*}
for all $a,b=1,\dotsc,n$ and $l=1,\dotsc, e$.
Note that the third equation implies $\eta \overline D\overline \nabla \delta = B$, as 
\begin{equation*}
\delta(dx_a dx_b\otimes v_l) = \epsilon_a \epsilon_b\cdot d_E^2 v_l
\end{equation*}
for all $a,b$ and $l$.

Clearly,
\begin{align*}
\overline D \overline \nabla (d_E^2 v_l) 
&= \overline D( d^1_{\Omega\otimes E}(\nabla v_l) )\\
&= D\nabla v_l
\end{align*}
for all $l$, as required.

Next, by the Leibniz rule, we have
\begin{align*}
\overline \nabla (\epsilon_a \cdot d_E^2 v_l)
&= \overline d(\epsilon_a) \otimes d_E^1 v_l +
\overline{\epsilon_a}\cdot \overline \nabla (d_E^2 v_l)
\end{align*}
for all $a$ and $l$. Here $\bar d$ denotes the $\O_X$-linear map $\PP_X^2 \to \PP_X^1 (\Omega_X)$ induced by the differential $d: \O_X\to \Omega_X$; the tensor product refers to the canonical isomorphism 
\begin{equation*}
\PP_X^1(\Omega_X)\otimes_{\PP_X^1} \PP_X^1(E) = \PP_X^1 (\Omega_X\otimes E);
\end{equation*}
and $\overline{\epsilon_a}$ denotes the image of $\epsilon_a$ under the natural map $\PP_X^2 \to \PP_X^1$.
Note that this image is identified with $dx_a$ under the canonical splitting $\PP_X^1 = \Omega_X\oplus \O_X$.
Note also that
\begin{align*}
\bar d (\epsilon_a)
&= \bar d ( d_X^2 x_a - x_a ) \\
&= d_\Omega^1(dx_a) - x_a d_\Omega^1 (d(1))\\
&= d_\Omega^1(dx_a)
\end{align*}
for all $a$. Thus
\begin{align*}
\overline \nabla (\epsilon_a \cdot d_E^2 v_l) 
&= d_\Omega^1(dx_a)\otimes d_E^1 v_l + dx_a \cdot d_{\Omega\otimes E}^1(\nabla v_l) \\
&= d_{\Omega\otimes E}^1 (dx_a\otimes v_l) + dx_a \cdot d_{\Omega\otimes E}^1(\nabla v_l),
\end{align*}
for all $a$ and $l$.

By the Leibniz rule,
\begin{equation*}
\overline D ( dx_a \cdot d_{\Omega\otimes E}^1(\nabla v_l) ) = \tilde d (dx_a) \otimes \nabla v_l + \overline{dx_a}\cdot \overline D( d_{\Omega\otimes E}^1(\nabla v_l) )
\end{equation*}
for all $a$ and $l$.
Here $\tilde d$ denotes the $\O_X$-linear map $\PP_X^1 \to \Omega_X$ induced by the differential $d: \O_X\to \Omega_X$, and $\overline{dx_a}$ denotes the image of $dx_a$ under the canonical map $\PP_X^1 \to \O_X$.
These are the two maps that define the canonical splitting $\PP_X^1 = \Omega_X\oplus \O_X$, so $\tilde d(dx_a)=dx_a$ and $\overline{dx_a}=0$ for all $a$.
Thus 
\begin{equation*}
\overline D ( dx_a \cdot d_{\Omega\otimes E}^1(\nabla v_l) ) = dx_a \otimes \nabla v_l,
\end{equation*}
hence
\begin{align*}
\overline D \overline \nabla (\epsilon_a \cdot d_E^2 v_l)
&= \overline D ( d_{\Omega\otimes E}^1 (dx_a\otimes v_l) + dx_a \cdot d_{\Omega\otimes E}^1(\nabla v_l) )\\
&= D(dx_a\otimes v_l) + dx_a \otimes \nabla v_l\\
&= dx_a \otimes \nabla v_l
\end{align*}
for all $a$ and $l$, as required.

By the Leibniz rule, we have
\begin{equation*}
\overline \nabla (\epsilon_a \epsilon_b \cdot d_E^2 v_l)
= \bar d(\epsilon_a \epsilon_b)\otimes d_E^1 v_l + 
\overline{\epsilon_a \epsilon_b} \cdot \overline\nabla (d_E^2 v_l)
\end{equation*} 
for all $a, b$ and $l$. Here $\overline{\epsilon_a \epsilon_b}$ denotes the image of $\epsilon_a\epsilon_b$ under the natural map $\PP_X^2 \to \PP_X^1$. This image is zero because $\epsilon_a \epsilon_b$ is the image of $dx_adx_b$ under the canonical map $\Sym^2 \Omega_X\to \PP_X^2$. Again by the Leibniz rule,
\begin{align*}
\bar d ( \epsilon_a \epsilon_b )
&= \overline{\epsilon_a} \bar d(\epsilon_b) + \overline{\epsilon_b} \bar d(\epsilon_a)\\
&= dx_a \cdot d_\Omega^1(dx_b) + dx_b \cdot d_\Omega^1(dx_a) 
\end{align*}
for all $a$ and $b$.
Therefore 
\begin{equation*}
\overline \nabla (\epsilon_a \epsilon_b \cdot d_E^2 v_l)=
dx_a \cdot d_{\Omega\otimes E}^1(dx_b\otimes v_l) + dx_b \cdot d_{\Omega\otimes E}^1(dx_a\otimes v_l)
\end{equation*}
for all $a, b$ and $l$.

To conclude, note that, by the Leibniz rule, 
\begin{align*}
\overline D(dx_a \cdot d_{\Omega\otimes E}^1(dx_b\otimes v_l))&=
\tilde d(dx_a)\otimes (dx_b\otimes v_l) + \overline{dx_a}\cdot \overline D( d_{\Omega\otimes E}^1(dx_b\otimes v_l) )\\
&= dx_a\otimes dx_b\otimes v_l
\end{align*}
for all $a,b$ and $l$. Thus
\begin{equation*}
\overline D\overline \nabla (\epsilon_a \epsilon_b \cdot d_E^2 v_l)=
dx_a\otimes dx_b\otimes v_l +
dx_b\otimes dx_a\otimes v_l
\end{equation*}
for all $a,b$ and $l$, as required.
\end{proof}

\section{Degeneracy loci in bundles of bilinear maps}

Let $S$ be a scheme.
Let $E$ and $F$ be locally free sheaves of finite rank on $S$.
Let $\Box^2 E$ denote one of two $\O_S$-modules: either $\Sym^2 E$ or $\wedge^2 E$.
Let
\begin{equation*}
H := \mathbf V(\sheafHom_S(\Box^2 E, F))
\end{equation*}
be the vector bundle over $S$ associated to the locally free $\O_S$-module
\linebreak
$\sheafHom_S(\Box^2 E,F)$.
Let $h : (\Box^2 E)_H\to F_H$ be the tautological map.

Let
\begin{equation*}
\theta : \sheafHom_S(E\otimes E, F) \xrightarrow\sim
\sheafHom_S(E,\sheafHom(E,F))
\end{equation*}
be the $\O_S$-linear isomorphism that sends $b\mapsto (v\mapsto b(v\otimes -))$.

Let $e$ and $f$ denote the respective ranks of $E$ and $F$ as locally free $\O_S$-modules.
Suppose that $f\ge 1$.
Let $j$ be an integer such that $0\le j\le e$.

\begin{proposition}
\label{univ-bil-deg}
The degeneracy locus $\Sigma^j(\theta h) \setminus \Sigma^{j+1}(\theta h)\subseteq H$ is smooth over $S$, of pure relative codimension
\begin{equation*}
j(e-j)(f-1) +
\tfrac 1 2 j(j\pm1) f
\end{equation*}
in $H$ over $S$. The symbol $\pm$ appearing in this expression should be read as ``plus'' if $\Box^2 E = \Sym^2 E$ and as ``minus'' if $\Box^2 E = \wedge^2 E$.\end{proposition}

\begin{proof}
Let
\begin{equation*}
(\theta h)^\vee : (E\otimes F^\vee)_H\to E^\vee_H
\end{equation*}
be the dual of $\theta h$. 
Then
\begin{equation*}
\Sigma^j(\theta h) \setminus \Sigma^{j+1}(\theta h)
=
\Sigma^j((\theta h)^\vee) \setminus \Sigma^{j+1}((\theta h)^\vee),
\end{equation*}
see Proposition \ref{deg-loci-dual}.
By Remark \ref{tjurina-isomorphism}, this degeneracy locus is isomorphic as a scheme over $S$ to an open subscheme of the scheme $Z$ of Proposition \ref{tjurina-univ-bil} below.
The result therefore follows from that proposition.
\end{proof}

As in Construction \ref{tjurina-transform}, let  $G:=G_j(E^\vee_H)$ be the Grassmannian of rank $j$ quotients of $E^\vee_H$ over $H$, let $q:E^\vee_G\twoheadrightarrow Q$ be the universal quotient on $G$, and let $Z$ be the closed subscheme of $G$ defined by the equation $q\circ (\theta h)^\vee_G=0$.

\begin{proposition}
\label{tjurina-univ-bil}
The scheme $Z$ is smooth over $S$, of pure relative dimension
\begin{equation*}
\tfrac 1 2 e(e\pm 1)f 
- j(e-j)(f-1) - \tfrac 1 2 j(j\pm1)f
\end{equation*}
over $S$. The symbol $\pm$ appearing in this expression should be read as ``plus'' if $\Box^2 E = \Sym^2 E$ and as ``minus'' if $\Box^2 E = \wedge^2 E$.
\end{proposition}

\begin{proof}
Suppose for definiteness that $\Box^2 E=\wedge^2 E$.
It will be clear that the argument that follows implies the result also in the case where $\Box^2 E = \Sym^2 E$.

The question being local on $S$, we may assume that $E$ and $F$ are free as $\O_S$-modules. Write $E=E_1\oplus E_2$, where the two summands are generated by complementary subsets of a basis of $E$ and $E_1$ has rank $j$.
Let $U\subseteq G$ be the largest open subset over which the $\O_G$-linear map
\begin{equation*}
q|_{E_1^\vee}:(E_1)_G^\vee \to Q
\end{equation*}
is surjective, and hence an isomorphism. Such open subsets of $G$ obtained from partitions of the basis of $E$ form a cover of $G$, so it suffices to show that $U\cap Z$ is a smooth $S$-scheme of relative dimension
\begin{equation*}
\tfrac 1 2 e(e- 1)f 
- j(e-j)(f-1) - \tfrac 1 2 j(j- 1)f.
\end{equation*}

Let $t_1,\dotsc,t_f\in \Gamma(S,F)$ form a basis of $F$ as an $\O_S$-module. 
Each of the $\O_H$-linear maps
\begin{equation*}
(\theta h^\vee)(t_l) : E_H\to E_H^\vee
\end{equation*}
is represented by a skew-symmetric matrix of size $e$ and entries in $\Gamma(H,\O_H)$.
The collection of these entries (with $l$ varying) induces an isomorphism of schemes
\begin{equation*}
H\xrightarrow\sim \mathbf A_S^{(1/2) e(e-1)f}
\end{equation*}
over $S$, see Remark \ref{vb-trivial}.

For each $l=1,\dots, f$ and $u,v = 1,2$, let $h_{uv}^l$ denote the composition 
\begin{equation*}
\begin{tikzcd}[column sep = large]
(E_v)_H \ar[r,hook] & 
E_H \ar[r,"(\theta h^\vee)(t_l)"] & 
E^\vee_H \ar[r, two heads] & 
(E_u)^\vee_H.
\end{tikzcd}
\end{equation*} 
Then we may write 
\begin{equation*}
(\theta h^\vee)(t_l) = \begin{bmatrix}
h_{11}^l & h_{12}^l\\
h_{21}^l & h_{22}^l
\end{bmatrix}.
\end{equation*}
Note the identities
\begin{equation*}
(h_{11}^l)^\vee = -h_{11}^l ,
\qquad
(h_{12}^l)^\vee = -h_{21}^l ,
\qquad\text{and}\qquad
(h_{22}^l)^\vee = -h_{22}^l.
\end{equation*}

Let $q'$ denote the composition
\begin{equation*}
\begin{tikzcd}
(E_2)_U^\vee \ar[r] &
E^\vee_U \ar[r, two heads, "q_U"] &
Q_U \ar[r, "(q_U|_{E_1})^{-1}", "\sim"'] &
(E_1)^\vee_U.
\end{tikzcd}
\end{equation*}
Then $q'$ is represented by a matrix of shape $j\times(e-j)$ and entries in $\Gamma(U,\O_G)$.
These entries define one of the standard affine charts on the Grassmannian, an isomorphism of schemes 
\begin{equation*}
U\xrightarrow\sim \A_H^{j(e-j)}
\end{equation*}
over $X$.

For each $l=1,\dotsc, f$, the composition
\begin{equation*}
\begin{tikzcd}[column sep = large]
(E_1\oplus E_2)_U \ar[r,"(\theta h^\vee)(t_l)"] & 
(E_1\oplus E_2)^\vee_U \ar[r, two heads, "q_U"] &
Q_U \ar[r, "(q_U|_{E_1})^{-1}", "\sim"'] &
(E_1)^\vee_U
\end{tikzcd}
\end{equation*} 
is given by the matrix product
\begin{equation*}
\begin{bmatrix}
1 & q'
\end{bmatrix}
\begin{bmatrix}
h_{11}^l & h_{12}^l\\
h_{21}^l & h_{22}^l
\end{bmatrix}=
\begin{bmatrix}
h_{11}^l + q' h_{21}^l & h_{12}^l + q' h_{22}^l
\end{bmatrix}.
\end{equation*}
Thus $Z\cap U$ is the subscheme of $U$ defined by the equations
\begin{equation}
\label{eqns-tjurina-I}
h_{11}^l + q' h_{21}^l=0
\quad\quad\text{and}\quad\quad
h_{12}^l + q' h_{22}^l = 0,
\end{equation}
where $l=1,\dotsc, f$.

Another set of equations defining $Z\cap U$ as a subscheme of $U$ is
\begin{equation}
\label{eqns-tjurina-II}
h_{11}^l - q' h_{22}^l (q')^\vee = 0
\quad\quad\text{and}\quad\quad
h_{12}^l + q' h_{22}^l = 0,
\end{equation}
where $l=1,\dotsc, f$. Indeed, in the presence of the dual $(h_{12}^l)^\vee  + (q' h_{22}^l)^\vee =0 $ of the second equation in (\ref{eqns-tjurina-I}) and (\ref{eqns-tjurina-II}), the following holds:
\begin{equation*}
h_{21}^l = -(h_{12}^l)^\vee = + (q' h_{22}^l)^\vee =
(h_{22}^l)^\vee (q')^\vee =
-h_{22}^l (q')^\vee
\end{equation*}

The $j(e-j)f$ entries of the matrices representing the $\O_U$-linear maps $h_{12}^l$ are independent coordinates on
\begin{equation*}
U \cong \A_S^{(1/2)e(e-1)f + j(e-j)}.
\end{equation*}
The entries of the matrices representing $q' h_{22}^l$ are (quadratic) polynomials in a disjoint set of coordinates.
Therefore the subscheme $W\subseteq U$ defined by the equations $h_{12}^l + q' h_{22}^l = 0$ (where $l$ varies) is an affine space over $S$ of relative dimension
\begin{equation*}
[\tfrac 1 2 e(e-1)f + j(e-j)] - j(e-j)f.
\end{equation*}

The equations $h_{11}^l - q' h_{22}^l (q')^\vee = 0$ define  $Z\cap U$ as a subscheme of $W$. The $j^2 f$ entries of the matrices representing the $\O_U$-linear maps $h_{11}^l$ are coordinates on $U$. Of these coordinates, $(1/2) j (j-1) f$ are independent, corresponding to the skew-symmetry of $h_{11}^l$. They do not occur in the matrices representing the $h_{12}^l$ and $q' h_{22}^l$, nor in skew-symmetric matrices representing the $q' h_{22}^l (q')^\vee$. Therefore $Z\cap U$ is an affine space over $S$ of relative dimension 
\begin{equation*}
[\tfrac 1 2 e(e-1)f + j(e-j)] - j(e-j)f - \tfrac 1 2 j(j-1) f.
\end{equation*}
The result follows.
\end{proof}

\section{Second-order singularities of generic sections}

Let $S$ be a scheme.
Suppose that either $2$ is invertible or $2=0$ in $\O_S$.
Let $X$ be a smooth scheme over $S$.
Let $E$ be a locally free sheaf of finite rank on $X$.
Let $\nabla : E\to \Omega_X\otimes E$ be a connection.
Suppose that $E$ is generated as an $\O_X$-module by its subsheaf of horizontal sections $\ker(\nabla)\subseteq E$. 

Let $i$ and $j$ be nonnegative integers. For $m=1,2$,
\begin{itemize}
\item let $J^m \to X$ be the vector bundle associated to the locally free $\O_X$-module $\PP_X^m E$
\item let $s_m \in \Gamma(J^m,(\PP_X^m E)_{J^m})$ be the tautological section; and
\item let $\Sigma_m$ denote the critical locus $\Sigma^i(s_m)\setminus \Sigma^{i+1}(s_m)\subseteq J^m$. 
\end{itemize}
Let $q : J^2 \twoheadrightarrow J^1$ be the morphism of schemes over $X$ induced by the natural surjection $\PP_X^2 E \to \PP_X^1 E$, see Remark \ref{surj-pp}. 
Note that $q$ is smooth and surjective by Proposition \ref{vb-surj}.

\begin{remark}
\label{relation-critical-loci}
The tautological sections $s_1$ and $s_2$ have covariant derivatives
\begin{equation*}
\nabla s_1 \in \Gamma(J^1, (\Omega_X\otimes E)_{J^1})
\qquad\text{and}\qquad
\nabla s_2 \in \Gamma(J^2, \PP_X^1(\Omega_X\otimes E)_{J^2}),
\end{equation*}
see Remark \ref{higher-univ-DO}.
These are related as follows: the image of $\nabla s_2$ under the natural $\O_X$-linear map 
\begin{equation*}
\PP_X^1 (\Omega_X\otimes E) \to \Omega_X\otimes E
\end{equation*}
is equal to $q^* (\nabla s_1)$.
Thus $q^{-1}\Sigma_1 = \Sigma_2$ as subschemes of $J^2$ by the functoriality of degeneracy loci, see Proposition \ref{deg-loci-functor}. In other words, we have a Cartesian diagram:
\begin{equation}
\label{critical-cartesian}
\begin{tikzcd}
\Sigma_2 \ar[r, hook] \ar[d, two heads] \ar[rd, phantom, "\square"] &
J^2 \ar[d, two heads, "q"] \\
\Sigma_1 \ar[r, hook] &
J^1
\end{tikzcd}
\end{equation}
\end{remark}

Let $n$ denote the relative dimension of $X$ over $S$.
Let $e$ denote the rank of $E$ as a locally free $\O_X$-module.
Let $m:=\min(n,e)$.

\begin{proposition}
\label{univ-sigma-i-j}
The second-order singular locus $\Sigma^{i,j}(s_2)\setminus \Sigma^{i,j+1}(s_2)\subseteq \Sigma_2$ is smooth over $X$, of pure relative codimension 
\begin{equation*}
\label{rel-codim-cases}
i(|n-e|+i) +
j(n-m+i-j)(e-m+i-1) +\tfrac 1 2 j(j\pm1)(e-m+i)
\end{equation*}
in $J^2$ over $X$. 
The symbol $\pm$ appearing in this expression should be read as ``plus'' if 2 is invertible in $\O_S$ and as ``minus'' if $2=0$ in $\O_S$.
\end{proposition}

\begin{proof}
Combining the Cartesian diagram (\ref{critical-cartesian}) with Proposition \ref{univ-1st-sings}, which states that the critical locus $\Sigma_1$ is smooth over $X$, of pure relative codimension $i(|n-e|+i)$ in $J^1$ over $X$, we see that the critical locus $\Sigma_2$ is smooth over $\Sigma_1$, hence over $X$, of pure relative codimension $i(|n-e|+i)$ in $J^2$ over $X$. The result therefore follows from Proposition \ref{univ-sigma-i-j-bis} below. 
\end{proof}

\begin{proposition}
\label{univ-sigma-i-j-bis}
The second-order singular locus $\Sigma^{i,j}(s_2)\setminus \Sigma^{i,j+1}(s_2)\subseteq \Sigma_2$ is smooth over $\Sigma_1$, of pure relative codimension 
\begin{equation*}
\label{rel-codim-cases-bis}
j(n-m+i-j)(e-m+i-1) +\tfrac 1 2 j(j\pm1)(e-m+i)
\end{equation*}
in $\Sigma_2$ over $\Sigma_1$.
The symbol $\pm$ appearing in this expression should be read as ``plus'' if 2 is invertible in $\O_S$ and as ``minus'' if $2=0$ in $\O_S$.
\end{proposition}

\begin{proof}
Let $K$ and $C$ respectively denote the kernel and cokernel of the $\O_{\Sigma_1}$-linear map $(\nabla s_1)_{\Sigma_1} : (\T_X)_{\Sigma_1} \to E_{\Sigma_1}$.
The second intrinsic differential of $s_2$ on $\Sigma_2$ is an $\O_{\Sigma_2}$-linear map
\begin{equation*}
\mathbf d_{\Sigma_2}^2 s_2 : q^* K \to q^* \sheafHom_{\Sigma_1}(K,C),
\end{equation*}
see Definition \ref{order-2-sings-jets}, Remark \ref{relation-critical-loci} and Proposition \ref{coker-deg-loci}.

Let 
\begin{equation*}
\theta : \sheafHom_{\Sigma_1}(K\otimes K, C) \xrightarrow\sim
\sheafHom_{\Sigma_1}(K, \sheafHom_{\Sigma_1}(K,C))
\end{equation*}
be the $\O_{\Sigma_1}$-linear isomorphism that sends $b\mapsto (v\mapsto b(v\otimes -))$.
Let
\begin{equation*}
\Box^2 K := \begin{cases}
\Sym^2 K & \text{if $2$ is invertible in $\O_S$}\\
\wedge^2 K & \text{if $2=0$ in $\O_S$}.
\end{cases}
\end{equation*}
By Proposition \ref{symmetry-2nd-id}, the inverse image $\theta^{-1} (\mathbf d_{\Sigma_2}^2 s_2)$ is contained in the subsheaf 
\begin{equation*}
\sheafHom_{\Sigma_1}(\Box^2 K, C) \subseteq 
\sheafHom_{\Sigma_1}(K\otimes K, C). 
\end{equation*}

Let $H \to \Sigma_1$ be the vector bundle associated to the locally free $\O_{\Sigma_1}$-module $\sheafHom_{\Sigma_1}(\Box^2 K, C)$.
Let $h : (\Box^2 K)_H \to C_H$ be the tautological map.
For each nonnegative integer $j'$, we have a commutative diagram with Cartesian square as follows.
\begin{equation*}
\begin{tikzcd}[column sep=large]
\Sigma^{i,j}(s_2)\setminus \Sigma^{i,j+1}(s_2) \ar[d, hook] \ar[r] \ar[rd, phantom, "\square"] &
\Sigma^{j}(\theta h) \setminus \Sigma^{j+1}(\theta h) \ar[d, hook]
\\
\Sigma_2 \ar[r, "\theta^{-1}(\mathbf d_{\Sigma_2}^2 s_2)"'] \ar[d, two heads] &
H \ar[ld, bend left = 10]
\\
\Sigma_1
\end{tikzcd}
\end{equation*}

The subscheme $\Sigma^j(\theta h)\setminus \Sigma^{j+1}(\theta h)\subseteq H$ is smooth over $\Sigma_1$, of pure relative codimension (\ref{rel-codim-cases-bis}) in $H$ over $\Sigma_1$, by Proposition \ref{univ-bil-deg}. To complete the proof, it suffices to show that the arrow $\theta^{-1}(\mathbf d_{\Sigma_2}^2 s_2)$ is smooth and surjective. We do this by verifying the hypotheses of Lemma \ref{equiv-surj} below.

First, the short exact sequence (\ref{pp-exact-seq}) gives the $\Sigma_1$-scheme $\Sigma_2$ the structure of a principal bundle under $\Sigma_1$-group scheme
\begin{equation*}
\mathbf V(\Sym^2 \Omega_X \otimes E)\times_X \Sigma_1,
\end{equation*}
see Example \ref{ses-vb}.
Second, there is a natural  $\O_{\Sigma_1}$-linear map
\begin{equation*}
B' : (\Sym^2 \Omega_X \otimes E)_{\Sigma_1} \twoheadrightarrow \sheafHom_{\Sigma_1}(\Box^2 K,C),
\end{equation*}
by Remark \ref{image-quad-to-bil}.
The map $B'$ is surjective, so induces a smooth surjection between the corresponding vector bundles over $\Sigma_1$, by Proposition \ref{vb-surj}.
And third, the map $\theta^{-1}(\mathbf d_{\Sigma_2}^2 s_2)$ is $B'$-equivariant, by Proposition \ref{symmetry-2nd-id}.
Thus Lemma \ref{equiv-surj} applies and the result follows.
\end{proof}

\begin{lemma}
\label{equiv-surj}
Let $Z$ be a scheme. Let $\phi:G\to H$ be a homomorphism of group schemes over $Z$. Let $P$ be a principal $G$-bundle. Let $Q$ be a principal $H$-bundle. Let $f:P\to Q$ be a $\phi$-equivariant morphism of schemes over $Z$. If $\phi$ is smooth and surjective, then $f$ is smooth and surjective.
\end{lemma}

\begin{proof}
The properties of smoothness and surjectivity of a morphism of schemes are local in the Zariski topology of its target.
We may therefore assume that $P=G$ and that $Q=H$.

Let $W$ be a scheme over $Z$.
Let $g:W\to G$ be a morphism over $Z$, which we view as a $W$-valued point of $G$.
Let $e:Z\to G$ denote the identity section of $G$. Then 
\begin{equation*}
f(g)=f(ge_W)=\phi(g)f(e)_W.
\end{equation*}
Thus $f$ agrees with $\phi$ up to the automorphism of $H$ given by right translation by $f(e)$. The result follows.
\end{proof}

Suppose that $S$ is the spectrum of an infinite field $k$.

\begin{corollary}
\label{generic-sigma-i-j}
Let $W\subset \Gamma(X,E)$ be a $k$-linear subspace of finite dimension.
Suppose that $d_E^2(W)$ generates $\PP_X^2 E$ as an $\O_X$-module.
Let $s\in W$ be a general section.
Then every irreducible component of the second-order singular locus $\Sigma^{i,j}(s)\setminus \Sigma^{i,j+1}(s)\subseteq X$ has codimension 
\begin{equation*}
i(|n-e|+i) +
j(n-m+i-j)(e-m+i-1) +\tfrac 1 2 j(j\pm1)(e-m+i)
\end{equation*}
in $X$.
The symbol $\pm$ appearing in this expression should be read as ``plus'' if $k$ has characteristic different from 2 and as ``minus'' if $k$ has characteristic 2.
Furthermore, if $k$ has characteristic zero, then $\Sigma^{i,j}(s)\setminus \Sigma^{i,j+1}(s)$ is smooth.
\end{corollary}

\begin{proof}
Follows from the combination of Propositions \ref{serre-bertini} and \ref{univ-sigma-i-j}, see the proof of the corresponding result for first-order singularties (Corollary \ref{1st-sings-generic}).
\end{proof}

\chapter{Irrationality}

\section{A necessary condition for separable uniruledness}

\begin{definition}
Let $X$ be an integral scheme.
Let $F$ be a sheaf of $\O_X$-modules. 
The \textbf{torsion subsheaf} of $F$, denoted $F_\mathrm{tors}$, is the kernel of the natural map $F\to F\otimes_{\O_X} k(X)$.
\end{definition}

\begin{remark}
\label{torsion-free-pullback}
A dominant map of integral schemes $f:Y\to X$ induces an $\O_X$-linear map $k(X)\to f_*k(Y)$, hence a natural diagram of $\O_Y$-modules:
\begin{equation*}
\begin{tikzcd}
\arrow{r} \arrow["\alpha"]{d} f^*(F_\mathrm{tors}) &
\arrow{r} \arrow[equals]{d} f^*F &
\arrow["\beta"]{d} f^* (F \otimes_{\O_X} k(X))\\
( f^*F )_\mathrm{tors} \arrow{r}&
\arrow{r}f^* F &
f^*F \otimes_{\O_Y} k(Y)
\end{tikzcd}
\end{equation*}
The natural map $f^*(F/F_\mathrm{tors}) \to (f^*F) /(f^*M)_\mathrm{tors}$ is an isomorphism, see Lemma \ref{tf-ring-change} and Remark \ref{torsion-stalks} below.
\end{remark}

The proof of the main theorem in this thesis makes use of the following strengthening of a result of Koll\'ar's \cite[Lemma 7]{kollar1995}.

\begin{lemma}
\label{not-uniruled}
Let $k$ be a field with algebraic closure $\bar k$.
Let $X$ be a geometrically integral, proper $k$-scheme.
Let $i$ be a positive integer.
Let $T:= (\Omega_X^i)_\mathrm{tors}$.
Let $Q$ be a big line bundle on $X$.
Suppose that there exists an injection of $\O_X$-modules $q:Q\hookrightarrow \Omega_X^i/T$.
Then $X_{\bar k}$ is not separably uniruled, hence not ruled.
\end{lemma}

\begin{remark}
Thus we do not assume that $X$ is smooth, as Koll\'ar does.
If resolutions of singularities are known to exist, the lemma stated follows from the one in \cite{kollar1995}.
The proof that we give below is modeled on Koll\'ar's. 
\end{remark}

\begin{remark}
The reason we divide by torsion in the statement of Lemma \ref{not-uniruled} is that, when $X$ is singular, it can be easier to produce a map from a line bundle into $\Omega_X^i/T$ than one into $\Omega_X^i$.
Below we will start with a map $q' : Q\to \Omega_X^i \otimes k(X)$ and show that, locally on $X$, it factors through a map $Q\to \Omega_X^i$.
The factoring maps won't glue, in general, but their existence implies the image of $q'$ is contained in the subsheaf $\Omega_X^i/T\subseteq \Omega_X^i\otimes k(X)$.
\end{remark}

\begin{proof}
By Chow's lemma, there exists a projective $k$-variety $X'$ and a birational map $\pi: X'\to X$. The composition
\begin{equation*}
\pi^* Q \xrightarrow{\pi^* q} q^*(\Omega_X^i)/(q^*\Omega_X^i)_\mathrm{tors} \to \Omega_{X'}^i/(\Omega_{X'}^i)_\mathrm{tors}
\end{equation*}
is injective at the generic point $X'$, hence everywhere, since $q^* Q$ is torsion free.
Therefore, after replacing $X$ with $X'$, we may assume that $X$ is projective.
Similarly, considering the projection $X_{\bar k}\to X$, we may furthermore assume that $k$ is algebraically closed.

Aiming for a contradiction, suppose that $X$ is separably uniruled. Then there exists an irreducible $k$-variety $Y$ together with a generically finite, dominant and separable rational map $f:Y\times \P^1\dashrightarrow X$.

The restriction
\begin{equation*}
f:\Spec k(Y)\times \P^1 = \P^1_{k(Y)}\dashrightarrow X
\end{equation*}
is a morphism by properness of $X$. Thus by shrinking $Y$ we may assume $f$ is a morphism.

Being $f$  separable and generically finite, the natural map
\begin{equation}
f^*\Omega_X\to \Omega_{Y\times \P^1}
\label{criterion-map}
\end{equation}
is surjective at the generic point of $Y\times \P^1$. Because $X$ is generically smooth, we have
\begin{equation*}
\dim_{k(X)}( \Omega_X\otimes k(X)) = \dim X = \dim Y +1 \le
\dim_{k(Y\times \P^1)} (\Omega_{Y\times \P^1}\otimes k(Y\times \P^1)).
\end{equation*}
It follows that  (\ref{criterion-map}) is bijective at the generic point of $Y\times \P^1$ and that $Y$ is generically smooth. Shrinking $Y$, we may assume  $Y$ is smooth.

The $i$th exterior power of (\ref{criterion-map}) vanishes on $f^*T$ by smoothness of $Y$. The resulting composition 
\begin{equation*}
f^*Q \to f^*(\Omega_X^i/T) \to \Omega_{Y\times \P^1}^i
\end{equation*}
is injective at the generic point of $Y\times \P^1$. Because $f^*Q$ is torsion free, the composition is in fact injective everywhere. More generally, for any integer $m>0$, the map
\begin{equation*}
f^*Q^{\otimes m}\to (\Omega_{Y\times \P^1}^i)^{\otimes m}
\end{equation*}
is injective at the generic point of $Y\times \P^1$ and therefore everywhere.

Let $V\subseteq Y\times \P^1$ be the open subset of consisting of the the points $z\in Y\times \P^1$ such that the map
\begin{equation*}
f^* Q \otimes k(z) \to \Omega_{Y\times \P^1}^i \otimes k(z)
\end{equation*}
is injective. Thus, given a section $s\in f^*Q^{\otimes m}$ and a point $z\in V$, to check whether $s$ vanishes at $z$ it suffices to check whether the image of $s$ in $(\Omega_{Y\times \P^1}^i)^{\otimes m}$ vanishes at $z$. Note that $V$ contains the generic point of $Y\times \P^1$, so is dense.

Because $Q$ is big, there exist, by Kodaira's lemma, Cartier divisors $A$ and $E$ on $X$, the first very ample and the second effective, and an integer $m>0$ such that there exists an isomorphism $Q^{\otimes m} \cong \O_{X}(A+E)$. Choose such $A$, $E$ and $m$ and let $U:= X\setminus E$.

Because $f$ is generically finite, there exists a dense open subset $U'\subseteq X$ over which $f$ is finite.

Let $W=f^{-1}(U\cap U')\cap V$. This is a dense open subset of $Y\times \P^1$. Hence there exists a closed point $y\in Y$ such that the intersection $(y\times \P^1)\cap W$ is nonempty. This intersection in fact contains infinitely many closed points, so by definition of $U'$ there are two of them, say $z_1,z_2\in (y\times \P^1)\cap W$, which have distinct images in $U'\cap U\subseteq X$. By definition of $U$, these two points $z_1$ and $z_2$ are separated by a global section of $f^*Q^{\otimes m}$. By definition of $V$, it follows that $z_1,z_2\in y\times \P^1$ are separated by global sections of $( \Omega_{Y\times\P^1}^i)^{\otimes m}$. This is impossible, since 
\begin{equation}
H^0(Y\times\P^1,(\Omega_{Y\times \P^1}^i)^{\otimes m})
= H^0(Y,(\Omega_Y^i)^{\otimes m})\otimes_k H^0(\P^1,\O_{\P^1}).
\label{Kuenneth}
\end{equation}

To verify  (\ref{Kuenneth}), we compute:
\begin{align*}
(\Omega_{Y\times \P^1}^i)^{\otimes m}
&= (\wedge^i (\mathrm{pr}_1^*\Omega_Y \otimes \mathrm{pr}_2^*\Omega_{\P^1}) )^{\otimes m}\\
&= ( \wedge^i \mathrm{pr}_1^*  \Omega_Y \oplus (\wedge^{i-1}\mathrm{pr}_1^*  \Omega_Y\otimes \mathrm{pr}_2^*\Omega_{\P^1}) )^{\otimes m}\\
&= \oplus_{a+b=m} \mathrm{pr}_1^*(
(\wedge^{i-1}\Omega_Y)^{\otimes a} \otimes (\wedge^i \Omega_Y)^{\otimes b} ) \otimes
\mathrm{pr}_2^* (\Omega_{\P^1})^{\otimes a}
\end{align*}
The K\"unneth formula (see \cite[Proposition 9.2.4]{kempf1993} or \cite[Tag 0BEF]{stacks-project}) now gives (\ref{Kuenneth}), since $\Omega_{\P^1}^{\otimes a}$ has no global sections for $a>0$.
\end{proof}

\section{Inseparable covers}

Let $k$ be a field of positive characteristic $p$.
Let $X$ be a smooth, connected scheme of dimension $n$ over $k$.

\begin{definition}
The \textbf{(absolute) Frobenius morphism} of $X$ is the morphism of schemes $F_X : X\to X$ that acts as the identity on the topological space of $X$ and whose co-morphism $F_X^\# : \O_X\to (F_X)_* \O_X$ sends $f\mapsto f^p$. 
\end{definition}

\begin{remark}
If $M$ is a sheaf of $\O_X$-modules, the pushforward $(F_X)_* M$ has the same underlying sheaf of abelian groups as $M$, but is equipped with the action of $\O_X$ given by $f\cdot s = f^p s$.
\end{remark}

\begin{example}
\label{frobenius-line-bundle}
Let $L$ be an invertible sheaf on $X$.
Then the sheaf morphism
$
L\to (F_X)_* L^{\otimes p}
$
that sends $s\mapsto s^{\otimes p}$ is $\O_X$-linear. By adjunction, it induces a map $\O_X$-modules $F_X^* L \to L^{\otimes p}$, which is an isomorphism. (To see this, consider transition functions.)
\end{example}

Let $E$ be a locally free $\O_X$-module of finite rank $e$.
Suppose that $e\le n$. 

\begin{definition}
\label{can-conn-def}
The \textbf{canonical connection} on the Frobenius pullback $F_X^*E$ is the sheaf morphism
\begin{equation*}
\nabla : F_X^* E\to \Omega_X \otimes F_X^* E.
\end{equation*}
induced by the universal derivation $d :\O_X\to \Omega_X$, which is an $F_X^{-1} \O_X$-linear map, via the canonical isomorphism  $F_X^* E \cong F_X^{-1} E \otimes_{F^{-1}\O}\O_X$.
\end{definition}

In the sequel we regard the Frobenius pullback $F_X^*E$ as equipped with its canonical connection $\nabla$. Hence we may speak of critical loci of sections of this bundle.

Let $s\in \Gamma(X,F_X^* E)$ be a global section.

\begin{example}
\label{can-conn-locally}
Suppose that the $\O_X$-module $E$ is free. Let  $\{v_1, \dotsc, v_e \} \subseteq \Gamma(X,E)$ be a basis. 
Then
\begin{equation*}
s = f_1 \cdot F_X^* v_1 + \dotsb + f_e \cdot F_X^* v_e
\end{equation*}
for uniquely determined $f_1,\dotsc,f_e\in \Gamma(X,\O_X)$, and 
\begin{equation*}
\nabla s = d f_1 \otimes F_X^* v_1 + \dotsb + df_e \otimes F_X^* v_e.
\end{equation*}
The first critical locus of $s$, which is by definition equal to the first degeneracy locus of the map $\T_X\to F_X^*E$ induced by $\nabla s$, is the subscheme
\begin{equation*}
\Sigma^1 (s) = \{ df_1\wedge \dotsb \wedge df_e =0\}\subseteq X.
\end{equation*}
\end{example}

\begin{construction}
\label{insep-cover-constr}
Let $V:= \mathbf V(E)$ be the vector bundle associated to $E$.
Let $\pi:V\to X$ be the projection.
Let $F_V:V\to V$ be the absolute Frobenius morphism of $V$. Then the following diagram commutes:
\begin{equation*}
\begin{tikzcd}
V \ar["F_V"]{r} \ar["\pi"]{d} & V\ar["\pi"]{d} \\
X \ar["F_X"]{r} & X
\end{tikzcd}
\end{equation*}
Let $\tau \in \Gamma(V,\pi^* E)$ be the tautological section. The \textbf{inseparable cover} of $X$ determined by $s$ is the restriction of $\pi$ to
\begin{equation*}
X[\pthroot{s}] := \{ F_V^* \tau  = \pi^* s \} \subseteq \mathbf V(E).
\end{equation*}
\end{construction}

\begin{example}
\label{insep-cover-locally}
Suppose that the $\O_X$-module $E$ is free. Let  $\{v_1, \dotsc, v_e \} \subseteq \Gamma(X,E)$ be a basis. 
Then
\begin{equation*}
s = f_1 \cdot F_X^* v_1 + \dotsb + f_e \cdot F_X^* v_e
\end{equation*}
for uniquely determined $f_1,\dotsc,f_e\in \Gamma(X,\O_X)$.
The vector bundle associated to $E$ is $V=X \times \A^e$.
Let $t_1,\dotsc,t_e$ denote the coordiantes on $\A^e$.
The tautological section $\tau\in \Gamma(V,\pi^* E)$ is
\begin{equation*}
\tau = t_1\cdot  \pi^* v_1 + \dotsb + t_e\cdot \pi^* v_e. 
\end{equation*}
Because $\pi F_X = F_V \pi$, the Frobenius pullback of $\tau$ is
\begin{equation*}
F_V^* \tau = t_1^p\cdot \pi^* F_X^* v_1 + \dotsb + t_e^p \cdot \pi^* F_X^* v_e. 
\end{equation*}
Thus $X[\pthroot{s}]$ is the subscheme 
\begin{equation*}
\{ t_1^p - f_1 = \dotsb = t_e^p - f_e = 0 \}\subset X\times \A^r. 
\end{equation*}
\end{example}

\begin{example}
If $E$ is invertible, the inseparable cover $\pi : X[\pthroot s]\to X$ is cyclic of degree $p$, see Example \ref{frobenius-line-bundle} and Definition \ref{CyclicCovers} below.
\end{example}

\begin{proposition}
The morphism $\pi: X[\pthroot s] \to X$ is finite, flat and induces a homeomorphism of underlying spaces.
\end{proposition}

\begin{proof}
If $A$ is a ring and $f\in A$, then $A[t]/(t^p-f)$ is free $A$-module of rank $p$. Hence $\pi: X[\pthroot s] \to X$ is finite and flat. If $K$ is an algebraically closed field of characteristic $p$ and $f\in K$, then the spectrum of $K[t]/(t^p-f)$ consists of a single point. Thus the fibers of $\pi: X[\pthroot s] \to X$ consist of single points. The result follows.
\end{proof}

\begin{proposition}
The inseparable cover $X[\pthroot s]$ is a local complete intersection of codimension $e$ in the smooth $k$-scheme $\mathbf V(E)$.
\end{proposition}

\begin{proof}
Clear from the local picture of Example \ref{insep-cover-locally}.
\end{proof}

\begin{proposition}
\label{insep-cov-sing}
The inverse image $\pi^{-1}\Sigma^1(s)\subseteq X[\pthroot s]$ is the locus where $X[\pthroot s]$ is not smooth over $k$.
\end{proposition}

\begin{proof}
This follows from Propositions \ref{omega-insep-cov} and \ref{insep-C-loc-free} below, but we give a direct argument.
In the notation of Example \ref{insep-cover-locally}, we have $d(t_l^p - f_l) = - df_l$ for all $l=1,\dotsc,e$.
On the other hand, $\Sigma^1(s) =\{df_1\wedge \dotsb \wedge df_e=0\}$ as subschemes of $X$, by Example \ref{can-conn-locally}.
Thus $\pi^{-1}\Sigma^1(s)\subseteq X[\pthroot s]$ is the locus where the differentials of the equations defining $X[\pthroot s]$ as a subscheme of $\mathbf V(E)$ are not independent.
\end{proof}

\begin{proposition}
\label{insep-cover-normal}
Let $c$ denote the codimension of $\Sigma^1(s)$ in $X$.
\begin{enumerate}
\item $X[\pthroot s]$ is geometrically integral if, and only if, $c\ge 1$.
\item $X[\pthroot s]$ is geometrically normal if, and only if, $c\ge 2$.
\end{enumerate}
\end{proposition}

\begin{proof}
Let $\overline k$ be an algebraic closure of $k$.
Note that $X[\pthroot s]_{\bar k}$ is homeomorphic to $X_{\bar k}$, hence irreducible.
Being a local complete intersection, $X[\pthroot s]_{\bar k}$ has Serre's property $S_l$ for every $l\ge 0$. In particular, $X[\pthroot s]_{\bar k}$ is reduced if, and only if, it is regular in codimension zero; and $X[\pthroot s]_{\bar k}$ is normal if, and only if, it is regular in codimension 1.
\end{proof}

Let $C$ denote the cokernel of the $\O_X$-linear map $F_X^* E^\vee \to \Omega_X$ induced by the covariant derivative $\nabla s$.

\begin{proposition}
\label{omega-insep-cov}
Let $Y := X[\pthroot s]$.
The sheaf of differentials of $Y$ sits in a short exact sequence 
\begin{equation*}
\begin{tikzcd}
0 \ar[r] &
\pi^*C \ar[r,"\alpha"] &
\Omega_Y \ar[r] &
\pi^* E^\vee \ar[r] &
0,
\end{tikzcd}
\end{equation*}
where the map $\alpha$ is induced by the co-differential $d\pi^\vee : \pi^* \Omega_X\to \Omega_Y$.
\end{proposition}

\begin{proof}
Let $I\subseteq \O_V$ denote the ideal sheaf of $Y$ in $V := \mathbf V(E)$.
Let $d_Y : I/I^2\to (\Omega_V)_Y$ be the $\O_Y$-linear map appearing in the co-normal sequence of $Y$ in $V$, so that $d_Y( \bar g ) = dg$ for all $g\in I$. 

Pairing with the section $F_V^* \tau - \pi^* s$ induces a surjective $\O_V$-linear map $(F_X^* E)_V^\vee\to I$. Let $\theta : \pi^* F_X^* E^\vee \to I/I^2$ be the corresponding $\O_Y$-linear map. 
Then $\theta$ is a surjection of locally free sheaves of the same rank $e$, hence an isomorphism.

By the second sentence in the proof of Proposition \ref{insep-cov-sing}, the following diagram is commutative:
\begin{equation*}
\begin{tikzcd}
0 \ar[r] &
\pi^*F_X^* E^\vee \ar[r,"\theta","\sim"'] \ar[d, "-\nabla s"] &
I/I^2 \ar[r] \ar[d, "d_Y"] &
0 \ar[r] \ar[d] &
0 \\
0 \ar[r] &
\pi^*\Omega_X \ar[r,"d\pi^\vee"] &
(\Omega_V)_Y \ar[r] &
(\Omega_{V/X})_Y \ar[r] &
0
\end{tikzcd}
\end{equation*}
Its bottom row is exact because $V$ is a vector bundle over $X$.
Noting the canonical isomorphism $E_V^\vee =\Omega_{V/X}$ and applying the Snake Lemma yields the desired short exact sequence.
\end{proof}

\begin{proposition}
\label{insep-C-loc-free}
The complement $X\setminus \Sigma^1(s)$ is the largest open subset of $X$ over which $C$ is locally free of rank $n-e$ and the sequence 
\begin{equation*}
\begin{tikzcd}
0 \ar[r] &
F_X^*E^\vee \ar[r, "\nabla s"] &
\Omega_X \ar[r] & 
C \ar[r] &
0
\end{tikzcd}
\end{equation*}
is exact.
\end{proposition}

\begin{proof}
Set-theoretically, the critical locus $\Sigma^1(s)$ is the subset of $X$ over which the map $\T_X \to F_X^* E$ induced by $\nabla s$ has less-than-full rank.
This implies that $X\setminus \Sigma^1(s)$ is the subset of $X$ over which $C$ has constant rank $n-e$, since the rank of matrix is equal to that of its transpose.
The scheme $X$ is reduced, so the result follows.
\end{proof}

Let $Q$ denote the double dual of the $\O_X$-module $\wedge^{n-e}C$.

\begin{proposition}
\label{this-is-Q}
Let $c$ denote the codimension of $\Sigma^1(s)$ in $X$.
\begin{enumerate}
\item If $c\ge 1$, then $Q$ is invertible.
\item If $c\ge 2$, then $Q \cong \omega_X \otimes (\det E)^{\otimes p}$.
\end{enumerate}
\end{proposition}

\begin{proof}
By Proposition \ref{insep-C-loc-free},
\begin{equation*}
Q = \omega_X\otimes \det (F_X^* E) = \omega_X \otimes (\det E)^{\otimes p}
\end{equation*}
over $X\setminus \Sigma^1(s)$. Both statements follow, since reflexive sheaves of rank 1 on regular schemes are invertible and are determined by their restrictions to open subsets whose complements have codimension at least 2.
\end{proof}

\begin{construction}
\label{map-from-Q}
Suppose $\Sigma^1(s)$ has codimension at least 1 in $X$, so that $Y:= X[\pthroot s]$ is integral and generically smooth, and $Q$ is an invertible $\O_X$-module.
By Proposition \ref{omega-insep-cov}, the co-differential $d\pi^\vee : \pi^* \Omega_X\to \Omega_Y$ factors through an injective $\O_Y$-linear map $\alpha : \pi^* C\to \Omega_Y$.
Taking duals of coherent sheaves commutes with flat pullback, so $\alpha$ induces an $\O_Y$-linear map
\begin{equation*}
\pi^*Q = (\wedge^{n-e} \pi^* C)^{\vee\vee} \to (\Omega_Y^{n-e})^{\vee\vee}.
\end{equation*}
Composing with the natural map
$
(\Omega_Y^{n-e})^{\vee\vee} \to \Omega_Y^{n-e}\otimes k(Y),
$
yields an $\O_Y$-linear map
\begin{equation*}
\pi^* Q \to \Omega_Y^{n-e}\otimes k(Y),
\end{equation*}
which is injective at the generic point of $Y$, hence everywhere.
\end{construction}

\section{Remarks on torsion-free quotients}

Although we are primarily interested in integral schemes, we record in this section observations that apply to arbitrary rings. This will allow us to comfortably work with completed local rings later.

If $R$ is a ring, we denote by $\Frac R$ its total ring of fractions. Thus $\Frac R$ is the localization $R$ at the multiplicative system consisting of the nonzerodivisors in $R$.

\begin{definition}
\label{torsion-submodule}
Let $R$ be a ring. Let $M$ be an $R$-module.
The \textbf{torsion submodule} of $M$, denoted $M_\mathrm{tors}$, is the kernel of the natural map $M\to M\otimes_R \Frac R$.
\end{definition}

\begin{remark}
\label{torsion-ring-map}
Let $\phi:R\to S$ be ring map.
Suppose that $\phi$ sends nonzerodivisors of $R$ to nonzerodivisors of $S$, which holds for example if $\phi$ is flat. Then $\phi$ induces a natural map $\Frac R\to \Frac S$, hence a diagram as follows for each $R$-module $M$.
\begin{equation*}
\begin{tikzcd}
0 \arrow{r}& \arrow{r} \arrow["\alpha"]{d} (M_\mathrm{tors}) \otimes_R S
&\arrow{r} \arrow[equals]{d} M\otimes_R S
& \arrow["\beta"]{d}(M\otimes_R \Frac R)\otimes_R S\\
0 \arrow{r}& (M\otimes_R S)_\mathrm{tors} \arrow{r}& \arrow{r}M \otimes_R S & (M\otimes_R S)\otimes_S \Frac S
\end{tikzcd}
\end{equation*}
\end{remark}

\begin{lemma}
\label{tf-ring-change}
Notation as in Remark \ref{torsion-ring-map}.
Suppose that $M\otimes_R \Frac R$ is flat over $\Frac R$. Then $\beta$ is injective. Hence, $\alpha$ is surjective and
\begin{equation*}
(M/M_\mathrm{tors})\otimes_R S = (M\otimes_R S)/(M\otimes_R S)_\mathrm{tors}.
\end{equation*}
\end{lemma}

Thus, if $R$ is a domain, then $\alpha$ is surjective. This can fail in general, see Example \ref{not-surjective} below.

\begin{proof}
The natural map $S\to \Frac S$ is injective by the definition of a nonzerodivisor. Taking the tensor product of this map with the flat $R$-module $\Frac R$, we obtain an injection
\begin{equation*}
\Frac(R)\otimes_R S \to \Frac(R) \otimes_R \Frac(S) = \Frac(S).
\end{equation*}
Taking the tensor product of this injection with the flat $\Frac(R)$-module $M\otimes_R \Frac(R)$, we recover the map $\beta$. Thus $\beta$ is injective, which implies by a diagram chase that $\alpha$ is surjective, as desired.
\end{proof}

\begin{example}
\label{not-surjective}
Let $k$ be a field. Let $x,y,z$ be indeterminates. Suppose that $R$ is the localization of the ring $k[x,y,z]/\langle xz,yz,z^2\rangle$ at the maximal ideal $\langle x,y,z\rangle$.
Suppose that $S=R_x$ and that $\phi:R\to S$ is the localization map, which is flat. Finally, suppose that $M=R/\langle y\rangle$.

Then $M_\mathrm{tors}=0$. Indeed, the unique maximal ideal of $R$ is an associated prime of $R$, hence consists entirely of zerodivisors. This implies that the natural map $R\to \Frac R$ is an isomorphism.

But $(M\otimes_R S)_\mathrm{tors}\ne 0$. To see this, we first note that $S=R_x = (k[x,y]_{\langle x,y\rangle})_x$, which implies that $\Frac S = k(x,y)$. Thus $M\otimes_R \Frac S=0$, while
\begin{equation*}
M\otimes_R S = (R/\langle y\rangle)_x =(k[x,z]/\langle xz,z^2\rangle_{\langle x,z\rangle})_x = (k[x]_{\langle x\rangle})_x \ne 0.
\end{equation*}
\end{example}

The following result will be used in the proof of Proposition \ref{inclusion-Q} below.

\begin{lemma}
\label{torsion-descent-I}
Let $p:P\to M\otimes_R \Frac R$ be a map of $R$-modules.
Let $p'$ denote the composition
\begin{equation*}
P\otimes_R S\xrightarrow{p\otimes 1} (M\otimes_R \Frac R )\otimes_R S \xrightarrow\beta M\otimes_R \Frac S.
\end{equation*}
Suppose that $\phi:R\to S$ is faithfully flat and that $M\otimes_R \Frac R$ is flat over $\Frac R$. Then $p$ factors through $M/M_\mathrm{tors}$ if, and only if, $p'$ factors through $(M\otimes_R S)/(M\otimes_R S)_\mathrm{tors}$.
\end{lemma}

\begin{proof}
Write $N=M/M_\mathrm{tors}$ and $N'=M\otimes_R \Frac R$, and let $\iota: N\to N'$ denote the natural inclusion.
To prove the nontrivial implication, suppose that $p'$ factors through $(M\otimes_R S)/(M\otimes_R S)_\mathrm{tors}$. Then $p\otimes 1$ factors through a map $\pi':P\otimes_R S\to N\otimes_R S$, by Lemma \ref{tf-ring-change}.
\begin{equation*}
\begin{tikzcd}
P\otimes_R S \ar["p\otimes 1"]{r} \ar["\pi'"',dashed]{rd} & 
N'\otimes_R S \ar["\beta"]{r}&
M\otimes_R \Frac S \\
&
N\otimes_R S \ar[hook, "\iota\otimes 1"']{u} \ar["\sim"]{r} &
(M\otimes_R S)/(M\otimes_R S)_\mathrm{tors} \ar[hook]{u}
\end{tikzcd}
\end{equation*}

Any ring map $S\to T$ induces a natural functor
\begin{equation*}
-\otimes_S T : (S\text{-modules}) \to (T\text{-modules}).
\end{equation*}
Applying to $\pi'$ the functors induced by the $R$-algebra maps $S\rightrightarrows S\otimes_R S$ that send $s\mapsto s\otimes 1$ and $s\mapsto 1\otimes s$, we obtain two $S\otimes_R S$-linear maps,
\begin{equation*}
P\otimes_R S \otimes_R S \rightrightarrows N\otimes_R S \otimes_R S.
\end{equation*}
By functoriality, the composition of either map with $\iota\otimes 1 \otimes 1$ (an injection) equals $p\otimes 1 \otimes 1$. Therefore the two maps coincide. It follows by faithfully flat descent that there exists a unique map of $R$-modules $\pi:P\to N$ such that $\pi'=\pi\otimes 1$. This map gives the desired factorization of $p$, since $p\otimes 1 = (\iota\circ \pi)\otimes 1$ implies $p=\iota\circ \pi$ by faithful flatness of $\phi:R\to S$.
\end{proof}

\begin{remark}
\label{torsion-stalks}
Let $X$ be an integral scheme.
Let $F$ be a sheaf of $\O_X$-modules.
Let $x\in X$ be a point.
Then $(F_\mathrm{tors})_x = (F_x)_\mathrm{tors}$, since passing to stalks is exact.
\end{remark}

\begin{lemma}
\label{torsion-descent-II}
Let $f:Y\to X$ be a faithfully flat map of integral schemes.
Let $F$ be a sheaf of $\O_X$-modules.
Let $q:Q\to F\otimes_{\O_X} k(X)$ be a map of $\O_X$-modules.
Let $q'$ denote the composition
\begin{equation*}
f^* Q \xrightarrow{f^* q} f^* (F\otimes_{\O_X} k(X)) \to f^* F \otimes_{\O_Y} k(Y).
\end{equation*}
Then $q$ factors through $F/F_\mathrm{tors}$ if, and only if, $q'$ factors through $(f^*F) /(f^*M)_\mathrm{tors}$.
\end{lemma}

\begin{proof}
Follows from Lemma \ref{torsion-descent-I} via Remark \ref{torsion-stalks}.
\end{proof}

\section{Irrational inseparable covers}

Let $k$ be a field of positive characteristic $p$.
Let $X$ be a smooth, connected scheme of dimension $n$ over $k$.
Let $E$ be a locally free of rank $e$ on $X$.
Suppose that $e\le n$.
Let $\nabla$ denote the canonical connection on $F_X^*E$.

Let $s\in \Gamma(X,F_X^*E)$ be a section.
Let $\pi : X[\pthroot s]\to X$ be the corresponding inseparable cover of $X$.
Write $Y := X[\pthroot s]$.
Let $\rho:B\to Y$ be the blowup of the scheme-theoretic inverse image $\pi^{-1}\Sigma^1(s) \subseteq Y$. 

Let $C$ be the cokernel of the $\O_X$-linear map $F_X^* E^\vee \to \Omega_X $ induced by $\nabla s$.
Let $Q := (\wedge^{n-e}C)^{\vee\vee}$.

\begin{proposition}
\label{inclusion-Q}
Suppose that
\begin{enumerate}
\item $\Sigma^1(s)\subseteq X$ has codimension at least 2;
\item $\Sigma^2(s)\subseteq X$ is empty; and
\item $\Sigma^{1,n-e-1}(s)\subseteq X$ is empty.
\end{enumerate}
Then $Y$ is geometrically normal, $B$ is geometrically integral, and the composition of natural maps
\begin{equation*}
\rho^*\pi^*Q\hookrightarrow \rho^* \Omega_Y^{n-e}\otimes k(B) \xrightarrow\sim \Omega_B^{n-e} \otimes k(B)
\end{equation*}
(see Construction \ref{map-from-Q})
factors through the subsheaf $\Omega_B^{n-e}/(\Omega_B^{n-e})_\mathrm{tors}\subseteq \Omega_B^{n-e}\otimes k(B)$.
\end{proposition}

The proof of Proposition \ref{inclusion-Q} will be given after that of Lemma \ref{eta} below.

\begin{remark}
Let $x\in X$ be a $k$-rational point.
The canonical connection $\nabla : F_X^* E\to \Omega_X\otimes F_X^* E$ is a first-order differential operator, so induces a $k$-linear map
\begin{equation*}
(F_X^* E) \otimes \O_{X,x}/\frak m_x^{i+1} \to 
(\Omega_X\otimes F_X^* E) \otimes \O_{X,x}/\frak m_x^i
\end{equation*}
for each $i\ge 0$.
Taking inverse limits yields a $k$-linear map 
\begin{equation*}
(F_X^* E) \otimes \widehat \O_{X,x} \to 
(\Omega_X\otimes F_X^* E) \otimes \widehat \O_{X,x},
\end{equation*}
which by abuse of notation we also denote by $\nabla$.
Similarly, the universal derivation $ d: \O_X\to\Omega_X$ induces a $k$-linear map $\widehat\O_{X,x}\to \Omega_X\otimes \widehat\O_{X,x}$, which we also denote by $d$.
\end{remark}

\begin{remark}
\label{2nd-id-insep}
Let $\Sigma$ denote the critical locus $\Sigma^1(s)\setminus \Sigma^2(s)\subseteq X$.
Set-theoretically, $\Sigma$ is the locus in $X$ where the covariant derivative $\nabla s : \T_X\to F_X^*E$ has constant rank $e-1$.
The second intrinsic differential
\begin{equation*}
\mathbf d_\Sigma^2 s : \ker((\nabla s)_\Sigma) \to \sheafHom_\Sigma(\ker((\nabla s)_\Sigma),\coker((\nabla s)_\Sigma)).
\end{equation*}
is a map of locally free $\O_\Sigma$-modules of rank $n-e+1$.
Set-theoretically, the singular locus $\Sigma^{1,n-e-1}(s)$ is the locus in $\Sigma$ where $\mathbf d_\Sigma^2 s$ has rank at most
\begin{equation*}
n-e+1 - (n-e-1)=2.
\end{equation*}

Note that $\Sigma^{1,n-e-1}(s)$ can only be empty if $n-e+1 \ge 3$, since the rank of a matrix is no larger than its dimensions.
If $p=2$, the second intrinsic differential $\mathbf d_\Sigma^2 s$ is skew-symmetric by Proposition \ref{symmetry-2nd-id}, hence has even rank.
In this case, $\Sigma^{1,n-e-1}(s)$ can only be empty if $n-e+1 \ge 4$.
\end{remark}

\begin{remark}
\label{cork-1-insep}
Let $x\in X$ be a $k$-rational point.
Suppose that $x\not\in \Sigma^2(s)$.
Then there exist
\begin{enumerate}
\item an isomorphism of local $k$-algebras  
\begin{equation*}
\widehat\O_{X,x} \cong k[[x_1,\dotsc,x_n]];
\end{equation*} 
\item a basis $\{v_1,\dotsc,v_e\}$ of $E\otimes \widehat \O_{X,x}$ as an $\widehat\O_{X,x}$-module;
\item constants $c_1,\dotsc,c_{e-1}\in k$; and
\item a power series $f\in k[[x_1,\dotsc,x_n]]$
\end{enumerate}
such that
\begin{equation*}
s = (c_1 + x_1)\cdot F_X^*v_1 + \dotsb + (c_{e-1} + x_{e-1})\cdot F_X^*v_{e-1} + f \cdot F_X^* v_e,
\end{equation*}
in $F_X^* E\otimes \widehat\O_{X,x}$, see Corollary \ref{local-coords-cor}.
\end{remark}

\begin{remark}
\label{morse-corank-1-insep}
Let $\Sigma$ denote the critical locus $\Sigma^1(s)\setminus \Sigma^2(s)\subseteq X$.
Let $x\in \Sigma$ be a $k$-rational point.
Suppose that $x\not\in \Sigma^{1,n-e-1}(s)$, so that the second intrinsic differential $\mathbf d_\Sigma^2 s$ has rank at least 3 at $x$.
Then the isomorphism, basis, constants and power series of Remark \ref{cork-1-insep} may be chosen so that $f= q+h$, where 
\begin{equation*}
q =
\begin{cases}
x_e^2 + x_{e+1}^2 + x_{e+2}^2 &\text{if }\operatorname{char}(k)\ne 2\\
x_ex_{e+1} + x_{e+2}x_{e+3} &\text{if }\operatorname{char}(k)= 2.
\end{cases}
\end{equation*}
and $h\in k[[x_1,\dotsc,x_n]]$ is a power series that does not involve the variables occurring in $q$, see  Proposition \ref{morse-corank-1}.
\end{remark}

\begin{lemma}
\label{eta}
Let $x\in X$ be a $k$-rational point.
Suppose that $x\not\in \Sigma^2(s)$.
Fix a choice of isomorphism, basis, constants and power series as in Remark \ref{cork-1-insep}.
Let $X' = \Spec \widehat\O_{X,x}$.
Then
\begin{equation*}
\Sigma^1(s) \cap X' =
\left\{
\dfrac{\partial f}{\partial x_e} = 
\dotsb =
\dfrac{\partial f}{\partial x_n} = 0
\right\}
\end{equation*}
as subschemes of $X'$.
Suppose that $\Sigma^1(s)\cap X'$ has codimension at least 2 in $X'$.
Let $i\in \{e,\dotsc,n\}$ be such that $\partial f/\partial x_i\ne 0$.
Let
\begin{equation*}
\eta_i := \dfrac{dx_e\dotsb\widehat{dx_i}\dotsb dx_n}{\partial f/\partial x_i}\in \Omega_X^{n-e}\otimes \Frac \widehat\O_{X,x}.
\end{equation*} 
Then the image of $\eta_i$ in $Q\otimes \Frac \widehat\O_{X,x}$ lies inside, and is a generator of, the $\widehat \O_{X,x}$-module $Q\otimes \widehat\O_{X,x}$.
\end{lemma}

\begin{proof}
Note that
\begin{equation*}
C_{X'} = \coker(\nabla s)_{X'} = (\Omega_X)_{X'}/\langle dx_1,\dotsc,dx_{e-1},df\rangle. 
\end{equation*}
Let $\mathscr S$ be the set of integers $j$ such that
$e\le j\le n$,
$j\ne i$, and 
$\partial f/\partial x_j\ne 0$.
For each $j\in \mathscr S$, the differential form
\begin{equation*}
df \wedge (dx_e\dotsb \widehat{ dx_i }\dotsb\widehat{ dx_j }\dotsb dx_n) \in (\wedge^{n-e}\Omega_X)\otimes \widehat\O_{X,x}
\end{equation*}
maps to zero in $(\wedge^{n-e}C)\otimes \widehat\O_{X,x}$.
In other words,
\begin{equation*}
(\partial f/\partial x_j) \cdot dx_e\dotsb\widehat{dx_i}\dotsb dx_n= \pm
(\partial f/\partial x_i) \cdot dx_e\dotsb\widehat{dx_j}\dotsb dx_n,
\end{equation*}
in $(\wedge^{n-e}C)\otimes \widehat\O_{X,x}$, so that $\eta_i = \pm \eta_j$ in $Q\otimes \Frac \widehat\O_{X,x}$, for all $j\in \mathscr S$.

Clearly,
\begin{equation*}
\Sigma' := \Sigma^1(s) \cap X' = 
\{ dx_1\wedge\dotsb \wedge dx_{e-1}\wedge df =0 \},
\end{equation*}
which implies the claimed description of $\Sigma'$. From that description and the fact that $\eta_i=\pm \eta_j$ in $Q\otimes \Frac \widehat\O_{X,x}$, it follows that $\eta_i$ defines a regular, generating section of $Q$ over $X'\setminus \Sigma'$.
The result follows as $\Sigma'$ has codimension 2 in $X'$, see  \cite[Proposition 6.3A]{hartshorne77}. 
\end{proof}

\begin{proof}[Proof of Proposition \ref{inclusion-Q}]
The inseparable cover $Y$ is geometrically normal by Proposition \ref{insep-cover-normal}.
Blowups of integral schemes are again integral, so $B$ is geometrically integral.

By Lemma \ref{torsion-descent-II}, it suffices to show the factoring morphism
\begin{equation*}
\rho^* \pi^* Q\dashrightarrow \Omega_B^{n-e}/(\Omega_B^{n-e})_\mathrm{tors}\hookrightarrow \Omega_B^{n-e}\otimes k(B)
\end{equation*}
exists on a faithfully flat cover of $B$.
Hence we may assume that $k$ is algebraically closed.
It suffices to show the factoring morphism exists after pulling back to the spectrum of $\O_{B,b}$ for all closed points $b\in B$. In fact, by Lemma \ref{torsion-descent-I}, it suffices to show that
\begin{equation*}
Q \otimes \widehat \O_{B,b}\to \Omega_B^{n-e} \otimes \Frac \widehat \O_{B,b}
\end{equation*}
maps into the image of $\Omega_B^{n-e} \otimes \widehat \O_{B,b}$ for all closed points $b\in B$.

Let $b\in B$ be a closed point. Write $y:=\rho(b)$ and $x:=\pi(y)$.
The existence of the factoring morphism is clear if $x\in X\setminus \Sigma^1(s)$, since the natural map $\wedge^{n-e}C \to Q$ induces an isomorphism of stalks over those points.
Assume that $x\in \Sigma^1(s)$ and fix a choice of isomorphism, basis, constants and power series as in Remark \ref{morse-corank-1-insep}.

Let $Y':=\Spec \widehat\O_{Y,y}$. 
Let $B'$ denote the base change of $B\to Y$ to $Y'$. Blowing up commutes with flat pullback, so $B'$ is the blowup of $Y'$ at the closed subscheme $Y'\cap \pi^{-1}\Sigma^1(s)$.
The ideal of this subscheme of $Y'$ is
\begin{equation*}
I:=\langle \partial_e f,\dotsc, \partial_n f\rangle \subseteq \widehat \O_{Y,y},
\end{equation*}
where $\partial_i f$ denotes the partial derivative $\partial f/\partial x_i$ for each $i=e,\dotsc, n$.
Let $\mathscr S$ be the set of integers $i$ such that $e\le i\le n$ and $\partial_i f\ne 0$ in $\widehat\O_{X,x}$. Then the map of graded $\widehat \O_{Y,y}$-algebras
\begin{equation*}
\widehat\O_{Y,y}[\{T_i\}_{i\in \mathscr S}] / \langle T_i \partial_j f - T_j \partial_i f\rangle_{i,j\in \mathscr S} \to \mathrm{Rees}(I):=
\widehat\O_{Y,y} \oplus I \oplus I^2 \oplus I^3 \oplus \dotsb
\end{equation*}
that sends $T_i$ to $\partial_i f$ (in degree 1) for each $i\in \mathscr S$ is well defined and surjective.
It follows that 
\begin{equation*}
B' =
\bigcup_{i\in \mathscr S} D_+(T_i),
\end{equation*}
where $D_+(T_i)\subseteq B'$ is an open subset isomorphic over $Y'$ to a closed subscheme of
\begin{equation*}
\Spec \widehat\O_{Y,y}[\{T_j\}_{j\in \mathscr S, j\ne i}] / \langle \partial_j f - T_j \partial_i f\rangle_{j\in \mathscr S, j\ne i}
\end{equation*}
for each $i\in \mathscr S$.

Let $b'\in B'$ be the unique point that lies over the closed point of $Y'$ and maps to $b\in B$. Then the completion of the local ring of $B'$ at $b'$ is naturally isomorphic to $\widehat \O_{B,b}$. Note that each of the natural maps
\begin{equation*}
\widehat \O_{X,x} \xrightarrow{\alpha_1}
\widehat \O_{Y,y} \xrightarrow{\alpha_2}
\O_{B',b'} \xrightarrow{\alpha_3}
\widehat \O_{B,b}
\end{equation*}
sends nonzerodivisors to nonzerodivisors. Indeed, $\alpha_1$ and $\alpha_3$ are flat, and if $A$ is a ring and $J$ is an ideal, it is easily checked that nonzerodivisors of $A$ map to nonzerodivisors of $\operatorname{Rees}(J)$. There exists therefore a natural map
\begin{equation*}
\Frac \widehat \O_{X,x} \to \Frac \widehat \O_{B,b}.
\end{equation*}

Let $i\in\mathscr S$ be such that $b'\in D_+(T_i)$.
To complete the proof, we now show that the rational differential form $\eta_i$ of Lemma \ref{eta} maps into the image of $\Omega_B^{n-e}\otimes \widehat \O_{B,b}$ under the natural map
\begin{equation*}
\Omega_X^{n-e}\otimes \Frac \widehat\O_{X,x}\to 
\Omega_B^{n-e}\otimes \Frac \widehat\O_{B,b}.
\end{equation*}

If $\operatorname{char}(k)\ne 2$, then
\begin{equation*}
2x_j = \partial_j f = T_j \partial_i f
\end{equation*}
in $\Gamma(D_+(T_i),\O_{B'})$ for all $j\in\{e, e+1, e+2\}\setminus\{i\}$. Similarly, if $\operatorname{char}(k)=2$, then $x_j$ is divisible by $\partial_i f$ in $\Gamma(D_+(T_i),\O_{B'})$ for all $j\in \{e,e+1,e+2,e+3\}\setminus\{i\}$. For this reason, the result is a consequence of the following observation.

Let $u_1,u_2,g\in \widehat \O_{B,b}$ be elements such that $g$ is a nonzerodivisor and divides $u_1$ and $u_2$. Then
\begin{equation*}
\dfrac{du_1\wedge du_2}{g}\in \Omega_B^2 \otimes \Frac \widehat \O_{B,b}
\end{equation*}
lies in the image of $\Omega_B^2 \otimes \widehat \O_{B,b}$. Indeed, if $u_1=v_1 g$ and $u_2=v_2 g$, then
\begin{align*}
du_1\wedge du_2 &=
(v_1 dg + g dv_1)\wedge( v_2 dg + g dv_2 )\\
&= v_1 g \cdot dg\wedge dv_2 + g v_2 \cdot dv_1\wedge dg + g^2 \cdot dv_1\wedge dv_2. \qedhere
\end{align*}
\end{proof}

\begin{definition}
\label{Exceptions}
Given a prime number $q$, we define a finite set $\mathscr E_q$ of  pairs of integers as follows. 
We let
\begin{align*}
\mathscr E' = \{&
(1, 2),
(2, 3),
(2, 4),
(3, 4),
(3, 5),
(4, 5),
(4, 6),
(4, 7),
(5, 7),\\
&(5, 8),
(6, 9),
(7, 11)\},\\
\mathscr E'' =\{&
(1, 3),
(2, 5),
(3, 6),
(4, 8),
(5, 9),
(6, 10),
(7, 12),
(8, 13)\}
\end{align*}
and set
\begin{equation*}
\mathscr E_q =\begin{cases}
\mathscr E' &\text{if }q\ne 2\\
\mathscr E' \sqcup \mathscr E'' &\text{if }q=2.
\end{cases}
\end{equation*}
\end{definition}

\begin{theorem}
\label{irrational-insep-covers}
Let $k$ be an infinite field of characteristic $p>0$. Let $X$ be a scheme over $k$. Let $E$ be an locally free sheaf of rank $e$ on $X$. Let $W\subset \Gamma(X,F_X^* E)$ be a $k$-linear subspace of finite dimension. 
Suppose that
\begin{enumerate}
\item $X$ is smooth, proper, connected and of dimension $n$;
\item $\omega_X\otimes \det(E)^{\otimes p}$ is a big invertible sheaf on $X$; 
\item $d^2 (W)$ generates $\PP_X^2(F_X^* E)$ as an $\O_X$-module; and
\item $e\le n-1$, $e\le \tfrac 1 2 (n+3)$ and $(e,n)\not\in \mathscr E_p$.
\end{enumerate}
Let $s\in W$ be a general section.
Then $X[\pthroot s]_{\bar k}$ is integral, normal and not separably uniruled.
\end{theorem}

\begin{proof}
Let $\delta := n-e$ and 
\begin{equation*}
C := \delta + 1 + \tfrac 1 2 (\delta-1)(\delta-1\pm 1),
\end{equation*}
where the symbol ``$\pm$'' should be read as ``plus'' if $k$ has characteristic different from 2, and ``minus'' if $k$ has characteristic 2.
By Corollaries \ref{1st-sings-generic} and 
\ref{generic-sigma-i-j}, the hypotheses of Proposition \ref{inclusion-Q} are satisfied for general $s\in W$ provided that
\begin{equation*}
n-e+1 \ge 2, \qquad
2(n-e+2) > n \qquad\text{and}\qquad
C > n.
\end{equation*}
The first two inequalities are satisfied if, and only if, $e\le n-1$ and $e\le \tfrac 1 2(n+3)$.
In this case, the third inequality is satisfied if, and only if,  $(r,n)\not\in \mathscr E_p$.
There exists therefore an injective map $\rho^* \pi^* Q\hookrightarrow \Omega_B^{n-e}/(\Omega_B^{n-e})_\mathrm{tors}$, where $Q\cong \omega_X\otimes (\det E)^{\otimes p}$ by Proposition \ref{this-is-Q}.
From Lemma \ref{not-uniruled} it follows that $B$ is not separably uniruled.
This implies the result, since $B$ is birationally equivalent to $X[\pthroot s]$.
\end{proof}

\section{Irrational complete intersections}

In this section we reduce Theorem \ref{MainResultIrrationalCIs}, our main result about complete intersections in characteristic zero, to Theorem \ref{irrational-insep-covers}, which concerns insparable covers in positive characteristic. This is made possible by the following result of Matsusaka's.

\begin{theorem}[{\cite[p. 233]{Matsusaka68}, \cite[Theorem IV.1.8.3]{Kollar1996}}]
\label{Matsusaka}
Let $f:Z\to S$ be a morphism of schemes. Suppose that $S$ is excellent and that $f$ is flat, proper and has geometrically integral fibers. Then there exist coutably many closed subsets $R_i\subset S$ such that $Z_s$ is geometrically ruled if, and only if, $s\in \bigcup_i R_i$.
\end{theorem}

\begin{remark}
\label{exhibit-one}
In order to show that the very general complete intersection of a given multi-degree is not geometrically ruled, it suffices to exhibit a single complete intersection of that multi-degree that is (geometrically integral, but) not geometrically ruled. Indeed, if the fiber of $f:Z\to S$ over $s\in S$ is not geometrically ruled, then $s\not \in R_i$ for all $i$, so that $\operatorname{Supp} R_i\ne S$ for all $i$.
\end{remark}

\begin{remark}
\label{MatsusakaDVR}
Suppose that $S$ is the spectrum of an excellent discrete valuation ring. Then the conclusion of Theorem \ref{Matsusaka} may be equivalently stated as the following implication: If the special fiber of $f:X\to S$ is not geometrically ruled, then the generic fiber isn't either. To see this, note that $S$ has two points. Its only closed subsets are the empty set, the singleton consisting of the closed point, and $S$ itself. Any closed subset $R_i\subset S$ which does not contain the closed point is therefore empty.
\end{remark}

\begin{construction}
\label{CyclicCovers}
Let $X$ be a scheme. Let $L$ be an invertible sheaf on $X$. Let $p\ge 1$ be an integer. Let $s\in\Gamma(X,L^{\otimes p})$ be a section.
Let $V=\mathbf V(L)$ be the vector bundle associated to $L$. Let $\pi:V\to X$ denote the projection. Let $\tau\in \Gamma(V,\pi^*L)$ be the tautological section, see Definition \ref{tautological-section}. The \textbf{cyclic cover} of $X$ defined by $s$ is the restriction of $\pi$ to the subscheme
\begin{equation*}
X[\displayroot p s] := \{ \tau^{\otimes p} = s\}\subseteq V.
\end{equation*}
\end{construction}

We now describe a variant of the hypersurface degeneration introduced in \cite[Example 4.3]{Mori1975}.

\begin{construction}
\label{MoriDegeneration}
Let $R$ be a discrete valuation ring. Let $K$ denote the fraction field of $R$. Let $t\in R$ be a uniformizer, that is, a generator of the maximal ideal of $R$. Let $k$ denote the residue field of $R$.

Let $c$ be an integer such that $1\le c\le N$. For each $i=1,\dotsc,c$, let $f_i, g_i\in R[x_0,\dotsc,x_N]$ be homogeneous elements with $\deg f_i$ divisible by $\deg g_i$, and write
\begin{equation*}
a_i := \deg g_i \qquad\text{and}\qquad
p_i := \deg f_i / \deg g_i.
\end{equation*}
Let $\rho:V\to \P^N_R$ denote the vector bundle associated to the locally free sheaf $\oplus_{i=1}^c \O(a_i)$. Let 
\begin{equation*}
\tau = (\tau_1,\dotsc,\tau_c)\in \Gamma(V,\rho^* \oplus_{i=1}^c \O(a_i))
\end{equation*}
denote the tautological section. Let $Z$ denote the subscheme of $V$ defined by the vanishing of the sections
\begin{equation*}
\tau_i^{\otimes p_i}- f_i \in \Gamma(V,\rho^* \O(p_ia_i))
\qquad\text{and}\qquad
t\tau_i - g_i \in \Gamma(V,\rho^* \O(a_i))
\end{equation*}
for all $i=1,\dotsc, c$. Let $\pi$ denote the composition
\begin{equation*}
Z \hookrightarrow V \xrightarrow\rho \P^N_R \to \Spec R. 
\end{equation*}
\end{construction}

\begin{lemma}
Notation as in Construction \ref{MoriDegeneration}. 
\begin{enumerate}
\item The morphism $\pi:Z\to \Spec R$ is proper.
\item If the special fiber $Z_k$ has dimension $N-c$, then the generic fiber $Z_K$ has also dimension $N-c$, and $\pi$ is flat.
\item If furthermore $Z_k$ is geometrically integral, then $Z_K$ is geometrically integral.
\end{enumerate}
\label{AboutDegeneration}
\end{lemma}

\begin{proof}
Let $P$ denote $\P^N_R$. 
Note that the closed immersion $Z\hookrightarrow V$ factors through the inclusion
\begin{equation*}
P[\displayroot {p_1} {f_1}]\times_P \dotsb \times_P P[\displayroot {p_c} {f_c}] \subset \mathbf V(\O(a_1))\times_P \dotsb\times_P \mathbf V(\O(a_c)) = V.
\end{equation*}
This implies (1) since each of the morphisms $P[\displayroot {p_i} {f_i}]\to P$ is finite.

If $\dim Z_k = N-c$, then $\dim Z_K \le N-c$ by upper-semicontinuity of fiber dimension (see \cite[Corollaire 13.1.5]{EGA}). This implies $\dim Z \le N-c+1$. To get the opposite inequality, note that $Z$ is defined by the vanishing of $2c$ sections of invertible sheaves on $V$, which is a regular scheme of pure dimension $N+c+1$. It follows that $\dim Z_K = N-c$ and that $Z$ is Cohen-Macaulay. By ``miracle flatness'' \cite[Theorem 23.1]{Matsumura86}, part (2) follows.

The property of geometric integrality of fibers is open on the base of a flat, proper and finitely presented morphism of schemes (see \cite[Th\'eor\`eme 12.2.4]{EGA}). This implies (3).
\end{proof}

\begin{remark}
Let $\pi:Z\to \Spec R$ be as in Construction \ref{MoriDegeneration}. Suppose that the residue field $k$ has positive characteristic $p$ and that $p_i=p$ for all $i=1,\dotsc, c$.

Let $X$ denote the subscheme
\begin{equation*}
\{g_1=\dotsb=g_c=0\}\subseteq \P^N_k.
\end{equation*}
Let $F_X:X\to X$ denote the absolute Frobenius morphism of $X$. Let $E=\oplus_{i=1}^c \O_X(a_i)$. Then $(f_1,\dotsc,f_c)$ induces a global section $s$ of $(F_X)^* E =\oplus_{i=1}^c \O_X(pa_i)$.

The special fiber $Z_k$ is isomorphic as a $k$-scheme to $X[\pthroot s]$, the inseparable cover of $X$ determined by $s$. The generic fiber $Z_K$, on the other hand, is isomorphic to the subscheme
\begin{equation*}
\{ g_1^p - t^p f_1 = \dotsb = g_c^p - t^p f_c =0 \} \subseteq \P^N_K.
\end{equation*}

Suppose that $X$ is a complete intersection in $\P_k^N$. Then $\pi:Z\to \Spec R$ is flat and $\dim Z_K=N-c$ by Lemma \ref{AboutDegeneration}. Thus $\pi$ is a degeneration of complete intersection over $K$ to an inseparable cover of a complete intersection of smaller multi-degree over $k$.
\end{remark}

\begin{proposition}
Let $N,d_1,\dotsc,d_c, p$ be positive integers. Suppose that
\begin{enumerate}
\item $p$ is a common prime factor of $d_1,\dotsc,d_c$;
\item $\sum_{i=1}^c d_i > \frac{p}{p+1}(N+1)$; and 
\item $c\le \tfrac 1 2 N-1$, $c\le \frac 1 3 N +1$ and $(c,N-c)\not\in \mathscr E_p$, the finite set of Definition \ref{Exceptions}.
\end{enumerate}
Then a complete intersection of $c$ very general hypersurfaces of degrees $d_1,\dotsc,d_c$ in $N$-dimensional complex projective space is not ruled.
\label{IrrationalCIs}
\end{proposition}

\begin{proof}
Let $s$ be an indeterminate. Let $\frak p$ denote the prime ideal $p\Z[s] \subset \Z[s]$. Thus $\frak p$ corresponds to the generic point of the fiber over $p$ of the natural map
\begin{equation*}
\Spec \Z[s]\to \Spec \Z.
\end{equation*}
Let $R=\Z[s]_\frak p$. Then $R$ is an excellent discrete valuation ring whose residue field $k=(\Z/(p))(s)$ is infinite of characteristic $p$ and whose fraction field $K = \Q(s)$ embeds in the field of complex numbers. 

Let $\bar g_1,\dotsc,\bar g_c\in k[x_0,\dotsc,x_N]$ be homogeneous polynomials of degrees $d_1/p,\dotsc, d_c/p$ such that 
\begin{equation*}
X := \{\bar g_1 =\dotsb =\bar g_c =0\}\subset \P^N_{k}
\end{equation*}
is a smooth complete intersection over $k$.

Let $E=\oplus_{i=1}^c \O_X(d_i/p)$. Then $F_X^* E= \oplus_{i=1}^c \O_X(d_i)$. Let $W$ denote the image of the natural $k$-linear map
\begin{equation*}
\oplus_{i=1}^c \Gamma(\P_k^N, \O_{\P^N}(d_i)) \to \Gamma(X,F_X^* E).
\end{equation*}
Let $\bar k$ be an algebraic closure of $k$.
Because $d_i\ge 2$ for $i=1,\dotsc,c$, the natural map
\begin{equation*}
W\otimes_k \bar k\to E_{X_{\bar k}}/\frak m_x^3 E_{X_{\bar k}}
\end{equation*}
is surjective for every closed point $x\in X_{\bar k}$. Note that the invertible sheaf
\begin{align*}
\omega_X \otimes (\det E)^{\otimes p} = \omega_X ( \textstyle\sum d_i ) = \O_X(-N-1 + (\tfrac 1 p + 1)\textstyle\sum d_i)
\end{align*}
is big because $\sum d_i > (N+1)/(1+\tfrac 1 p)$ by hypothesis.

It follows from Proposition \ref{irrational-insep-covers} that there exist homogeneous polynomials
$\bar f_1,\dotsc,\bar f_c\in k[x_0,\dotsc,x_N]$
of degrees $d_1,\dotsc,d_c$ with the following property: Let $s$ denote the image of $(\bar f_1,\dotsc,\bar f_c)$ in $\Gamma(X,F_X^* E)$. Then $X[\pthroot s]$ is geometrically integral and not separably uniruled.

Let $\pi: Z\to \Spec R$ be the morphism obtained by applying Construction \ref{MoriDegeneration} to a choice of homogeneous lifts 
\begin{equation*}
g_1,\dotsc,g_c,f_1,\dotsc, f_c\in R[x_0,\dotsc,x_N]
\end{equation*}
of the $\bar g_i$ and $\bar f_i$. Then $\pi$ is equidimensional and flat with geometrically integral fibers, and $Z_k \cong X[\pthroot s]$ is not ruled.

Applying Matsusaka's Theorem \ref{Matsusaka}, we obtain that the generic fiber of $\pi$ is not ruled (see Remark \ref{MatsusakaDVR}). Thus there exists a complex complete intersection in $\P^N$  of multidegree $(d_1,\dotsc,d_c)$ that is not ruled. Another application of Theorem \ref{Matsusaka} to the family of all complex complete intersections in $\P^N$ of that multi-degree  yields the result (see Remark \ref{exhibit-one}).
\end{proof}

\begin{theorem}
\label{MainResultIrrationalCIs}
Let $N,d_1,\dotsc,d_c$ be positive integers. Let $p$ be a prime number. For $i=1,\dotsc,c$, let $r_i\in \{0,1,\dotsc,p-1\}$ be the remainder of the division of $d_i$ by $p$. Suppose that
\begin{enumerate}
\item $c\le \tfrac 1 2 N-1$, $c\le \frac 1 3 N +1$ and $(c,N-c)\not\in \mathscr E_p$ (see Definition \ref{Exceptions}); 
\item $d_1,\dotsc,d_c \ge p$; and
\item $\sum_{i=1}^c (d_i-r_i) > \frac{p}{p+1}(N+1)$.
\end{enumerate}
Then a complete intersection of $c$ very general hypersurfaces of degrees $d_1,\dotsc,d_c$ in $N$-dimensional complex projective space is not ruled.
\end{theorem}

\begin{proof}
Let $X\subseteq \P^N_\mathbf C$ be a nonruled smooth complete intersection of multi-degree $(d_1-r_1,\dotsc, d_c - r_c)$, which exists by Proposition \ref{IrrationalCIs}.
Let $X'\subseteq \P^N_\mathbf C$ be a generically smooth complete intersection of multidegree $(d_1,\dotsc,d_e)$ that contains $X$ as an irreducible component.
(To obtain such $X'$, one may multiply equations defining $X$ in $\P^N_\mathbf C$ by general homogeneous polynomials of degrees $r_1, \dotsc, r_c$.)
Let $Y\subseteq \P^N_\mathbf C$ be a smooth complete intersection of multi-degree $(d_1,\dotsc,d_c)$.

Let $x_0,\dotsc, x_N$ denote the coordinates on $\P^N$.
Let 
\begin{equation*}
F_1,\dotsc, F_c, G_1,\dotsc, G_c \in \mathbf C[x_0,\dotsc,x_N]
\end{equation*}
be homogeneous polynomials such that $\deg F_i = \deg G_i= d_i$ for all $i$; such that the equations $F_1=\dotsb=F_c=0$ define $X'$ in $\P^N_\mathbf C$; and such that the equations $G_1=\dotsb=G_c=0$ define $Y$ in $\mathbf P^N_\mathbf C$.
Let $t$ denote the coordinate on $\A^1$.
Let $Z\subseteq (\P^N\times \A^1)_\mathbf C$ be subscheme defined by the equations
\begin{equation*}
(1-t) F_i + t G_i =0
\end{equation*}
with $i=1,\dotsc, c$.
Let $\pi : Z\to \A^1_\mathbf C$ be the second projection.

Let $U\subseteq \A^1_\mathbf C$ denote the largest open subset over which $\pi$ is smooth of relative dimension $N-c$.
Note that $1\in U$, so that $U$ is nonempty.
By Remark \ref{exhibit-one}, it suffices to show that the geometric generic fiber of $Z_U\to U$ is not ruled.

Let $T$ denote the spectrum local ring of the point $0\in \A^1_\mathbf C$, which is a DVR.
Note that $T$ contains two points, which correspond to 0 and the generic point of $\A^1_\mathbf C$.
The special fiber of $\pi : Z_T\to T$ is $X'$, which contains a nonruled component.
Furthermore, the total space $Z_T$ is normal and irreducible by Lemma \ref{family-CIs} below.
By a variant of Matsusaka's result \cite[Theorem IV.1.6.2]{Kollar1996}, these observations imply that the geometric generic fiber of $Z_T\to T$ is not ruled.
\end{proof}

\begin{lemma}
\label{family-CIs}
Let $S$ be a scheme.
Let $c$ and $N$ be integers such that $1\le c\le N-1$.
Let $Z\subseteq \P^N\times S$ be a closed subscheme defined by the vanishing of $c$ homogeneous polynomials in the coordinates on $\P^N$ with coefficients in $\Gamma(S,\O_S)$.
Let $\pi : Z\to S$ be the second projection, which is proper.
Assume:
\begin{itemize}
\item $S$ is locally Noetherian, regular and connected;
\item for every $s\in S$, the fiber $Z_s := Z\times_S \Spec \kappa(s)$ is generically smooth of dimension $N-c$ over $\kappa(s)$; and
\item there exists $s\in S$ such that $Z_s$ is smooth over $\kappa(s)$.
\end{itemize}
Then $Z$ is normal and irreducible, and $\pi:Z\to S$ is flat.
\end{lemma}

\begin{proof}
The assumptions imply that $Z$ is a complete intersection in the regular scheme $\P^N\times S$.
Thus $Z$ is Cohen-Macaulay.
The projection $\pi : Z\to S$ is flat by miracle flatness.
The scheme $Z$ is regular away from the singular locus of $\pi$, which has codimension 2.
Thus $Z$ is normal.
For each $s\in S$, the fiber $Z_s$ is a positive-dimensional complete intersection in $\P^N_{\kappa(s)}$, hence connected.
Thus $Z$ is connected (and normal), so irreducible.
\end{proof}


\end{document}